\documentclass{article}

\usepackage[T1]{fontenc}
\usepackage{geometry}
\usepackage{microtype}
\usepackage{mathtools}
\usepackage{amssymb}
\usepackage{amsthm}
\usepackage[shortlabels]{enumitem}
\usepackage{float}
\usepackage{caption}
\usepackage{subcaption}
\usepackage{tikz}
\usepackage{xcolor}
\usepackage{hyperref}
	\hypersetup{
		colorlinks=true,
		linkcolor=red,
		urlcolor=blue,
		citecolor=magenta
	}
\newcommand{\Nat}{\mathbb{N}}
\newcommand{\defining}[1]{\emph{#1}}
\newcommand{\defn}{\coloneqq}
\newcommand{\NP}{\text{NP}}
\newcommand{\List}[2]{#1_{#2}}
\newcommand{\ListAsgn}[1]{\mathcal{#1}}
\newcommand{\Set}[1]{\{ #1 \}}
\newcommand{\card}[1]{\lvert #1 \rvert}
\newcommand{\floor}[1]{\lfloor #1 \rfloor}
\DeclareMathOperator{\chr}{\chi}
\DeclareMathOperator{\ch}{\chi_{\ell}}
\DeclareMathOperator{\cho}{\chi_{\ell}}
\DeclareMathOperator{\Color}{\mathsf{color}}	
\DeclareMathOperator{\degree}{\mathsf{degree}}
\DeclareMathOperator{\outdegree}{\mathsf{outdegree}}
\DeclareMathOperator{\jump}{\mathsf{jump}}

\theoremstyle{plain}
\newtheorem{theorem}{Theorem}
\newtheorem{corollary}[theorem]{Corollary}
\newtheorem{lemma}[theorem]{Lemma}
\newtheorem{question}[theorem]{Question}

\theoremstyle{definition}
\newtheorem{definition}[theorem]{Definition}

\definecolor{myred}{RGB}{230,97,1}	
\definecolor{mygreen}{RGB}{253,184,99}	
\definecolor{myblue}{RGB}{94,60,153}	
\definecolor{myyellow}{RGB}{178,171,210}	

\title{$5$-list coloring toroidal $6$-regular triangulations\\ in linear time\thanks{\copyright 2026. This manuscript version is made available under the CC BY-NC-ND 4.0 license \url{https://creativecommons.org/licenses/by-nc-nd/4.0/}. The published journal article is available at \emph{Discrete Applied Math.}\ \textbf{386} (2026), 75--103, \href{https://doi.org/10.1016/j.dam.2026.01.032}{\texttt{doi:10.1016/j.dam.2026.01.032}}. An extended abstract of this work was published in: A.\ Bagchi, R.\ Muthu (eds.) Algorithms and Discrete Applied Mathematics, LNCS 13947, pp.\ 134--146, Springer, Cham (2023), \href{https://doi.org/10.1007/978-3-031-25211-2_10}{\texttt{doi:10.1007/978-3-031-25211-2\_10}}.}}

\author{Niranjan Balachandran\\
Department of Mathematics,\\
Indian Institute of Technology Bombay,\\
Powai, Mumbai 400076,\\
Maharashtra, India.\\
\texttt{niranj@math.iitb.ac.in}\\
\and
Brahadeesh Sankarnarayanan\thanks{This work was done while the author was at the Indian Institute of Technology Bombay, and was supported by the National Board for Higher Mathematics (NBHM), Department of Atomic Energy (DAE), Govt.\ of India.}\\
Department of Mathematics,\\
Indian Institute of Technology Jodhpur,\\
Karwar 342030, Jodhpur District,\\
Rajasthan, India.\\
\texttt{brahadeesh@iitj.ac.in}\\
}

\date{29 January, 2026}

\begin{document}
\maketitle

\begin{abstract}
We give an explicit procedure for $5$-list coloring a large class of toroidal $6$-regular triangulations in linear time.
We also show that these graphs are not $3$-choosable.\\

\noindent\textbf{MSC (2020):} 05C15 (Primary) 05C85, 05C10, 05C75 (Secondary)\\
\textbf{Keywords:} list coloring, toroidal graph, triangulation, regular graph, linear time algorithm

\end{abstract}

\newpage
\tableofcontents

\newpage
\section{Introduction}\label{S:Introduction}

We will be concerned with the following coloring variant known as \defining{list coloring}, defined independently by Vizing~\cite{Vizing1976} and by Erd{\H o}s, Rubin, and Taylor~\cite{ErdosRubinEtAl1980}.
A \defining{list assignment} \(\ListAsgn{L}\) on a graph \(G = (V, E)\) is a collection of sets of the form \(\ListAsgn{L} = \Set{ \List{L}{v} \subset \Nat : v \in V(G) }\), where one thinks of each \(\List{L}{v}\) as a \defining{list} of colors available for coloring the vertex \(v \in V(G)\).
A graph \(G\) is \defining{\(\ListAsgn{L}\)-choosable} if there exists a function \(\Color \colon V(G) \to \Nat\) such that \(\Color(v) \in \List{L}{v}\) for every \(v \in V(G)\) and \(\Color(v) \neq \Color(w)\) whenever \(vw \in E(G)\).
A graph \(G\) is called \defining{\(k\)-choosable} if it is \(\ListAsgn{L}\)-choosable for every \(k\)-list assignment \(\ListAsgn{L}\) (i.e., an assignment of lists of size at least \(k\), also called \(k\)-lists).
The least integer \(k\) for which \(G\) is \(k\)-choosable is the \defining{choice number}, or \defining{list chromatic number}, of \(G\) and is denoted \(\ch(G)\).
If \(\ch(G) = k\), we also say that \(G\) is \defining{\(k\)-list chromatic}.
Notice that the usual notion of graph coloring is equivalent to \(\ListAsgn{L}\)-coloring when all the lists assigned by \(\ListAsgn{L}\) are identical.
This also shows that \(\chr(G) \leq \ch(G)\) for all graphs \(G\), and in general the inequality can be strict~\cite{ErdosRubinEtAl1980,Vizing1976}.

\subsection{Motivation}
\subsubsection{\texorpdfstring{\(k\)}{k}-choosability is computationally hard}
It is well-known that computing the chromatic number is an \NP-hard problem~\cite{KleinbergTardos2014}.
The restricted problem of finding a \(4\)-coloring of a \(3\)-chromatic graph is also \NP-hard~\cite{KhannaLinialEtAl2000}.
Even the problem of \(3\)-colorability of \(4\)-regular planar graphs
is known to be \NP-complete~\cite{Dailey1980}.

Naturally, list coloring is also a computationally hard problem, but much more: for instance, it is well-known~\cite{Gutner1996} that the problem of deciding whether a given planar graph is \(4\)-choosable is \NP-hard---even if the \(4\)-lists are all chosen from \(\{1,2,3,4,5\}\)~\cite{DabrowskiDrossEtAl2019}---and so is deciding whether a given planar triangle-free graph is \(3\)-choosable~\cite{Gutner1996}.
But, contrast the latter with the fact that every planar triangle-free graph is \(3\)-colorable by Gr{\"o}tzsch's theorem~\cite{Grotzsch1958}, and that a \(3\)-coloring can be found in linear time~\cite{DvorakKawarabayashiEtAl2011}.
In other words, restrictions on graph parameters---such as the girth, as in Gr{\"o}tzsch's theorem---that allow for efficient coloring algorithms need to be strengthened further in order to get list coloring algorithms of a similar flavor.

Note that even proving nontrivial bounds for the choice number is far tougher than the corresponding problem for the chromatic number.
Some of the notable instances of such bounds being determined include Brooks's theorem for choosability~\cite{Vizing1976,ErdosRubinEtAl1980}, Thomassen's remarkable proof that every planar graph is \(5\)-choosable~\cite{Thomassen1994b}, and Galvin's solution to the famous Dinitz problem~\cite{Galvin1995}.
Other interesting examples include the fact that planar bipartite graphs are \(3\)-choosable~\cite{AlonTarsi1992} and that any \(4\)-regular graph decomposable into a Hamiltonian circuit and vertex-disjoint triangles is \(3\)-choosable~\cite{FleischnerStiebitz1992}. However, there is a fundamental difference between the former and latter examples, as we elaborate below.

\subsubsection{\texorpdfstring{\(\ListAsgn{L}\)}{L}-coloring is algorithmically hard}
Consider the problem: given a list assignment \(\ListAsgn{L}\) on a graph \(G\), can one efficiently determine whether or not \(G\) is \(\ListAsgn{L}\)-choosable, and in the case when \(G\) is \(\ListAsgn{L}\)-choosable can one also efficiently specify a proper coloring from these lists?
The theorems of Brooks, Thomassen and Galvin mentioned earlier are some of the few instances where such algorithms are known for a large class of graphs.
In the other examples that we mentioned, the proof uses the combinatorial nullstellensatz~\cite{Alon1999}, and in particular a powerful application is found by Alon and Tarsi~\cite{AlonTarsi1992}.
Hence, it does not allow one to extract an efficient algorithmic solution to the problem of \(\ListAsgn{L}\)-coloring when the list assignment \(\ListAsgn{L}\) is specified, except in certain special cases.
That there is no known efficient algorithm that produces a \(3\)-list coloring from a given list assignment in these examples illustrates the difficulty of the problem of efficiently finding a proper \(\ListAsgn{L}\)-coloring even for graphs of small maximum degree.
Even just for planar bipartite graphs, an algorithmic determination of a list coloring largely remains open~\cite{DabrowskiDrossEtAl2019}.

Hence, efficient \(\ListAsgn{L}\)-coloring algorithms for large classes of graphs are interesting.
We also place our work within the context of recent results on efficient list coloring algorithms for similar classes of graphs in Section~\ref{SS:related} below.

\subsection{Our work}

A natural choice of a large class of graphs for which list colorings have been studied is the class of graphs that are embeddable on a fixed surface (see~\cite{Mohar2015} and the references therein).
Here, by a \defining{surface} we mean a compact connected \(2\)-manifold, and a graph is \defining{embeddable} in a surface if, informally speaking, it can be drawn on the surface without any crossing edges (for further details, see~\cite{BalachandranSankarnarayanan2023}).
In this paper, we will be concerned only with \emph{toroidal graphs}, that is, graphs that are embeddable on the torus \(S_{1}\), which is the orientable surface of genus \(1\).

Let \(G = (V, E)\) be a toroidal graph, and let \(F\) be the set of its faces in an embedding into \(S_{1}\).
The graphs satisfying \(\degree(v) = d\) for all \(v \in V\) and \(\degree(f) = m\) for all \(f \in F\), for some \(d, m \geq 1\), have been of interest~\cite{Altshuler1972,Altshuler1973} especially in the study of vertex-transitive graphs~\cite{Babai1991,Thomassen1991}.
A simple calculation using Euler's formula shows that the only possible values of \((d, m)\) are \((3, 6)\), \((4, 4)\) and \((6, 3)\).
Our focus will be on the graphs of the last kind, namely the \(6\)-regular triangulations on the torus.
Since triangulations have the maximum possible number of edges in any graph with a fixed number of vertices and embeddable on a given surface, one might additionally expect this class of graphs to present a greater obstacle to an efficient solution to the list coloring problem as compared to the others.

The main result of this paper, Theorem~\ref{T:Main}, is a linear time algorithm for \(5\)-list coloring a large class of these toroidal \(6\)-regular triangulations.
Furthermore, the choice number of any graph in this class is at least \(4\), so our result is nearly tight for this class.
In fact, in Corollary~\ref{C:main} we find an infinite family of \(5\)-chromatic-choosable graphs for which a list coloring can be specified in linear time.

Let \(T(r, s, t)\) be a triangulation obtained from an \(r \times s\) toroidal grid, \(r, s \geq 1\) (see Definition~\ref{D:triangulation} for a precise statement):
\begin{theorem}\label{T:Main}
	Let \(G\) be a simple \(6\)-regular toroidal triangulation.
	Then, \(G\) is \(5\)-choosable under any of the following conditions:
	\begin{enumerate}[label=(\arabic*),nosep,noitemsep]
		\item\label{l1} \(G\) is isomorphic to \(T(r, s, t)\) for \(r \geq 4\);
		\item\label{l2} \(G\) is isomorphic to \(T(1, s, 2)\) for \(s \geq 9\), \(s \neq 11\);
		\item\label{l3} \(G\) is isomorphic to \(T(2, s, t)\) for \(s\) and \(t\) both even;
		\item\label{l4} \(G\) is \(3\)-chromatic.
	\end{enumerate}
	Moreover, the \(5\)-list colorings can be given in linear time.
	Furthermore, none of these graphs are \(3\)-choosable.
	Hence, \(\ch(G) \in \Set{4, 5}\) if any of the cases~\ref{l1} to~\ref{l4} hold for \(G\).
\end{theorem}
We note that we have proved case~\ref{l4} in a previous paper~\cite{BalachandranSankarnarayanan2021}, albeit without any efficient algorithm for \(5\)-choosability.

We are currently unable to comment on the choosability of the excluded graphs, but we note that they consist only of eight nonisomorphic \(5\)-chromatic graphs, as well as a subcollection of triangulations of the specific form \(T(1, s, t)\) that are \(4\)-chromatic.
For any tuple \((r, s, t)\), there is a simple formula describing each tuple \((r', s', t')\) such that \(T(r, s, t)\) is isomorphic to \(T(r', s', t')\) (see~\cite{Altshuler1972,Sankarnarayanan2022}), and there are at most \(6\) such tuples for any \((r, s, t)\).
It is also not difficult to see that the loopless multigraphs \(T(r, s, t)\) are all \(5\)-choosable.
So, in this sense, Theorem~\ref{T:Main} covers the \(5\)-choosability of ``most'' \(6\)-regular toroidal triangulations.
Furthermore, among those graphs covered in Theorem~\ref{T:Main}, the \(5\)-chromatic ones are precisely those isomorphic to \(T(1, s, 2)\) for \(s \not\equiv 0 \pmod{4}\).
Thus, we have:
\begin{corollary}\label{C:main}
	If \(G\) is isomorphic to \(T(1, s, 2)\) for \(s \not\equiv 0 \pmod{4}\), \(s \geq 9\), \(s \neq 11\), then \(G\) is \(5\)-chromatic-choosable, i.e. \(\chr(G) = \ch(G) = 5\).
	Moreover, a \(5\)-list coloring can be found in linear time.
\end{corollary}
To the best of our knowledge, the method of proof that we employ is novel, in that we develop a framework that allows us to systematically compare the lists on vertices that are not too far apart, and that allows us to compute the list coloring in an efficient manner. By using the differential information between lists on nearby vertices, we reduce the \emph{list configurations} that need to be considered. This kind of ``list calculus'' differs from other list coloring algorithms in the literature, which instead reduce the possible \emph{graph configurations} by exploiting general structure results on the family of graphs under consideration (minimum girth, edge-width, etc.), while the specific lists on the graphs remain nebulous. Our method of proof could prove fruitful in other areas where a structure theorem---such as Theorem~\ref{T:Altshuler} in our case---allows one to shift attention towards the configuration of the lists themselves.
We also emphasize that our linear-time algorithm for \(5\)-list coloring these graphs is nearly best possible, since any fixed vertex needs to be ``scanned'' very few times.

\subsection{Related work}\label{SS:related}
\subsubsection{Colorability vs.\ choosability}
Note that it follows from Brooks's theorem for choosability that any \(6\)-regular toroidal triangulation not isomorphic to \(K_{7}\) is \(6\)-choosable.
Albertson and Hutchinson~\cite{AlbertsonHutchinson1980} showed that there is a unique simple graph in this family that is \(6\)-chromatic, which has \(11\) vertices, and Thomassen~\cite{Thomassen1994a} later classified all the \(5\)-colorable toroidal graphs.
But a precise characterization of all the \(5\)-chromatic \(6\)-regular toroidal triangulations was completed only recently~\cite{CollinsHutchinson1999,YehZhu2003,Sankarnarayanan2022}.
Our results are the first in this line to attempt to characterize the list colorability of the \(6\)-regular triangulations on the torus.

\subsubsection{Choosability of grids}
The problem of determining the choice number of \(4\)-regular toroidal \(m \times n\) grids, for \(m, n \geq 3\), has been raised by Cai, Wang and Zhu~\cite{CaiWangEtAl2010}.
These graphs are a special case of those satisfying \((d, m) = (4, 4)\).
It is easy to show by induction that these grids are all \(3\)-colorable, and the above authors conjecture that they are also \(3\)-choosable.
Recent work by Li, Shao, Petrov and Gordeev~\cite{LiShaoEtAl2023} has nearly determined the choice number of these grids as follows: if \(mn\) is even, then the choice number is \(3\), else it is either \(3\) or \(4\).
Contrasting this with Theorem~\ref{T:Main}, we note that both nearly determine the choice number in the sense that the true value of the choice number is either equal to, or one less than, the computed value for each member of the family.
However, their result does not a priori give an efficient algorithm for \(\ListAsgn{L}\)-coloring the toroidal grids since their proof uses the combinatorial nullstellensatz, whereas our result actually gives a linear time algorithm for \(\ListAsgn{L}\)-coloring the toroidal triangulations.

\subsubsection{Recent algorithmic advances for list colorings of graphs on surfaces}
Dvo\v{r}\'{a}k and Kawarabayashi \cite{DvorakKawarabayashi2013} have shown that for \(\ListAsgn{L}\)-coloring a graphs embedded on a fixed surface \(\Sigma\), where \(\ListAsgn{L}\) is a \(5\)-list assignment, there exists a \(O(\card{V(G)}^{O(g(\Sigma)+1)})\)-time algorithm, where \(g(\Sigma)\) is the genus of the surface \(\Sigma\).
Postle and Thomas~\cite{PostleThomas2018} have proved that for any surface \(\Sigma\) and every \(k \in \Set{3, 4, 5}\) there exists a linear time algorithm for determining whether or not an input graph \(G\) embedded in \(\Sigma\) and having girth at least \(8 - k\) is \(k\)-choosable.
In particular, when \(\Sigma = S_{1}\) and \(k = 5\), this implies that there is a linear time algorithm for determining whether or not any of the \(6\)-regular triangulations under consideration in this paper are \(5\)-choosable.
This work was later extended by Postle in~\cite{Postle2019}, wherein he showed that for each fixed surface \(\Sigma\) there exists a linear time algorithm to find a \(k\)-list coloring of a graph \(G\) with girth at least \(8 - k\) for \(k \in \{3, 4, 5\}\).
Again, when \(\Sigma = S_{1}\) and \(k = 5\), this says that there is a linear time algorithm to find a \(5\)-list coloring of a \(6\)-regular triangulation on the torus.

Our results in this paper are stronger than those mentioned above for the class of \(6\)-regular toroidal triangulations.
Firstly, the high degree of the polynomial time algorithm in~\cite{DvorakKawarabayashi2013} makes it impractical to implement, though the authors suggest that it should likely be possible to reduce the bound enough to make the algorithm practical at least for planar graphs.
Secondly, the linear time algorithm in~\cite{PostleThomas2018} is contingent upon an enumeration of the \defining{\(6\)-list critical} graphs on the torus.
Indeed, the authors show that there are only finitely many \(6\)-list critical graphs on the torus, but a full list of these graphs is not explicitly known, and their bound on the maximum number of vertices any \(6\)-list critical graph on the torus can have is far too large to be amenable to a straightforward enumerative check.\footnote{It is worth contrasting this with the corresponding colorability problem: while Thomassen~\cite{Thomassen1997} has shown that for every fixed surface there are only finitely many \defining{\(6\)-critical} graphs that embed on that surface, explicit lists of these \(6\)-critical graphs are known only for the projective plane~\cite{AlbertsonHutchinson1979}, the torus~\cite{Thomassen1994a} and the Klein bottle~\cite{ChenettePostleEtAl2012,KawarabayashiKralEtAl2009}.}
Also, their linear time algorithm does not specify an \(\ListAsgn{L}\)-coloring in the case when the graph is \(\ListAsgn{L}\)-choosable for a given list assignment \(\ListAsgn{L}\).
Thirdly, the linear time algorithm in~\cite{Postle2019} first requires a brute-force computation of the list colorings for any such list assignment on graphs of ``small'' order.
However, the bound on the sizes of these small graphs is far too large to be computationally feasible, which makes the algorithm itself of mostly theoretical interest, as noted in a recent work by Dvo\v{r}\'{a}k and Postle~\cite{DvorakPostle2022}.

This is in contrast with the results in this paper, wherein the \(5\)-choosable graphs identified in Theorem~\ref{T:Main} can also be given \(5\)-list colorings in linear time, unlike as in~\cite{PostleThomas2018}.
Furthermore, the non-\(3\)-choosability of the  \(3\)-chromatic graphs \(T(r, s, t)\) is not covered by the results in~\cite{PostleThomas2018} since these graphs have girth equal to \(3\), whereas their algorithm for \(3\)-list coloring is applicable only for graphs having girth at least \(5\).
Lastly, our proof of Theorem~\ref{T:Main} supplies an implementable algorithm for \(5\)-list coloring all the toroidal graphs under consideration without the need for running a brute-force check on any of them, in contrast with~\cite{Postle2019}.

\subsection*{Structure of this paper}
In Section~\ref{S:Preliminaries}, we setup the necessary preliminary material; in particular, Lemmas~\ref{L:cycle-choosability} and \ref{L:N-algorithm} will be used at several points in the rest of the paper.
We provide an overview of the proof of Theorem~\ref{T:Main} in Section~\ref{S:Overview}.
In Section~\ref{S:Preparation}, we prove a succession of technical lemmas to prepare the proof of case~\ref{l1} in Theorem~\ref{T:Main}. The proof of this case is completed in Section~\ref{S:Main}. In Section~\ref{S:Sub}, we prove cases~\ref{l2} and \ref{l3} in Theorem~\ref{T:Main}. In Section~\ref{S:baaki}, we prove case~\ref{l4} in Theorem~\ref{T:Main}. In Section~\ref{S:Conclusion}, we analyze the remaining cases not covered by Theorem~\ref{T:Main} and conclude with some conjectures concerning their choosability.

\section{Preliminaries}\label{S:Preliminaries}
Altshuler~\cite{Altshuler1972,Altshuler1973} showed that every \(6\)-regular toroidal triangulation \(G\) can be described as a regular triangulation obtained from an \(r \times s\) toroidal grid in which the edges between the first and last column are connected by a shift of \(t\) vertices. Concretely:
\begin{definition}\label{D:triangulation}
For integers \(r \geq 1\), \(s \geq 1\) and \(0 \leq t \leq s - 1\),
take \(V = \{ (i, j) : 1 \leq i \leq r, 1 \leq j \leq s \}\)
to be the vertex set of the graph \(T(r, s, t)\) equipped with the following edges:
\begin{itemize}
	\item For each \(1 < i < r\), \((i, j)\) is adjacent to
	\((i, j \pm 1)\), \((i \pm 1, j)\) and
	\((i \pm 1, j \mp 1)\).
	
	\item If \(r > 1\), \((1, j)\) is adjacent to \((1, j \pm 1)\), \((2, j)\),
	\((2, j - 1)\), \((r, j + t + 1)\) and \((r, j + t)\). 
	
	\item If \(r > 1\), \((r, j)\) is adjacent to \((r, j \pm 1)\),
	\((r - 1, j + 1)\), \((r - 1, j)\), \((1, j - t)\) and \((1, j - t - 1)\).
	
	\item If \(r = 1\), \((1, j)\) is adjacent to \((1, j \pm 1)\), \((1, j \pm t)\)
	and \((1, j \pm (t + 1))\).
\end{itemize}
\end{definition}
Here, addition in the first coordinate is taken modulo \(r\) and
in the second coordinate is taken modulo \(s\).
Figure~\ref{F:T(562)} depicts the graph \(G = T(5, 6, 2)\);
note that the edges between the top and bottom rows are not shown in this and all subsequent figures.

We shall use the notation \(C_{i}\), for \(1 \leq i \leq r\), to denote the induced subgraph of
\(T(r, s, t)\) on the \(i\)th column of \(T(r, s, t)\), that is, on the set of vertices
\(\Set{ (i, j) : 1 \leq j \leq s }\). Note that
each \(C_{i}\) is a cycle of length \(s\) when \(r > 1\).

\begin{figure}
	\centering
	\begin{tikzpicture}[font=\scriptsize]
		\draw (1,1) grid (6,6);
		\foreach \y in {2,...,6}
		\foreach \x in {1,...,5}{
			\draw (\x,\y) -- (\x+1,\y-1);
		}
		\foreach \x in {1,...,6}
		\foreach \y in {1,...,6}{
			\draw[fill=black] (\x,\y) circle(2pt);
		}
		\foreach \y in {1,...,6}{
			\node[anchor=east] at (1,\y) {(1,\y)};
			\pgfmathparse{Mod(\y-5,6)}
			\node[anchor=west] at (6,\pgfmathresult+1) {(1,\y)};
		}
		\node[anchor=south west] at (5,2) {\!\!(5,2)};
		\node[anchor=south west] at (5,3) {\!\!(5,3)};
		\node[anchor=south west] at (5,4) {\!\!(5,4)};
		\node[anchor=south west] at (5,5) {\!\!(5,5)};
		\foreach \x in {2,...,5}{
			\node[anchor=north] at (\x,1) {(\x,1)};
			\node[anchor=south west] at (\x,6) {\!\!(\x,6)};
		}
	\end{tikzpicture}
	\caption{\(G = T(5,6,2)\); the edges between the top and bottom rows are not shown in this
and all subsequent figures.}\label{F:T(562)}
\end{figure}

It is clear that each \(T(r,s,t)\) is a \(6\)-regular triangulation of the torus.
Altshuler's theorem says that these are all the \(6\)-regular triangulations on the torus up to isomorphism (similar constructions also appear in \cite{Negami1983,Thomassen1991}).
\begin{theorem}[Altshuler~\cite{Altshuler1973}, 1973]\label{T:Altshuler}
	Every \(6\)-regular triangulation on the torus is isomorphic to \(T(r, s, t)\) for some integers \(r \geq 1\), \(s \geq 1\), and \(0 \leq t < s\).
\end{theorem}

Altshuler also showed~\cite{Altshuler1972,Altshuler1973} that through every vertex \(v\) of \(T(r, s, t)\) there are three \defining{normal circuits}, which are the simple cycles obtained by traversing through \(v\) along each of the three directions (vertical, horizontal, and diagonal) in the natural fashion.
These normal circuits have lengths \(s\), \(n/\gcd(s, t)\), and \(n/\gcd(s, r + t)\), respectively, where \(n = rs\) is the order of \(T(r, s, t)\).

By picking a different normal circuit to be represented as the vertical cycle, one can see that there exist integers \(t_1, t_2\) such that \(0 \leq t_{1} < n/\gcd(s, t)\) and \(0 \leq t_{2} < n/\gcd(s, r + t)\) and \(T(r, s, t)\) is isomorphic to \(T\bigl(\gcd(s, t), n/\gcd(s, t), t_{1}\bigr)\) as well as to \(T\bigl(\gcd(s, r + t), n/\gcd(s, r + t), t_{2}\bigr)\).
Similarly, by swapping the horizontal and diagonal normal circuits, one can see that \(T(r, s, t)\) is isomorphic to \(T(r, s, t')\) where \(0 \leq t' < s\) such that \(t' \equiv - r - t \pmod{s}\).

The following lemma is useful in simplifying arguments through the use of
symmetry:

\begin{lemma}\label{L:auto}
	Let \(r \geq 1\), \(s \geq 1\) and \(0 \leq t \leq s - 1\). The map
	\((i, j) \mapsto (r - i + 1, s - j + 1)\) on \(V\bigl(T(r, s, t)\bigr)\)
	induces an automorphism of \(T(r, s, t)\).
\end{lemma}

In particular, this automorphism reverses the ordering of the rows (as well as of the columns).

We will need the following theorem on finding matchings in regular
bipartite graphs.

\begin{definition}
	A \defining{matching} in a graph \(G = (V, E)\) is a subset
	\(M\) of \(E\) such that no two edges in \(M\) have a common
	vertex. A matching is said to be \defining{perfect} if
	every vertex \(v \in V\) belongs to some edge in the matching.
\end{definition}

\begin{theorem}[Cole--Ost--Schirra~\cite{ColeOstEtAl2001}, 2001]\label{T:matching}
	Let \(G = (V, E)\) be a regular bipartite graph. Then,
	a perfect matching \(M\) of \(G\) can be found in \(O(\card{E})\) time.
\end{theorem}

We make the following definition which will simplify some of the terminology
in the proofs that follow.

\begin{definition}
	Let \(\ListAsgn{L}\) be a list assignment on a graph \(G = (V, E)\).
	If, in a partial \(\ListAsgn{L}\)-coloring of \(G\),
	a vertex \(v \in V\) is colored with \(c \in \List{L}{v}\),
	then the color \(c\) is no longer available for use on the
	uncolored neighbors of \(v\). So, the color \(c\) is removed from
	the lists of the neighbors of \(v\), and we do this for each
	vertex colored in this partial \(\ListAsgn{L}\)-coloring of \(G\).
	The new lists on \(G\) are also called the \defining{residual}
	lists on \(G\), and we shall say that the list on an uncolored
	vertex \(u\) \defining{reduces by \(k\)} if the residual
	list on \(u\) is a \((\card{\List{L}{u}} - k)\)-list.
\end{definition}

The following well-known lemma, due to Bondy--Bopanna--Siegel~\cite[Remark~2.4]{AlonTarsi1992}, gives an
algorithmic proof of the theorem of Alon and Tarsi~\cite{AlonTarsi1992} in the case when \(G\) is given an orientation containing
no odd directed cycle. For the sake of completeness, we provide a proof of this lemma along
the lines in \cite{West2002}.

\begin{definition}
	Let \(G\) be a digraph, and \(v \to w\) be an edge of \(G\).
	We also call \(w\) the \defining{successor} of \(v\), and \(v\) the
	\defining{predecessor} of \(w\).
	A \defining{kernel} of \(G\) is an independent set \(S\) such that
	every \(v \not\in S\) has a successor in \(S\).
\end{definition}

\begin{lemma}[{Bondy--Bopanna--Siegel~\cite[Remark~2.4]{AlonTarsi1992}, 1992}]\label{L:BBS}
	Let \(G = (V, E)\) be a simple graph that is given an orientation containing no odd directed cycle.
	Suppose \(\ListAsgn{L}\) is a list assignment on \(G\) such that
	\(\List{L}{v}\) is an \((\outdegree(v) + 1)\)-list for all \(v \in V\).
	Then, \(G\) is \(\ListAsgn{L}\)-choosable.
\end{lemma}
\begin{proof}
	If \(\card{V} = 1\), then the statement is trivial, so suppose that \(\card{V} = n > 1\),
	and that the lemma is true for all graphs with fewer than \(n\) vertices.
	Let \(c\) be a color that occurs in some list assigned by \(\ListAsgn{L}\).
	Consider the induced subgraph \(H\) on the set \(U \defn \Set{ v \in V : c \in \List{L}{v} }\).
	Clearly, the induced orientation on \(H\) also does not have any odd directed cycle.
	Now, Richardson's theorem~\cite{Richardson1953b} says that any digraph without odd directed cycles has a kernel,
	so let \(S\) be a kernel in \(H\). Assign the color \(c\) to every vertex in \(S\),
	and now consider \(G' \defn G - S\). Notice that the residual lists
	have reduced in size by \(1\) for every list on \(U - S\),
	but every vertex in \(U - S\) also has a successor in \(S\). Thus, \(G'\)
	satisfies the induction hypothesis, and we are done.
\end{proof}

One can find in the literature~\cite{Neumann-Lara1971,SzwarcfiterChaty1994} proofs of Richardson's theorem that output
a kernel in polynomial time. However,
the graphs that we consider (cf.~Lemmas~\ref{L:cycle-choosability} and~\ref{L:fix})
have enough structure that they permit straightforward
linear time algorithms for finding a kernel.

For integers \(r, s \geq 3\), define the \defining{cylindrical triangulation}
\(C(r, s)\) to be the graph obtained from \(T(r + 1, s, 0)\) by deleting
the column \(C_{r + 1}\). More formally, let
\(V\bigl(C(r, s)\bigr) \defn \Set{ (i, j) : 1 \leq i \leq r, 1 \leq j \leq s }\)
and let \(E\bigl(C(r, s)\bigr)\) contain the following edges:
\begin{itemize}
	\item For \(1 < i < r\), let \((i, j)\) be adjacent to
	\((i, j \pm 1)\), \((i \pm 1, j)\) and
	\((i \pm 1, j \mp 1)\).
	
	\item Let \((1, j)\) be adjacent to \((1, j \pm 1)\), \((2, j)\) and
	\((2, j - 1)\).
	
	\item Let \((r, j)\) be adjacent to \((r, j \pm 1)\),
	\((r - 1, j + 1)\) and \((r - 1, j)\).
\end{itemize}
Again, addition in the second coordinate is taken modulo \(s\).
Note that every \defining{interior vertex} of \(C(r, s)\), that is,
any vertex \((i, j)\) with \(1 < j < r\), has degree \(6\) and
every \defining{exterior vertex} of \(C(r, s)\), that is, any vertex \((i, j)\)
with \(j = 1\) or \(j = r\), has degree \(4\).

By an abuse of notation, we shall use \(C_{i}\) to denote the induced subgraph
on the \(i\)th column of \(C(r, s)\), too. Note that if we delete
any column of the graph \(T(r + 1, s, t)\) for any \(0 \leq t \leq s - 1\),
we still get a graph isomorphic to \(C(r, s)\).

We will often need to color paths and cycles in \(T(r, s, t)\) and \(C(r, s)\),
so we compile well-known results (see~\cite{ErdosRubinEtAl1980}, for instance)
on the colorability of these graphs in the following lemma.
Moreover, from Lemma~\ref{L:BBS} and the above comments, we can give linear time algorithms
for \(\ListAsgn{L}\)-coloring these graphs.

\begin{lemma}\label{L:cycle-choosability}
	\hfill
	\begin{enumerate}
		\item An even cycle is \(2\)-list chromatic.
		
		\item An odd cycle is not \(2\)-colorable, and hence not \(2\)-choosable.
		However, if \(\ListAsgn{L}\) is a list assignment of \(2\)-lists
		on an odd cycle such that not all the lists are identical, then the
		cycle is \(\ListAsgn{L}\)-choosable.
		
		\item If \(\ListAsgn{L}\) is a list assignment
		on an odd cycle having one \(1\)-list, one \(3\)-list,
		and all the rest as \(2\)-lists,
		then the cycle is \(\ListAsgn{L}\)-choosable.
		
		\item If \(\ListAsgn{L}\) is a list assignment on a path graph
		having one \(1\)-list, and all the rest as \(2\)-lists,
		then the path is \(\ListAsgn{L}\)-choosable.
	\end{enumerate}
	Moreover, the \(\ListAsgn{L}\)-colorings can all be found in linear time.
\end{lemma}

The following lemma, due to S.~Sinha (during an undergraduate research internship
with the first author), is in a similar spirit to Thomassen's list coloring
of a near-triangulation of the plane~\cite{Thomassen1994b},
and it will be repeatedly invoked in the proof of case~\ref{l1} in Theorem~\ref{T:Main}.

\begin{lemma}[{Sinha [personal communication], 2014}]\label{L:N-algorithm}
	For \(r \geq 3\), \(s \geq 3\), let \(G = C(r, s)\) be a cylindrical triangulation.
	Suppose that \(\ListAsgn{L}\) is a list assignment on \(G\) such that:
	\begin{enumerate}
		\item there exists \(1 \leq j \leq s\) such that the exterior vertices \((1, j)\)
		and \((1, j - 1)\) have lists of size equal to \(4\);
		\item every other exterior vertex has a list of size equal to \(3\);
		\item every interior vertex has a list of size equal to \(5\).
	\end{enumerate}
	Then, \(G\) is \(\ListAsgn{L}\)-choosable. Moreover, an \(\ListAsgn{L}\)-coloring
	can be found in linear time.
\end{lemma}
\begin{proof}
	By Lemma~\ref{L:cycle-choosability}, there
	is a proper coloring of \(C_{r}\) since it is
	assigned \(3\)-lists under \(\ListAsgn{L}\).
	Since every vertex of \(C_{r - 1}\) is adjacent to exactly
	two vertices of \(C_{r}\), a proper coloring
	of \(C_{r}\) reduces the \(5\)-lists on \(C_{r - 1}\) to \(3\)-lists.
	
	Thus, by inductively coloring the columns of \(C(r, s)\) from the right, we may assume
	without loss of generality that \(r = 3\). We also assume without loss
	of generality that the lists of size equal to \(4\) are on the
	vertices \((1, s - 1)\) and \((1, s)\) in the column \(C_{1}\).
	Now, color \((2, s)\) with \(c \in \List{L}{(2, s)} \setminus \List{L}{(1, s)}\),
	which exists since \(C_{2}\) has \(5\)-lists.
	This reduces the sizes of the lists on each of the neighbors of \((2, s)\) by \(1\), except
	for \(\List{L}{(1, s)}\), which still has size equal to \(4\).
	Now, \(C_{3}\) has \(3\)-lists on every vertex, except for \((3, s)\) and \((3, s - 1)\),
	which have \(2\)-lists. So, properly color \(C_{3}\) using Lemma~\ref{L:cycle-choosability}.
	Then, color the remaining vertices in a zigzag fashion from the bottom row, coloring \((1, s)\)
	last, in the following order:
	\((1, 1)\), \((2, 1)\), \((1, 2)\), \((2, 2)\), \ldots, \((1, j)\), \((2, j)\), \ldots,
	\((1, s - 2)\), \((2, s - 2)\), \((2, s - 1)\), \((1, s - 1)\), \((1, s)\).
	
	A proper coloring can always be found by coloring the vertices in the above sequence for
	the following reason.
	After coloring \(C_{3}\), the list sizes
	on the remaining vertices are as follows:
	\((1, s)\) and \((1, s - 1)\) have \(4\)-lists,
	\((1, 1)\), \((2, 1)\) and \((2, s - 1)\) have \(2\)-lists, and
	all other vertices have \(3\)-lists.
	So, color the vertex \((1, 1)\) using a color from its list, and the list sizes
	then are as follows:
	\((1, s - 1)\) has a \(4\)-list,
	\((1, 2)\) and \((2, s - 1)\) have \(2\)-lists,
	\((2, 1)\) has a \(1\)-list,
	and all other vertices have \(3\)-lists.
	Next, color \((2, 1)\) using a color from its list, and observe
	that the next vertex that is to be colored in the sequence always has at least
	one color left in its list. The last three vertices left to be colored are in
	a \(3\)-cycle, with lists of sizes at least \(1\), \(2\) and \(3\). This cycle is properly colorable
	by Lemma~\ref{L:cycle-choosability}, so this completes the proof.
	
	It is clear from the proof that this algorithm produces an \(\ListAsgn{L}\)-coloring
	in linear time. Figure~\ref{F:C-algorithm} illustrates the sizes of the lists at each step of the above
	coloring sequence for the graph \(G = C(3,5)\).
\end{proof}

\begin{figure}
\centering
	\begin{subfigure}{0.2\textwidth}
	\centering
		\begin{tikzpicture}[font=\scriptsize]
			\draw (1,1) grid (3,5);
			\foreach \y in {2,...,5}
			\foreach \x in {1,2}
			\draw (\x,\y) -- (\x+1,\y-1);
			\foreach \x in {1,2,3}
			\foreach \y in {1,...,5}
			\draw[fill=black] (\x,\y) circle(2.5pt);
			\foreach \y in {1,...,3}{
				\node[anchor=east] at (1,\y) {3};
			}
			\foreach \y in {1,...,5}{
				\node[anchor=west] at (3,\y) {3};
			}
			\foreach \y in {4,5}{
				\node[anchor=east] at (1,\y) {4};
			}
			\node[anchor=north] at (2,1) {5};
			\node[anchor=south] at (2,5) {5};
			\foreach \y in {2,...,4}
			\node[anchor=south west] at (2,\y) {5};
		\end{tikzpicture}
		\caption{}
	\end{subfigure}
	\begin{subfigure}{0.2\textwidth}
	\centering
		\begin{tikzpicture}[font=\scriptsize]
			\draw (1,1) grid (3,4);
			\foreach \y in {2,...,4}
			\foreach \x in {1,2}
			\draw (\x,\y) -- (\x+1,\y-1);
			\draw
			(1,4) -- (1,5) -- (2,4)
			(3,4) -- (3,5);
			\draw[dotted]
			(1,5) -- (2,5) -- (2,4)
			(3,4) -- (2,5) -- (3,5);
			\foreach \x in {1,2,3}
			\foreach \y in {1,...,5}
			\draw[fill=black] (\x,\y) circle(2.5pt);
			\foreach \y in {2,3}{
				\node[anchor=east] at (1,\y) {3};
			}
			\foreach \y in {1,...,3}{
				\node[anchor=west] at (3,\y) {3};
			}
			\foreach \y in {4,5}{
				\node[anchor=east] at (1,\y) {4};
				\node[anchor=west] at (3,\y) {2};
			}
			\node[anchor=north] at (2,1) {4};
			\node[anchor=south] at (2,5) {\(c\)};
			\foreach \y in {2,3}
			\node[anchor=south west] at (2,\y) {5};
			\node[anchor=south west] at (2,4) {4};
			\node[anchor=east] at (1,1) {2};
		\end{tikzpicture}
		\caption{}
	\end{subfigure}
	\begin{subfigure}{0.2\textwidth}
	\centering
		\begin{tikzpicture}[font=\scriptsize]
			\draw (1,1) grid (2,4);
			\foreach \y in {2,...,4}
			\draw (1,\y) -- (2,\y-1);
			\draw
			(1,4) -- (1,5) -- (2,4);
			\draw[dotted]
			(1,5) -- (2,5) -- (2,4);
			\foreach \x in {1,2,3}
			\foreach \y in {1,...,5}
			\draw[fill=black] (\x,\y) circle(2.5pt);
			\foreach \y in {2,3}{
				\node[anchor=east] at (1,\y) {3};
				\node[anchor=south west] at (2,\y) {3};
			}
			\foreach \y in {4,5}{
				\node[anchor=east] at (1,\y) {4};
			}
			\node[anchor=north] at (2,1) {2};
			\node[anchor=south west] at (2,4) {2};
			\node[anchor=east] at (1,1) {2};
			\foreach \y in {2,...,5}{
				\draw[dotted]
				(2,\y) -- (3,\y)
				(2,\y) -- (3,\y-1)
				(3,\y) -- (3,\y-1);		
			}
			\draw[dotted] (2,1) -- (3,1);
		\end{tikzpicture}
		\caption{}
	\end{subfigure}
	\begin{subfigure}{0.12\textwidth}
		\begin{tikzpicture}[font=\scriptsize]
			\draw (1,2) grid (2,4);
			\foreach \y in {2,...,5}
			\draw (1,\y) -- (2,\y-1);
			\draw
			(1,4) -- (1,5)
			(2,1) -- (2,2);
			\draw[dotted]
			(1,5) -- (2,5) -- (2,4)
			(2,1) -- (1,1) -- (1,2);
			\foreach \x in {1,2}
			\foreach \y in {1,...,5}
			\draw[fill=black] (\x,\y) circle(2.5pt);
			\node[anchor=east] at (1,2) {2};
			\node[anchor=east] at (1,3) {3};
			\node[anchor=east] at (1,4) {4};
			\node[anchor=east] at (1,5) {3};
			\node[anchor=north] at (2,1) {1};
			\node[anchor=west] at (2,2) {3};
			\node[anchor=west] at (2,3) {3};
			\node[anchor=west] at (2,4) {2};
		\end{tikzpicture}
		\caption{}
	\end{subfigure}
	\begin{subfigure}{0.12\textwidth}
	\centering
		\begin{tikzpicture}[font=\scriptsize]
			\draw (1,2) grid (2,4);
			\foreach \y in {3,...,5}
			\draw (1,\y) -- (2,\y-1);
			\draw
			(1,4) -- (1,5);
			\draw[dotted]
			(1,5) -- (2,5) -- (2,4)
			(2,1) -- (1,1) -- (1,2)
			(1,2) -- (2,1) -- (2,2);
			\foreach \x in {1,2}
			\foreach \y in {1,...,5}
			\draw[fill=black] (\x,\y) circle(2.5pt);
			\node[anchor=east] at (1,3) {3};
			\node[anchor=west] at (2,3) {3};
			\node[anchor=east] at (1,4) {4};
			\node[anchor=east] at (1,5) {3};
			\node[anchor=west] at (2,2) {2};
			\node[anchor=east] at (1,2) {1};
			\node[anchor=west] at (2,4) {2};
			\node[anchor=north] at (1,1) {\vphantom{1}};
		\end{tikzpicture}
		\caption{}
	\end{subfigure}
	\begin{subfigure}{0.12\textwidth}
		\begin{tikzpicture}[font=\scriptsize]
			\draw
			(1,5) -- (1,4) -- (2,4) -- (1,5);		
			\draw[dotted]
			(1,1) grid (2,3)
			(1,5) -- (2,5) -- (2,4)
			(1,3) -- (1,4) -- (2,3) -- (2,4)
			(1,2) -- (2,1)
			(1,3) -- (2,2);
			\foreach \x in {1,2}
			\foreach \y in {1,...,5}
			\draw[fill=black] (\x,\y) circle(2.5pt);
			\node[anchor=east] at (1,5) {3};
			\node[anchor=east] at (1,4) {2};
			\node[anchor=west] at (2,4) {1};
			\node[anchor=north] at (1,1) {\vphantom{1}};
		\end{tikzpicture}
		\caption{}
	\end{subfigure}
	\caption{Illustration of the sizes of the lists on the vertices at each step for \(G = C(3,5)\).}\label{F:C-algorithm}
\end{figure}

\section{Outline of the proof of \texorpdfstring{Theorem~\ref{T:Main}}{Theorem 1}}\label{S:Overview}

Suppose that \(T(r, s, t)\) has a \(5\)-list assignment \(\ListAsgn{L}\) in which not all the lists are identical.
For case~\ref{l1}, assume that for \(r \geq 4\) and \(s \geq 3\).
We will use Lemma~\ref{L:N-algorithm} to reduce the number of possibilities for the lists assigned by \(\ListAsgn{L}\), arriving at four criteria that \(\ListAsgn{L}\) must satisfy.

First, we show that if \(v\) is a vertex whose list \(\List{L}{v}\) is not contained in the union of its lists on its two neighbors on an adjacent column, then we can choose a color for \(v\) that is not in the union of the lists of those two neighboring vertices; then we use Lemma~\ref{L:cycle-choosability} to color the entire column containing \(v\), and notice that Lemma~\ref{L:N-algorithm} is now applicable.
Hence, we arrive at our first reduction, called criterion~\ref{criterion2.1}.

Next, we focus on a pair of adjacent vertices on the same column that have distinct lists.
Applying Lemmas~\ref{L:cycle-choosability} and \ref{L:N-algorithm} as before to this pair and their neighbors on an adjacent column, we arrive at criterion~\ref{criterion2.2}, illustrated in Figure~\ref{F:config1}.

Next, we focus on a pair of adjacent vertices \(u\) and \(v\) on adjacent columns that have the same lists.
Using criterion~\ref{criterion2.2} on these vertices, we deduce that there is a vertex \(w\) adjacent to both \(v\) and \(w\) and having the same list; we call this criterion~\ref{criterion2.3}.

Lastly, we focus on a face \(uvw\) in which the vertices \(v\) and \(w\) lie on the same column and have the same lists.
Using Lemmas~\ref{L:cycle-choosability} and \ref{L:N-algorithm}, we deduce that at least one of the two neighbors of \(u\) on the same column of \(u\) have a list identical to \(\List{L}{u}\); we call this criterion~\ref{criterion2.4}.

What remains is to exploit the structure of \(6\)-regular triangulations given by Theorem~\ref{T:Altshuler} with the rigidity imposed on the list assignment \(\ListAsgn{L}\) by criteria~\ref{criterion2.1} through \ref{criterion2.4}.
Lemma~\ref{L:iso-noniso-config} shows that either a vertex has a list that is different from the lists on any of its neighbors, or the list is shared by a neighbor in the same column.
This essentially gives a complete description of the list assignment \(\ListAsgn{L}\) from only the information of lists assigned on every four or five consecutive vertices in any one column:
the lists propagate across columns in any of precisely ten ways, as shown in Figures~\ref{F:noniso-config1}--\ref{F:iso-config}. This completes the preparation for the proof of case~\ref{l1}.

Ideally, one would like to complete the proof with another application of Lemma~\ref{L:N-algorithm}.
However, an induction argument as in the proof of the lemma does not directly work here, since a naive coloring of the column \(C_1\) need not give a cylindrical triangulation containing two adjacent vertices on \(C_2\) that have lists of size \(4\).
Applying a little more discretion in our choices, we use the small set of allowed configurations for \(\ListAsgn{L}\) to arrive at a two-step coloring scheme (assume \(r= 4\) without loss of generality):
\begin{enumerate}
\item Properly color \(C_1\) and a set \(J\) of alternate vertices in \(C_3\) such that (after reducing the lists) \(C_2\) has one \(4\)-list and the remaining as \(3\)-lists.
\item Properly color \(C_4\), then the remaining vertices in \(C_3\), and finally \(C_2\).
\end{enumerate}
Assuming step 1 is successfully achieved, we complete step 2 as illustrated in Figures~\ref{F:example1} and \ref{F:example2}.

Step 1 crucially uses the reduction into the ten cases illustrated in Figures~\ref{F:noniso-config1}--\ref{F:iso-config}.
Indeed, for each of the ten configurations that could appear on the column \(C_1\), we describe an explicit procedure for coloring \(C_1\), as well as for picking out the set \(J\) and a coloring for it, so that step 1 is completed.
This is a three stage process, depending on the configuration on \(C_1\) and the set \(J\).
This completes the proof of case~\ref{l1}.

Notice that Lemma~\ref{L:N-algorithm} is not applicable on \(C(r,s)\) for \(r \leq 2\), so cases~\ref{l2} and \ref{l3} of Theorem~\ref{T:Main} require a different line of attack.
So, we shall instead use the narrow length of the \(r \times s\) grid to place restrictions on the list assignment \(\ListAsgn{L}\).
The analysis is therefore shorter in these cases compared to case~\ref{l1} as discussed above.

For case~\ref{l4}, as mentioned earlier, the \(5\)-choosability of the \(3\)-chromatic \(6\)-regular toroidal triangulations was settled in a previous work~\cite{BalachandranSankarnarayanan2021}, but a small modification is required to get a linear time algorithm, for which we apply Lemma~\ref{L:BBS} instead of the theorem of Alon and Tarsi~\cite{AlonTarsi1992}.

Lastly, to show that the \(3\)-chromatic graphs \(T(r,s,t)\) are not \(3\)-choosable, we assign specific \(3\)-lists column-wise to \(T(r,s,t)\) such that distinct lists on adjacent columns share exactly two colors.
A crucial observation is that there is essentially a unique \(3\)-coloring on the subset of columns that are assigned the same list, and this lets us deduce that \(T(r,s,t)\) is not \(\ListAsgn{L}\)-colorable for this choice of \(3\)-list assignment.
A similar argument works when there are too few columns by instead assigning the lists row-wise.
A handful of exceptional cases are dealt with in \ref{S:Appendix} in an ad hoc manner.

\section{Preparation for the proof of \texorpdfstring{case~\ref{l1}}{case (1)} in \texorpdfstring{Theorem~\ref{T:Main}}{Theorem 1}}\label{S:Preparation}

For \(r \geq 4\), \(s \geq 3\) and \(0 \leq t \leq s - 1\), let
\(G \defn T(r, s, t)\). Fix \(\ListAsgn{L}\) to be a list assignment
on \(G\) of lists of size equal to \(5\).
We start be eliminating the trivial case: if all the lists of \(\ListAsgn{L}\)
are identical, then \(G\) is \(\ListAsgn{L}\)-choosable because
\(G\) is \(5\)-colorable~\cite{Thomassen1994a}. Moreover, a \(5\)-coloring
can be found in linear time: see~\cite{CollinsHutchinson1999,Sankarnarayanan2022,YehZhu2003}.

For a vertex \((i, j) \in C_{i}\), let its \defining{left neighbors} be the two adjacent vertices
in \(C_{i - 1}\), its \defining{right neighbors} be the two adjacent vertices in
\(C_{i + 1}\), and its \defining{vertical neighbors} be the two adjacent vertices
in \(C_{i}\). We shall repeatedly invoke Lemma~\ref{L:N-algorithm} to cut down on the possible
choices for the lists assigned by \(\ListAsgn{L}\), until it becomes simple
enough to directly specify a proper coloring.

\begin{lemma}\label{L:union1}
	Suppose that not all the lists in \(\ListAsgn{L}\) are identical.
	If there is a vertex \(v \in V(G)\)
	such that its list is not contained in the union of the lists of its
	two left neighbors, then \(G\) is \(\ListAsgn{L}\)-choosable in linear time.
\end{lemma}

\begin{proof}
	Choose a color for \(v\) that is not in the list of either left neighbor of \(v\),
	and extend the coloring to the cycle \(C_{i}\) containing \(v\) by
	Lemma~\ref{L:cycle-choosability}. Then,
	we are left to color a graph isomorphic to \(C(r - 1, s)\) equipped with lists
	whose sizes satisfy the hypotheses of Lemma~\ref{L:N-algorithm}.
	Hence, the coloring on \(C_{i}\) extends to a proper coloring
	of \(G\) in linear time by Lemma~\ref{L:N-algorithm}, and so we are done.
\end{proof}

Note that by Lemma~\ref{L:auto} the above lemma is also true when ``left neighbors'' is replaced
by ``right neighbors'' in the statement.
Thus, it suffices to assume that the list assignment \(\ListAsgn{L}\) satisfies
the following criterion:
\begin{enumerate}[label=(C\arabic*)]
	\item\label{criterion2.1}
	Not all the lists in \(\ListAsgn{L}\) are identical, and for every vertex \((i, j) \in V(G)\),
	\(\List{L}{(i, j)} \subseteq \List{L}{(i - 1, j)} \cup \List{L}{(i - 1, j + 1)}\)
	and \(\List{L}{(i, j)} \subseteq \List{L}{(i + 1, j)} \cup \List{L}{(i + 1, j - 1)}\).
\end{enumerate}
In particular, we may assume that no column has identical lists,
for if \(C_{i}\) has identical lists, then so do \(C_{i - 1}\) and \(C_{i + 1}\)
by criterion~\ref{criterion2.1}, so all the lists in \(\ListAsgn{L}\) are identical by induction,
a contradiction.

\begin{lemma}\label{L:three-criteria}
	Suppose that \(\ListAsgn{L}\) satisfies criterion~\ref{criterion2.1}.
	\begin{enumerate}[label=(\arabic*),ref=\ref{L:three-criteria}(\arabic*)]
		\item\label{LI:1} Let \((i, j), (i, j - 1) \in V(G)\) have distinct lists.
		Suppose one of the following conditions holds:
		\begin{enumerate}[label=(\alph*),ref=\theenumi\alph*]
			\item\label{LI:part1} \(\List{L}{(i, j)} \neq \List{L}{(i - 1, j + 1)}\) and \(\List{L}{(i, j - 1)} \neq \List{L}{(i - 1, j - 1)}\);
			\item\label{LI:part2} \(\List{L}{(i, j)} \neq \List{L}{(i - 1, j + 1)}\) and \(\List{L}{(i, j - 1)} \neq \List{L}{(i - 1, j)}\);
			\item\label{LI:part3} \(\List{L}{(i, j)} \neq \List{L}{(i - 1, j)}\) and \(\List{L}{(i, j - 1)} \neq \List{L}{(i - 1, j - 1)}\).
		\end{enumerate}
		Then, \(G\) is \(\ListAsgn{L}\)-choosable in linear time.
		
		\item\label{LI:2} Suppose \(u, v \in V(G)\) are adjacent
		vertices lying on distinct columns such that \(\List{L}{u} = \List{L}{v}\).
		If for every vertex \(w \in V(G)\) that is adjacent to both \(u\) and \(v\)
		we have \(\List{L}{w} \neq \List{L}{u}\), then \(G\) is \(\ListAsgn{L}\)-choosable in linear time.
		
		\item\label{LI:3} Let \((i, j) \in V(G)\) be a vertex such that both its left neighbors
		have lists identical to \(\List{L}{(i, j)}\). Suppose that
		\(\List{L}{(i, j)} \neq \List{L}{(i, j + 1)}\) and  \(\List{L}{(i, j)} \neq \List{L}{(i, j - 1)}\).
		Then, \(G\) is \(\ListAsgn{L}\)-choosable in linear time.
	\end{enumerate}
\end{lemma}

\begin{proof}\hfill
	\begin{enumerate}[(1)]
		\item
		\begin{enumerate}
			\item Choose a color for \((i - 1, j + 1)\) from \(\List{L}{(i - 1, j + 1)} \setminus \List{L}{(i, j)}\),
			for \((i - 1, j - 1)\) from \(\List{L}{(i - 1, j - 1)} \setminus \List{L}{(i, j - 1)}\),
			and extend this to a proper coloring of \(C_{i - 1}\) by Lemma~\ref{L:cycle-choosability}. Then, we are
			in the scenario of Lemma~\ref{L:N-algorithm}, and so we are done.
			
			\item\label{pf:b}
			Choose a color \(c\) for \((i - 1, j)\) from \(\List{L}{(i - 1, j)} \setminus \List{L}{(i, j - 1)}\).
			By criterion~\ref{criterion2.1}, \(c \in \List{L}{(i, j)}\), so there exists a color
			\(d\ (\neq c) \in \List{L}{(i - 1, j + 1)} \setminus \List{L}{(i, j)}\).
			Color \((i - 1, j + 1)\) with \(d\) and extend this to a proper
			coloring of \(C_{i - 1}\) by Lemma~\ref{L:cycle-choosability}.
			Then, we are in the scenario of Lemma~\ref{L:N-algorithm}, and so we are done.
			
			\item This is similar to the proof of Lemma~\ref{LI:part2} above.
			Choose a color \(c\) for \((i - 1, j)\) from \(\List{L}{(i - 1, j)} \setminus \List{L}{(i, j)}\).
			By criterion~\ref{criterion2.1}, \(c \in \List{L}{(i, j - 1)}\), so there exists a color
			\(d\ (\neq c) \in \List{L}{(i - 1, j - 1)} \setminus \List{L}{(i, j - 1)}\).
			Color \((i - 1, j - 1)\) with \(d\) and extend this to a proper
			coloring of \(C_{i - 1}\) by Lemma~\ref{L:cycle-choosability}.
			Then, we are in the scenario of Lemma~\ref{L:N-algorithm}, and so we are done.
		\end{enumerate}
		
		\item Suppose \(u = (i, j)\) and \(v = (i - 1, j)\),
		Then, neither \((i - 1, j + 1)\) nor \((i, j - 1)\) has a list identical to \(\List{L}{u}\).
		But then we are in the scenario of Lemma~\ref{LI:part2}, so \(G\) is \(\ListAsgn{L}\)-choosable
		in linear time.
		Similarly, let \(u = (i, j - 1)\) and \(w = (i - 1, j)\),
		Then, neither \((i - 1, j - 1)\) nor \((i, j)\) has a list identical to \(\List{L}{u}\).
		But then we are in the scenario of Lemma~\ref{LI:part3}, so again \(G\) is \(\ListAsgn{L}\)-choosable
		in linear time.
		
		\item Choose a color for \((i, j + 1)\) from \(\List{L}{(i, j + 1)} \setminus \List{L}{(i - 1, j + 1)}\),
		for \((i, j - 1)\) from \(\List{L}{(i, j - 1)} \setminus \List{L}{(i - 1, j)}\),
		and extend this to a proper coloring of \(C_{i}\) by Lemma~\ref{L:cycle-choosability}. Then,
		we are in the scenario of Lemma~\ref{L:N-algorithm}, so we are done.
	\end{enumerate}
\end{proof}

Note that, by Lemma~\ref{L:auto}, Lemma~\ref{LI:1} is also true when the list
assignment \(\ListAsgn{L}\) instead satisfies one of three analogous conditions
relating the lists on \((i, j)\) and \((i, j - 1)\) with their right neighbors,
and Lemma~\ref{LI:3} is also true when ``left neighbors'' is replaced
by ``right neighbors'' in the statement.

Thus, in addition to criterion~\ref{criterion2.1}, we may also assume the following criteria:
\begin{enumerate}[label=(C\arabic*)]
	\setcounter{enumi}{1}
	\item\label{criterion2.2} Whenever \((i, j)\) and \((i, j - 1)\)
	have distinct lists assigned by \(\ListAsgn{L}\), one of the following
	three configurations holds:
	\begin{enumerate}[(a)]
		\item\label{config-a} \(\List{L}{(i, j)} = \List{L}{(i - 1, j + 1)}\) and \(\List{L}{(i, j - 1)} = \List{L}{(i - 1, j - 1)}\);
		\item\label{config-b} \(\List{L}{(i, j)} = \List{L}{(i - 1, j + 1)} = \List{L}{(i - 1, j)}\) and \(\List{L}{(i, j - 1)} \neq \List{L}{(i - 1, j - 1)}\);
		\item\label{config-c} \(\List{L}{(i, j)} \neq \List{L}{(i - 1, j + 1)}\) and \(\List{L}{(i, j - 1)} = \List{L}{(i - 1, j)} = \List{L}{(i - 1, j - 1)}\).
	\end{enumerate}
	
	\item\label{criterion2.3} whenever \(u\) and \(v\) are adjacent vertices
	on distinct columns with \(\List{L}{u} = \List{L}{v}\),
	there is a vertex \(w\) adjacent to both \(u\) and \(v\) such that
	\(\List{L}{w} = \List{L}{u} = \List{L}{v}\).
	
	\item\label{criterion2.4} whenever \(u\), \(v\) and \(w\) are mutually adjacent vertices having identical lists,
	with \(v\) and \(w\) lying on the same column, at least one of the vertical neighbors
	of \(u\) has a list identical to \(\List{L}{u}\).	
\end{enumerate}

\begin{figure}
\centering
	\begin{subfigure}{0.15\textwidth}
	\centering
		\begin{tikzpicture}[font=\scriptsize]
			\draw
			(1,1) grid (2,3)
			(1,2) -- (2,1)
			(1,3) -- (2,2);
			\draw[dotted]
			(1,0) -- (1,1) -- (2,0) -- (2,1)
			(1,3) -- (1,4) -- (2,3) -- (2,4);
			\draw[fill=myblue] (1,1) circle(2.5pt) node[anchor=east] {\(\List{L}{2}\)};
			\draw[fill=myyellow] (1,2) circle(2.5pt) node[anchor=east] {};
			\draw[fill=myred] (1,3) circle(2.5pt) node[anchor=east] {\(\List{L}{1}\)};
			\draw[fill=myblue] (2,1) circle(2.5pt) node[anchor=west] {\(\List{L}{2}\)};
			\draw[fill=myred] (2,2) circle(2.5pt) node[anchor=west] {\(\List{L}{1}\)};
			\draw[fill=myyellow] (2,3) circle(2.5pt) node[anchor=west] {};
			
			\node[anchor=south] at (1,4) {\(C_{i - 1}\)};
			\node[anchor=south] at (2,4) {\(C_{i}\)};
		\end{tikzpicture}
		\caption{}
	\end{subfigure}
	\begin{subfigure}{0.15\textwidth}
	\centering
		\begin{tikzpicture}[font=\scriptsize]
			\draw
			(1,1) grid (2,3)
			(1,2) -- (2,1)
			(1,3) -- (2,2);
			\draw[dotted]
			(1,0) -- (1,1) -- (2,0) -- (2,1)
			(1,3) -- (1,4) -- (2,3) -- (2,4);
			\draw[fill=mygreen] (1,1) circle(2.5pt) node[anchor=east] {\(\List{L}{3}\)};
			\draw[fill=myred] (1,2) circle(2.5pt) node[anchor=east] {\(\List{L}{1}\)};
			\draw[fill=myred] (1,3) circle(2.5pt) node[anchor=east] {\(\List{L}{1}\)};
			\draw[fill=myblue] (2,1) circle(2.5pt) node[anchor=west] {\(\List{L}{2}\)};
			\draw[fill=myred] (2,2) circle(2.5pt) node[anchor=west] {\(\List{L}{1}\)};
			\draw[fill=myyellow] (2,3) circle(2.5pt) node[anchor=west] {};
			
			\node[anchor=south] at (1,4) {\(C_{i - 1}\)};
			\node[anchor=south] at (2,4) {\(C_{i}\)};
		\end{tikzpicture}
		\caption{}
	\end{subfigure}
	\begin{subfigure}{0.15\textwidth}
	\centering
		\begin{tikzpicture}[font=\scriptsize]
			\draw
			(1,1) grid (2,3)
			(1,2) -- (2,1)
			(1,3) -- (2,2);
			\draw[dotted]
			(1,0) -- (1,1) -- (2,0) -- (2,1)
			(1,3) -- (1,4) -- (2,3) -- (2,4);
			\draw[fill=myblue] (1,1) circle(2.5pt) node[anchor=east] {\(\List{L}{2}\)};
			\draw[fill=myblue] (1,2) circle(2.5pt) node[anchor=east] {\(\List{L}{2}\)};
			\draw[fill=mygreen] (1,3) circle(2.5pt) node[anchor=east] {\(\List{L}{3}\)};
			\draw[fill=myblue] (2,1) circle(2.5pt) node[anchor=west] {\(\List{L}{2}\)};
			\draw[fill=myred] (2,2) circle(2.5pt) node[anchor=west] {\(\List{L}{1}\)};
			\draw[fill=myyellow] (2,3) circle(2.5pt) node[anchor=west] {};
			
			\node[anchor=south] at (1,4) {\(C_{i - 1}\)};
			\node[anchor=south] at (2,4) {\(C_{i}\)};
		\end{tikzpicture}
		\caption{}
	\end{subfigure}
	\caption{Illustrations of configurations~\ref{config-a} through~\ref{config-c} in criterion~\ref{criterion2.2}.}\label{F:config1}
\end{figure}

The configurations~\ref{config-a} through~\ref{config-c} in criterion~\ref{criterion2.2} are illustrated in Figure~\ref{F:config1}.
By Lemma~\ref{L:auto}, we also assume one of three analogous configurations holds
for the lists on the right neighbors of \((i, j)\) and \((i, j - 1)\) under
the hypothesis of criterion~\ref{criterion2.2}, but for the sake of brevity we avoid listing them explicitly.

We now make the following definitions. For a list \(L\),
define the \defining{list-class} of \(L\) in \(G\), denoted \(G[L]\),
to be induced subgraph of \(G\) on those vertices \(v\) such that \(\List{L}{v} = L\).
Let \(L \in \ListAsgn{L}\) and let \(H\) be a (maximal connected) component
of \(G[L]\). If \(V(H)\) is a singleton, we call \(H\) an \defining{isolated component},
else we call \(H\) a \defining{nonisolated component}.

\begin{lemma}\label{L:iso-noniso-config}
	Suppose that \(\ListAsgn{L}\) satisfies criteria~\ref{criterion2.1} through~\ref{criterion2.4}.
	\begin{enumerate}[label=(\arabic*),ref=\ref{L:iso-noniso-config}(\arabic*)]
		\item\label{LI:iso-config} Let \(H\) be an isolated component of a list-class \(G[L]\), with
		\(V(H) = \Set{ (i, j) }\). Then, there are distinct lists \(L', L'' \in \ListAsgn{L}\)
		such that \(\List{L}{(i - 1, j + 1)} = \List{L}{(i, j + 1)} = \List{L}{(i + 1, j + 1)} = \List{L}{(i + 1, j)} = L'\)
		and \(\List{L}{(i - 1, j)} = \List{L}{(i - 1, j - 1)} = \List{L}{(i, j - 1)} = \List{L}{(i + 1, j - 1)} = L''\).
		
		\item\label{LI:noniso-config} Let \(H\) be a nonisolated component of a list-class \(G[L]\), with \(v \in V(H)\).
		Then, at least one vertical neighbor of \(v\) also belongs to \(V(H)\).
	\end{enumerate}
\end{lemma}

\begin{proof}\hfill
	\begin{enumerate}[(1)]
		\item Since \((i, j)\) belongs to an isolated component of \(G[L]\), the lists
		on its vertical neighbors are distinct from \(L\). Let \(\List{L}{(i, j + 1)} = L'\)
		and \(\List{L}{(i, j - 1)} = L''\). By applying criterion~\ref{criterion2.2} on the vertices
		\((i, j)\) and \((i, j - 1)\) with respect to their left neighbors, we see that
		only configuration~\ref{config-c} can hold, else \((i, j)\) will not belong to an isolated component.
		Thus, \(\List{L}{(i - 1, j)} = \List{L}{(i - 1, j - 1)} = L''\) and \(\List{L}{(i, j + 1)} \neq L\).
		
		Now, if \(\List{L}{(i - 1, j)} \neq L'\), then we can apply criterion~\ref{criterion2.2} on the vertices
		\((i - 1, j + 1)\) and \((i - 1, j)\) with respect to their right neighbors, and we see that
		none of the analogues of configurations~\ref{config-a} to~\ref{config-c} hold, a contradiction. Hence,
		\(\List{L}{(i - 1, j)} = L'\).
		
		Next, by Lemma~\ref{L:auto}, we also get \(\List{L}{(i + 1, j + 1)} = \List{L}{(i + 1, j)} = L'\)
		and \(\List{L}{(i, j - 1)} = L''\).
		
		Lastly, if \(L' = L''\), then we will also have \(L = L'\) by criterion~\ref{criterion2.1}, a contradiction.
		
		\item Since \(v\) is assumed to belong to a nonisolated component \(H\) of some list-class
		\(G[L]\), let \(u \in V(H)\) with \(u\) adjacent to \(v\). If \(u\) lies in the
		same column as \(v\), then we are done, so assume that \(u\) and \(v\) lie
		in distinct columns. Then, by criterion~\ref{criterion2.3}, there is a vertex \(w\) adjacent
		to both \(u\) and \(v\) such that \(w \in V(H)\). If \(w\) lies in the same column
		as \(v\), then we are done, so assume that \(v\) and \(w\) lie in distinct columns.
		Then, \(u\) and \(w\) lie on the same column, so by criterion~\ref{criterion2.4} at least
		one of the vertical neighbors of \(v\) also belongs to \(V(H)\).
	\end{enumerate}
\end{proof}

\begin{figure}
\centering
	\begin{tikzpicture}[font=\scriptsize]
		\draw (1,1) grid (3,3);
		\foreach \x in {1,2}
		\foreach \y in {2,3}
		\draw (\x,\y) -- (\x+1, \y-1);
		\draw[dotted]
		(1,0) -- (1,1) -- (2,0) -- (2,1) -- (3,0) -- (3,1)
		(1,3) -- (1,4) -- (2,3) -- (2,4) -- (3,3) -- (3,4);
		
		\draw[fill=myblue] (1,1) circle(2.5pt) node[anchor=east] {\(L''\)};
		\draw[fill=myblue] (1,2) circle(2.5pt) node[anchor=east] {\(L''\)};
		\draw[fill=mygreen] (1,3) circle(2.5pt) node[anchor=east] {\(L'\)};
		
		\draw[fill=myblue] (2,1) circle(2.5pt) node[anchor=south west] {\(L''\)};
		\draw[fill=myred] (2,2) circle(2.5pt) node[anchor=south west] {\(L\)};
		\draw[fill=mygreen] (2,3) circle(2.5pt) node[anchor=south west] {\(L'\)};
		
		\draw[fill=myblue] (3,1) circle(2.5pt) node[anchor=west] {\(L''\)};
		\draw[fill=mygreen] (3,2) circle(2.5pt) node[anchor=west] {\(L'\)};
		\draw[fill=mygreen] (3,3) circle(2.5pt) node[anchor=west] {\(L'\)};
		
		\node[anchor=south] at (1,4) {\(C_{i - 1}\)};
		\node[anchor=south] at (2,4) {\(C_{i}\)};
		\node[anchor=south] at (3,4) {\(C_{i + 1}\)};
	\end{tikzpicture}
	\caption{The configuration of an isolated component in Lemma~\ref{LI:iso-config}.}\label{F:iso}
\end{figure}

The configuration in Lemma~\ref{LI:iso-config} is illustrated in Figure~\ref{F:iso}.
Note that Lemma~\ref{LI:noniso-config} implies that for every \(v \in V(H)\),
where \(H\) is a nonisolated component of some list-class \(G[L]\), at least
one left neighbor and one right neighbor of \(v\) also belongs to \(V(H)\),
by criterion~\ref{criterion2.1}.
Hence, Lemma~\ref{LI:noniso-config} can be applied successively on vertices across columns,
starting from any \(v \in V(H)\).
Thus, if \((i, j)\) and \((i, j - 1)\) are adjacent vertices in the column \(C_{i}\)
with distinct lists, then using Lemma~\ref{L:iso-noniso-config}, we can pin down the possible list configurations
on the nearby vertices in the columns \(C_{i+1}\) and \(C_{i+2}\) to a manageable
number, as follows.

\begin{lemma}\label{L:without-iso}
	Suppose that \(\ListAsgn{L}\) satisfies criteria~\ref{criterion2.1} to~\ref{criterion2.4}.
	Let \((i, j + 1), (i, j) \in V(G)\) have distinct lists \(\List{L}{1}, \List{L}{2}\),
	respectively, and suppose that neither vertex belongs to an isolated component.
	Then, one of the following configurations holds:
	\begin{enumerate}[label=(\Roman*)]
		\item\label{config-I} The vertices \((i, k)\), \((i + 1, k)\) and \((i + 2, k)\) have lists
		identical to \(\List{L}{1}\) for \(k = j + 2, j + 1\),
		and have lists identical to \(\List{L}{2}\) for \(k = j, j - 1\).
		
		\item\label{config-II} The vertices \((i, k)\), \((i + 1, k)\) and \((i + 2, k - 1)\) have lists
		identical to \(\List{L}{1}\) for \(k = j + 2, j + 1\),
		and have lists identical to \(\List{L}{2}\) for \(k = j, j - 1\).
		
		\item\label{config-III} The vertices \((i, k)\), \((i + 1, k - 1)\) and \((i + 2, k - 1)\) have lists
		identical to \(\List{L}{1}\) for \(k = j + 2, j + 1\),
		and have lists identical to \(\List{L}{2}\) for \(k = j, j - 1\).
		
		\item\label{config-IV} The vertices \((i, k)\), \((i + 1, k - 1)\) and \((i + 2, k - 2)\) have lists
		identical to \(\List{L}{1}\) for \(k = j + 2, j + 1\),
		and have lists identical to \(\List{L}{2}\) for \(k = j, j - 1\).
		
		\item\label{config-V} The vertices \((i, k)\), \((i + 1, k)\) and \((i + 2, k)\) have lists
		identical to \(\List{L}{1}\) for \(k = j + 2, j + 1\), the vertices
		\((i, k)\), \((i + 1, k)\) and \((i + 2, k - 1)\)
		have lists identical to \(\List{L}{2}\) for \(k = j, j - 1\),
		and the vertex \((i + 2, j)\) belongs to an isolated component of some list-class \(G[\List{L}{3}]\),
		where \(\List{L}{3} \neq \List{L}{1}\) and \(\List{L}{3} \neq \List{L}{2}\).
		
		\item\label{config-VI} The vertices \((i, k)\), \((i + 1, k - 1)\) and \((i + 2, k - 1)\) have lists
		identical to \(\List{L}{1}\) for \(k = j + 2, j + 1\), the vertices
		\((i, k)\), \((i + 1, k - 1)\) and \((i + 2, k - 2)\)
		have lists identical to \(\List{L}{2}\) for \(k = j, j - 1\),
		and the vertex \((i + 2, j - 1)\) belongs to an isolated component of some list-class \(G[\List{L}{3}]\),
		where \(\List{L}{3} \neq \List{L}{1}\) and \(\List{L}{3} \neq \List{L}{2}\).
		
		\item\label{config-VII} The vertices \((i, k)\), \((i + 1, k)\) and \((i + 2, k - 1)\) have lists
		identical to \(\List{L}{1}\) for \(k = j + 2, j + 1\), the vertices
		\((i, k)\), \((i + 1, k - 1)\) and \((i + 2, k - 1)\)
		have lists identical to \(\List{L}{2}\) for \(k = j, j - 1\),
		and the vertex \((i + 1, j)\) belongs to an isolated component of some list-class \(G[\List{L}{3}]\),
		where \(\List{L}{3} \neq \List{L}{1}\) and \(\List{L}{3} \neq \List{L}{2}\).
	\end{enumerate}
\end{lemma}

\begin{figure}
\centering
\begingroup
  \renewcommand\thesubfigure{\Roman{subfigure}}

	\begin{subfigure}{0.22\textwidth}
	\centering
		\begin{tikzpicture}[font=\scriptsize]
			\draw (1,1) grid (3,6);
			\foreach \x in {1,2}
			\foreach \y in {2,...,6}
			\draw (\x,\y) -- (\x+1,\y-1);
			\draw[dotted]
			(1,0) -- (1,1) -- (2,0) -- (2,1) -- (3,0) -- (3,1)
			(1,6) -- (1,7) -- (2,6) -- (2,7) -- (3,6) -- (3,7);
			
			\draw[fill=myyellow] (1,1) circle(2.5pt) node[anchor=east] {};
			\draw[fill=myyellow] (1,2) circle(2.5pt) node[anchor=east] {};
			\draw[fill=myblue] (1,3) circle(2.5pt) node[anchor=east] {\(\List{L}{2}\)};
			\draw[fill=myblue] (1,4) circle(2.5pt) node[anchor=east] {\(\List{L}{2}\)};
			\draw[fill=myred] (1,5) circle(2.5pt) node[anchor=east] {\(\List{L}{1}\)};
			\draw[fill=myred] (1,6) circle(2.5pt) node[anchor=east] {\(\List{L}{1}\)};
			\draw[fill=myyellow] (2,1) circle(2.5pt) node[anchor=south west] {};
			\draw[fill=myyellow] (2,2) circle(2.5pt) node[anchor=south west] {};
			\draw[fill=myblue] (2,3) circle(2.5pt) node[anchor=south west] {\(\List{L}{2}\)};
			\draw[fill=myblue] (2,4) circle(2.5pt) node[anchor=south west] {\(\List{L}{2}\)};
			\draw[fill=myred] (2,5) circle(2.5pt) node[anchor=south west] {\(\List{L}{1}\)};
			\draw[fill=myred] (2,6) circle(2.5pt) node[anchor=south west] {\(\List{L}{1}\)};
			\draw[fill=myyellow] (3,1) circle(2.5pt) node[anchor=west] {};
			\draw[fill=myyellow] (3,2) circle(2.5pt) node[anchor=west] {};
			\draw[fill=myblue] (3,3) circle(2.5pt) node[anchor=west] {\(\List{L}{2}\)};
			\draw[fill=myblue] (3,4) circle(2.5pt) node[anchor=west] {\(\List{L}{2}\)};
			\draw[fill=myred] (3,5) circle(2.5pt) node[anchor=west] {\(\List{L}{1}\)};
			\draw[fill=myred] (3,6) circle(2.5pt) node[anchor=west] {\(\List{L}{1}\)};
			
			\node[anchor=south] at (1,7) {\(C_{i}\)};
			\node[anchor=south] at (2,7) {\(C_{i + 1}\)};
			\node[anchor=south] at (3,7) {\(C_{i + 2}\)};
		\end{tikzpicture}
		\caption{}
	\end{subfigure}
	\begin{subfigure}{0.22\textwidth}
	\centering
		\begin{tikzpicture}[font=\scriptsize]
			\draw (1,1) grid (3,6);
			\foreach \x in {1,2}
			\foreach \y in {2,...,6}
			\draw (\x,\y) -- (\x+1,\y-1);
			\draw[dotted]
			(1,0) -- (1,1) -- (2,0) -- (2,1) -- (3,0) -- (3,1)
			(1,6) -- (1,7) -- (2,6) -- (2,7) -- (3,6) -- (3,7);
			
			\draw[fill=myyellow] (1,1) circle(2.5pt) node[anchor=east] {};
			\draw[fill=myyellow] (1,2) circle(2.5pt) node[anchor=east] {};
			\draw[fill=myblue] (1,3) circle(2.5pt) node[anchor=east] {\(\List{L}{2}\)};
			\draw[fill=myblue] (1,4) circle(2.5pt) node[anchor=east] {\(\List{L}{2}\)};
			\draw[fill=myred] (1,5) circle(2.5pt) node[anchor=east] {\(\List{L}{1}\)};
			\draw[fill=myred] (1,6) circle(2.5pt) node[anchor=east] {\(\List{L}{1}\)};
			\draw[fill=myyellow] (2,1) circle(2.5pt) node[anchor=south west] {};
			\draw[fill=myyellow] (2,2) circle(2.5pt) node[anchor=south west] {};
			\draw[fill=myblue] (2,3) circle(2.5pt) node[anchor=south west] {\(\List{L}{2}\)};
			\draw[fill=myblue] (2,4) circle(2.5pt) node[anchor=south west] {\(\List{L}{2}\)};
			\draw[fill=myred] (2,5) circle(2.5pt) node[anchor=south west] {\(\List{L}{1}\)};
			\draw[fill=myred] (2,6) circle(2.5pt) node[anchor=south west] {\(\List{L}{1}\)};
			\draw[fill=myyellow] (3,1) circle(2.5pt) node[anchor=west] {};
			\draw[fill=myblue] (3,2) circle(2.5pt) node[anchor=west] {\(\List{L}{2}\)};
			\draw[fill=myblue] (3,3) circle(2.5pt) node[anchor=west] {\(\List{L}{2}\)};
			\draw[fill=myred] (3,4) circle(2.5pt) node[anchor=west] {\(\List{L}{1}\)};
			\draw[fill=myred] (3,5) circle(2.5pt) node[anchor=west] {\(\List{L}{1}\)};
			\draw[fill=myyellow] (3,6) circle(2.5pt) node[anchor=west] {};
			
			\node[anchor=south] at (1,7) {\(C_{i}\)};
			\node[anchor=south] at (2,7) {\(C_{i + 1}\)};
			\node[anchor=south] at (3,7) {\(C_{i + 2}\)};
		\end{tikzpicture}
		\caption{}
	\end{subfigure}
	\begin{subfigure}{0.22\textwidth}
	\centering
		\begin{tikzpicture}[font=\scriptsize]
			\draw (1,1) grid (3,6);
			\foreach \x in {1,2}
			\foreach \y in {2,...,6}
			\draw (\x,\y) -- (\x+1,\y-1);
			\draw[dotted]
			(1,0) -- (1,1) -- (2,0) -- (2,1) -- (3,0) -- (3,1)
			(1,6) -- (1,7) -- (2,6) -- (2,7) -- (3,6) -- (3,7);
			
			\draw[fill=myyellow] (1,1) circle(2.5pt) node[anchor=east] {};
			\draw[fill=myyellow] (1,2) circle(2.5pt) node[anchor=east] {};
			\draw[fill=myblue] (1,3) circle(2.5pt) node[anchor=east] {\(\List{L}{2}\)};
			\draw[fill=myblue] (1,4) circle(2.5pt) node[anchor=east] {\(\List{L}{2}\)};
			\draw[fill=myred] (1,5) circle(2.5pt) node[anchor=east] {\(\List{L}{1}\)};
			\draw[fill=myred] (1,6) circle(2.5pt) node[anchor=east] {\(\List{L}{1}\)};
			\draw[fill=myyellow] (2,1) circle(2.5pt) node[anchor=south west] {};
			\draw[fill=myblue] (2,2) circle(2.5pt) node[anchor=south west] {\(\List{L}{2}\)};
			\draw[fill=myblue] (2,3) circle(2.5pt) node[anchor=south west] {\(\List{L}{2}\)};
			\draw[fill=myred] (2,4) circle(2.5pt) node[anchor=south west] {\(\List{L}{1}\)};
			\draw[fill=myred] (2,5) circle(2.5pt) node[anchor=south west] {\(\List{L}{1}\)};
			\draw[fill=myyellow] (2,6) circle(2.5pt) node[anchor=south west] {};
			\draw[fill=myyellow] (3,1) circle(2.5pt) node[anchor=west] {};
			\draw[fill=myblue] (3,2) circle(2.5pt) node[anchor=west] {\(\List{L}{2}\)};
			\draw[fill=myblue] (3,3) circle(2.5pt) node[anchor=west] {\(\List{L}{2}\)};
			\draw[fill=myred] (3,4) circle(2.5pt) node[anchor=west] {\(\List{L}{1}\)};
			\draw[fill=myred] (3,5) circle(2.5pt) node[anchor=west] {\(\List{L}{1}\)};
			\draw[fill=myyellow] (3,6) circle(2.5pt) node[anchor=west] {};
			
			\node[anchor=south] at (1,7) {\(C_{i}\)};
			\node[anchor=south] at (2,7) {\(C_{i + 1}\)};
			\node[anchor=south] at (3,7) {\(C_{i + 2}\)};
		\end{tikzpicture}
		\caption{}
	\end{subfigure}
	\begin{subfigure}{0.22\textwidth}
	\centering
		\begin{tikzpicture}[font=\scriptsize]
			\draw (1,1) grid (3,6);
			\foreach \x in {1,2}
			\foreach \y in {2,...,6}
			\draw (\x,\y) -- (\x+1,\y-1);
			\draw[dotted]
			(1,0) -- (1,1) -- (2,0) -- (2,1) -- (3,0) -- (3,1)
			(1,6) -- (1,7) -- (2,6) -- (2,7) -- (3,6) -- (3,7);
			
			\draw[fill=myyellow] (1,1) circle(2.5pt) node[anchor=east] {};
			\draw[fill=myyellow] (1,2) circle(2.5pt) node[anchor=east] {};
			\draw[fill=myblue] (1,3) circle(2.5pt) node[anchor=east] {\(\List{L}{2}\)};
			\draw[fill=myblue] (1,4) circle(2.5pt) node[anchor=east] {\(\List{L}{2}\)};
			\draw[fill=myred] (1,5) circle(2.5pt) node[anchor=east] {\(\List{L}{1}\)};
			\draw[fill=myred] (1,6) circle(2.5pt) node[anchor=east] {\(\List{L}{1}\)};
			\draw[fill=myyellow] (2,1) circle(2.5pt) node[anchor=south west] {};
			\draw[fill=myblue] (2,2) circle(2.5pt) node[anchor=south west] {\(\List{L}{2}\)};
			\draw[fill=myblue] (2,3) circle(2.5pt) node[anchor=south west] {\(\List{L}{2}\)};
			\draw[fill=myred] (2,4) circle(2.5pt) node[anchor=south west] {\(\List{L}{1}\)};
			\draw[fill=myred] (2,5) circle(2.5pt) node[anchor=south west] {\(\List{L}{1}\)};
			\draw[fill=myyellow] (2,6) circle(2.5pt) node[anchor=south west] {};
			\draw[fill=myblue] (3,1) circle(2.5pt) node[anchor=west] {\(\List{L}{2}\)};
			\draw[fill=myblue] (3,2) circle(2.5pt) node[anchor=west] {\(\List{L}{2}\)};
			\draw[fill=myred] (3,3) circle(2.5pt) node[anchor=west] {\(\List{L}{1}\)};
			\draw[fill=myred] (3,4) circle(2.5pt) node[anchor=west] {\(\List{L}{1}\)};
			\draw[fill=myyellow] (3,5) circle(2.5pt) node[anchor=west] {};
			\draw[fill=myyellow] (3,6) circle(2.5pt) node[anchor=west] {};
			
			\node[anchor=south] at (1,7) {\(C_{i}\)};
			\node[anchor=south] at (2,7) {\(C_{i + 1}\)};
			\node[anchor=south] at (3,7) {\(C_{i + 2}\)};
		\end{tikzpicture}
		\caption{}
	\end{subfigure}
	\caption{Illustration of configurations \ref{config-I} through \ref{config-IV} of Lemma~\ref{L:without-iso}.}\label{F:noniso-config1}
\endgroup
\end{figure}

\begin{figure}
\centering
\begingroup
  \renewcommand\thesubfigure{\Roman{subfigure}}

	\begin{subfigure}{0.3\textwidth}
	\centering
	\setcounter{subfigure}{4}
		\begin{tikzpicture}[font=\scriptsize]
			\draw (1,1) grid (3,6);
			\foreach \x in {1,2}
			\foreach \y in {2,...,6}
			\draw (\x,\y) -- (\x+1,\y-1);
			\draw[dotted]
			(1,0) -- (1,1) -- (2,0) -- (2,1) -- (3,0) -- (3,1)
			(1,6) -- (1,7) -- (2,6) -- (2,7) -- (3,6) -- (3,7);
			
			\draw[fill=myyellow] (1,1) circle(2.5pt) node[anchor=east] {};
			\draw[fill=myyellow] (1,2) circle(2.5pt) node[anchor=east] {};
			\draw[fill=myblue] (1,3) circle(2.5pt) node[anchor=east] {\(\List{L}{2}\)};
			\draw[fill=myblue] (1,4) circle(2.5pt) node[anchor=east] {\(\List{L}{2}\)};
			\draw[fill=myred] (1,5) circle(2.5pt) node[anchor=east] {\(\List{L}{1}\)};
			\draw[fill=myred] (1,6) circle(2.5pt) node[anchor=east] {\(\List{L}{1}\)};
			\draw[fill=myyellow] (2,1) circle(2.5pt) node[anchor=south west] {};
			\draw[fill=myyellow] (2,2) circle(2.5pt) node[anchor=south west] {};
			\draw[fill=myblue] (2,3) circle(2.5pt) node[anchor=south west] {\(\List{L}{2}\)};
			\draw[fill=myblue] (2,4) circle(2.5pt) node[anchor=south west] {\(\List{L}{2}\)};
			\draw[fill=myred] (2,5) circle(2.5pt) node[anchor=south west] {\(\List{L}{1}\)};
			\draw[fill=myred] (2,6) circle(2.5pt) node[anchor=south west] {\(\List{L}{1}\)};
			\draw[fill=myyellow] (3,1) circle(2.5pt) node[anchor=west] {};
			\draw[fill=myblue] (3,2) circle(2.5pt) node[anchor=west] {\(\List{L}{2}\)};
			\draw[fill=myblue] (3,3) circle(2.5pt) node[anchor=west] {\(\List{L}{2}\)};
			\draw[fill=mygreen] (3,4) circle(2.5pt) node[anchor=west] {\(\List{L}{3}\)};
			\draw[fill=myred] (3,5) circle(2.5pt) node[anchor=west] {\(\List{L}{1}\)};
			\draw[fill=myred] (3,6) circle(2.5pt) node[anchor=west] {\(\List{L}{1}\)};
			
			\node[anchor=south] at (1,7) {\(C_{i}\)};
			\node[anchor=south] at (2,7) {\(C_{i + 1}\)};
			\node[anchor=south] at (3,7) {\(C_{i + 2}\)};
		\end{tikzpicture}
		\caption{}
	\end{subfigure}
	\begin{subfigure}{0.3\textwidth}
	\centering
		\begin{tikzpicture}[font=\scriptsize]
			\node at (2,8) {\vphantom{(2,7)}};
			\draw (1,1) grid (3,6);
			\foreach \x in {1,2}
			\foreach \y in {2,...,6}
			\draw (\x,\y) -- (\x+1,\y-1);
			\draw[dotted]
			(1,0) -- (1,1) -- (2,0) -- (2,1) -- (3,0) -- (3,1)
			(1,6) -- (1,7) -- (2,6) -- (2,7) -- (3,6) -- (3,7);
			
			\draw[fill=myyellow] (1,1) circle(2.5pt) node[anchor=east] {};
			\draw[fill=myyellow] (1,2) circle(2.5pt) node[anchor=east] {};
			\draw[fill=myblue] (1,3) circle(2.5pt) node[anchor=east] {\(\List{L}{2}\)};
			\draw[fill=myblue] (1,4) circle(2.5pt) node[anchor=east] {\(\List{L}{2}\)};
			\draw[fill=myred] (1,5) circle(2.5pt) node[anchor=east] {\(\List{L}{1}\)};
			\draw[fill=myred] (1,6) circle(2.5pt) node[anchor=east] {\(\List{L}{1}\)};
			\draw[fill=myyellow] (2,1) circle(2.5pt) node[anchor=south west] {};
			\draw[fill=myblue] (2,2) circle(2.5pt) node[anchor=south west] {\(\List{L}{2}\)};
			\draw[fill=myblue] (2,3) circle(2.5pt) node[anchor=south west] {\(\List{L}{2}\)};
			\draw[fill=myred] (2,4) circle(2.5pt) node[anchor=south west] {\(\List{L}{1}\)};
			\draw[fill=myred] (2,5) circle(2.5pt) node[anchor=south west] {\(\List{L}{1}\)};
			\draw[fill=myyellow] (2,6) circle(2.5pt) node[anchor=south west] {};
			\draw[fill=myblue] (3,1) circle(2.5pt) node[anchor=west] {\(\List{L}{2}\)};
			\draw[fill=myblue] (3,2) circle(2.5pt) node[anchor=west] {\(\List{L}{2}\)};
			\draw[fill=mygreen] (3,3) circle(2.5pt) node[anchor=west] {\(\List{L}{3}\)};
			\draw[fill=myred] (3,4) circle(2.5pt) node[anchor=west] {\(\List{L}{1}\)};
			\draw[fill=myred] (3,5) circle(2.5pt) node[anchor=west] {\(\List{L}{1}\)};
			\draw[fill=myyellow] (3,6) circle(2.5pt) node[anchor=west] {};
			
			\node[anchor=south] at (1,7) {\(C_{i}\)};
			\node[anchor=south] at (2,7) {\(C_{i + 1}\)};
			\node[anchor=south] at (3,7) {\(C_{i + 2}\)};
		\end{tikzpicture}
		\caption{}
	\end{subfigure}
	\begin{subfigure}{0.3\textwidth}
	\centering
		\begin{tikzpicture}[font=\scriptsize]
			\draw (1,1) grid (3,6);
			\foreach \x in {1,2}
			\foreach \y in {2,...,6}
			\draw (\x,\y) -- (\x+1,\y-1);
			\draw[dotted]
			(1,0) -- (1,1) -- (2,0) -- (2,1) -- (3,0) -- (3,1)
			(1,6) -- (1,7) -- (2,6) -- (2,7) -- (3,6) -- (3,7);
			
			\draw[fill=myyellow] (1,1) circle(2.5pt) node[anchor=east] {};
			\draw[fill=myyellow] (1,2) circle(2.5pt) node[anchor=east] {};
			\draw[fill=myblue] (1,3) circle(2.5pt) node[anchor=east] {\(\List{L}{2}\)};
			\draw[fill=myblue] (1,4) circle(2.5pt) node[anchor=east] {\(\List{L}{2}\)};
			\draw[fill=myred] (1,5) circle(2.5pt) node[anchor=east] {\(\List{L}{1}\)};
			\draw[fill=myred] (1,6) circle(2.5pt) node[anchor=east] {\(\List{L}{1}\)};
			\draw[fill=myyellow] (2,1) circle(2.5pt) node[anchor=south west] {};
			\draw[fill=myblue] (2,2) circle(2.5pt) node[anchor=south west] {\(\List{L}{2}\)};
			\draw[fill=myblue] (2,3) circle(2.5pt) node[anchor=south west] {\(\List{L}{2}\)};
			\draw[fill=mygreen] (2,4) circle(2.5pt) node[anchor=south west] {\(\List{L}{3}\)};
			\draw[fill=myred] (2,5) circle(2.5pt) node[anchor=south west] {\(\List{L}{1}\)};
			\draw[fill=myred] (2,6) circle(2.5pt) node[anchor=south west] {\(\List{L}{1}\)};
			\draw[fill=myyellow] (3,1) circle(2.5pt) node[anchor=west] {};
			\draw[fill=myblue] (3,2) circle(2.5pt) node[anchor=west] {\(\List{L}{2}\)};
			\draw[fill=myblue] (3,3) circle(2.5pt) node[anchor=west] {\(\List{L}{2}\)};
			\draw[fill=myred] (3,4) circle(2.5pt) node[anchor=west] {\(\List{L}{1}\)};
			\draw[fill=myred] (3,5) circle(2.5pt) node[anchor=west] {\(\List{L}{1}\)};
			\draw[fill=myyellow] (3,6) circle(2.5pt) node[anchor=west] {};
			
			\node[anchor=south] at (1,7) {\(C_{i}\)};
			\node[anchor=south] at (2,7) {\(C_{i + 1}\)};
			\node[anchor=south] at (3,7) {\(C_{i + 2}\)};
		\end{tikzpicture}
		\caption{}
	\end{subfigure}
	\caption{Illustration of the configurations \ref{config-V} through \ref{config-VII} of Lemma~\ref{L:without-iso}.}\label{F:noniso-config2}
\endgroup
\end{figure}

\begin{proof}
	We start with \(\List{L}{(i, j + 1)} = L_{1}\) and \(\List{L}{(i, j)} = L_{2}\).
	By Lemma~\ref{LI:noniso-config}, this implies that \(\List{L}{(i, j + 2)} = L_{1}\)
	and \(\List{L}{(i, j - 1)} = L_{2}\). By criterion~\ref{criterion2.1},
	\(\List{L}{(i + 1, j + 1)} = L_{1}\) and \(\List{L}{(i + 1, j - 1)} = L_{2}\).
	Now, again by Lemma~\ref{LI:noniso-config}, we have three cases:
	\begin{enumerate}[(\arabic*)]
		\item\label{case1} \(\List{L}{(i + 1, j + 2)} = L_{1}\) and \(\List{L}{(i + 1, j)} = L_{2}\);
		\item\label{case2} \(\List{L}{(i + 1, j)} = L_{1}\) and \(\List{L}{(i + 1, j - 2)} = L_{2}\);
		\item\label{case3} \(\List{L}{(i + 1, j + 2)} = L_{1}\), \(\List{L}{(i + 1, j - 2)} = L_{2}\) and
		\(\List{L}{(i + 1, j)} = L_{3}\) where \(L_{1} \neq L_{3}\) and \(L_{2} \neq L_{3}\). In particular,
		by Lemma~\ref{LI:noniso-config}, \((i + 1, j)\) must belong to an isolated component of the list-class \(G[L_{3}]\).
	\end{enumerate}
	We consider each of these cases in turn.
	
	First, suppose case~\ref{case1} holds. Then, by criterion~\ref{criterion2.1},
	\(\List{L}{(i + 2, j + 1)} = L_{1}\) and \(\List{L}{(i + 2, j - 1)} = L_{2}\).
	Then, again by Lemma~\ref{LI:noniso-config}, we have three cases:
	\begin{itemize}
		\item \(\List{L}{(i + 2, j + 2)} = L_{1}\) and \(\List{L}{(i + 2, j)} = L_{2}\). This
		is configuration~\ref{config-I}.
		\item \(\List{L}{(i + 2, j)} = L_{1}\) and \(\List{L}{(i + 2, j - 2)} = L_{2}\). This
		is configuration~\ref{config-II}.
		\item \(\List{L}{(i + 2, j + 2)} = L_{1}\), \(\List{L}{(i + 2, j - 2)} = L_{2}\) and
		\(\List{L}{(i + 2, j)} = L_{3}\) where \(L_{1} \neq L_{3}\) and \(L_{2} \neq L_{3}\). In particular,
		by Lemma~\ref{LI:noniso-config}, \((i + 2, j)\) must belong to an isolated component of the list-class \(G[L_{3}]\). This
		is configuration~\ref{config-V}.
	\end{itemize}
	
	Next, suppose case~\ref{case2} holds. Then, by criterion~\ref{criterion2.1},
	\(\List{L}{(i + 2, j)} = L_{1}\) and \(\List{L}{(i + 2, j - 2)} = L_{2}\).
	Again by Lemma~\ref{LI:noniso-config}, we have three cases:
	\begin{itemize}
		\item \(\List{L}{(i + 2, j + 1)} = L_{1}\) and \(\List{L}{(i + 2, j - 1)} = L_{2}\). This
		is configuration~\ref{config-III}.
		\item \(\List{L}{(i + 2, j - 1)} = L_{1}\) and \(\List{L}{(i + 2, j - 3)} = L_{2}\). This
		is configuration~\ref{config-IV}.
		\item \(\List{L}{(i + 2, j + 1)} = L_{1}\), \(\List{L}{(i + 2, j - 3)} = L_{2}\) and
		\(\List{L}{(i + 2, j - 1)} = L_{3}\) where \(L_{1} \neq L_{3}\) and \(L_{2} \neq L_{3}\). In particular,
		by Lemma~\ref{LI:noniso-config}, \((i + 2, j - 1)\) must belong to an isolated component of the list-class \(G[L_{3}]\). This
		is configuration~\ref{config-VI}.
	\end{itemize}
	
	Lastly, suppose case~\ref{case3} holds. Then, by Lemma~\ref{LI:iso-config},
	\(\List{L}{(i + 2, j + 1)} = \List{L}{(i + 2, j)} = L_{1}\) and
	\(\List{L}{(i + 2, j - 1)} = L_{2}\). By Lemma~\ref{LI:noniso-config}, we also have \(\List{L}{(i + 2, j - 2)} = L_{2}\).
	This is configuration~\ref{config-VII}.
\end{proof}

\begin{lemma}\label{L:with-iso}
	Suppose that \(\ListAsgn{L}\) satisfies criteria~\ref{criterion2.1} to~\ref{criterion2.4}.
	Let \((i, j + 1)\), \((i, j)\) and \((i, j - 1)\) have mutually distinct lists
	\(\List{L}{1}\), \(\List{L}{3}\) and \(\List{L}{2}\), respectively,
	and suppose that \((i, j)\) corresponds to an isolated component in \(G[\List{L}{3}]\). Then,
	one of the following configurations holds:
	\begin{enumerate}[label=(\Roman*)]
		\setcounter{enumi}{7}
		\item\label{config-I'} The vertices \((i, k)\), \((i + 1, k - 1)\) and \((i + 2, k - 1)\) have lists
		identical to \(\List{L}{1}\) for \(k = j + 2, j + 1\), the vertices
		\((i, k)\), \((i + 1, k)\) and \((i + 2, k)\)
		have lists identical to \(\List{L}{2}\) for \(k = j - 1, j - 2\).
		
		\item\label{config-II'} The vertices \((i, k)\), \((i + 1, k - 1)\) and \((i + 2, k - 2)\) have lists
		identical to \(\List{L}{1}\) for \(k = j + 2, j + 1\), the vertices
		\((i, k)\), \((i + 1, k)\) and \((i + 2, k - 1)\)
		have lists identical to \(\List{L}{2}\) for \(k = j, j - 1\).
		
		\item\label{config-III'} The vertices \((i, k)\), \((i + 1, k - 1)\) and \((i + 2, k - 1)\) have lists
		identical to \(\List{L}{1}\) for \(k = j + 2, j + 1\), the vertices
		\((i, k)\), \((i + 1, k)\) and \((i + 2, k - 1)\)
		have lists identical to \(\List{L}{2}\) for \(k = j - 1, j - 2\),
		and the vertex \((i + 2, j - 1)\) belongs to an isolated component of some list-class \(G[\List{L}{4}]\),
		where \(\List{L}{1} \neq \List{L}{4}\) and \(\List{L}{2} \neq \List{L}{4}\),
		but \(\List{L}{4}\) may be identical to \(\List{L}{3}\).
	\end{enumerate}
\end{lemma}

\begin{figure}
\centering
\begingroup
  \renewcommand\thesubfigure{\Roman{subfigure}}
  
	\begin{subfigure}{0.3\textwidth}
	\centering
    \setcounter{subfigure}{7}
		\begin{tikzpicture}[font=\scriptsize]
			\draw (1,1) grid (3,6);
			\foreach \x in {1,2}
			\foreach \y in {2,...,6}
			\draw (\x,\y) -- (\x+1,\y-1);
			\draw[dotted]
			(1,0) -- (1,1) -- (2,0) -- (2,1) -- (3,0) -- (3,1)
			(1,6) -- (1,7) -- (2,6) -- (2,7) -- (3,6) -- (3,7);
			
			\draw[fill=myyellow] (1,1) circle(2.5pt) node[anchor=east] {};
			\draw[fill=myblue] (1,2) circle(2.5pt) node[anchor=east] {\(\List{L}{2}\)};
			\draw[fill=myblue] (1,3) circle(2.5pt) node[anchor=east] {\(\List{L}{2}\)};
			\draw[fill=mygreen] (1,4) circle(2.5pt) node[anchor=east] {\(\List{L}{3}\)};
			\draw[fill=myred] (1,5) circle(2.5pt) node[anchor=east] {\(\List{L}{1}\)};
			\draw[fill=myred] (1,6) circle(2.5pt) node[anchor=east] {\(\List{L}{1}\)};
			\draw[fill=myyellow] (2,1) circle(2.5pt) node[anchor=south west] {};
			\draw[fill=myblue] (2,2) circle(2.5pt) node[anchor=south west] {\(\List{L}{2}\)};
			\draw[fill=myblue] (2,3) circle(2.5pt) node[anchor=south west] {\(\List{L}{2}\)};
			\draw[fill=myred] (2,4) circle(2.5pt) node[anchor=south west] {\(\List{L}{1}\)};
			\draw[fill=myred] (2,5) circle(2.5pt) node[anchor=south west] {\(\List{L}{1}\)};
			\draw[fill=myyellow] (2,6) circle(2.5pt) node[anchor=south west] {};
			\draw[fill=myyellow] (3,1) circle(2.5pt) node[anchor=west] {};
			\draw[fill=myblue] (3,2) circle(2.5pt) node[anchor=west] {\(\List{L}{2}\)};
			\draw[fill=myblue] (3,3) circle(2.5pt) node[anchor=west] {\(\List{L}{2}\)};
			\draw[fill=myred] (3,4) circle(2.5pt) node[anchor=west] {\(\List{L}{1}\)};
			\draw[fill=myred] (3,5) circle(2.5pt) node[anchor=west] {\(\List{L}{1}\)};
			\draw[fill=myyellow] (3,6) circle(2.5pt) node[anchor=west] {};
			
			\node[anchor=south] at (1,7) {\(C_{i}\)};
			\node[anchor=south] at (2,7) {\(C_{i + 1}\)};
			\node[anchor=south] at (3,7) {\(C_{i + 2}\)};
		\end{tikzpicture}
		\caption{}
	\end{subfigure}
	\begin{subfigure}{0.3\textwidth}
	\centering
		\begin{tikzpicture}[font=\scriptsize]
			\draw (1,1) grid (3,6);
			\foreach \x in {1,2}
			\foreach \y in {2,...,6}
			\draw (\x,\y) -- (\x+1,\y-1);
			\draw[dotted]
			(1,0) -- (1,1) -- (2,0) -- (2,1) -- (3,0) -- (3,1)
			(1,6) -- (1,7) -- (2,6) -- (2,7) -- (3,6) -- (3,7);
			
			\draw[fill=myyellow] (1,1) circle(2.5pt) node[anchor=east] {};
			\draw[fill=myblue] (1,2) circle(2.5pt) node[anchor=east] {\(\List{L}{2}\)};
			\draw[fill=myblue] (1,3) circle(2.5pt) node[anchor=east] {\(\List{L}{2}\)};
			\draw[fill=mygreen] (1,4) circle(2.5pt) node[anchor=east] {\(\List{L}{3}\)};
			\draw[fill=myred] (1,5) circle(2.5pt) node[anchor=east] {\(\List{L}{1}\)};
			\draw[fill=myred] (1,6) circle(2.5pt) node[anchor=east] {\(\List{L}{1}\)};
			\draw[fill=myyellow] (2,1) circle(2.5pt) node[anchor=south west] {};
			\draw[fill=myblue] (2,2) circle(2.5pt) node[anchor=south west] {\(\List{L}{2}\)};
			\draw[fill=myblue] (2,3) circle(2.5pt) node[anchor=south west] {\(\List{L}{2}\)};
			\draw[fill=myred] (2,4) circle(2.5pt) node[anchor=south west] {\(\List{L}{1}\)};
			\draw[fill=myred] (2,5) circle(2.5pt) node[anchor=south west] {\(\List{L}{1}\)};
			\draw[fill=myyellow] (2,6) circle(2.5pt) node[anchor=south west] {};
			\draw[fill=myblue] (3,1) circle(2.5pt) node[anchor=west] {\(\List{L}{2}\)};
			\draw[fill=myblue] (3,2) circle(2.5pt) node[anchor=west] {\(\List{L}{2}\)};
			\draw[fill=myred] (3,3) circle(2.5pt) node[anchor=west] {\(\List{L}{1}\)};
			\draw[fill=myred] (3,4) circle(2.5pt) node[anchor=west] {\(\List{L}{1}\)};
			\draw[fill=myyellow] (3,5) circle(2.5pt) node[anchor=west] {};
			\draw[fill=myyellow] (3,6) circle(2.5pt) node[anchor=west] {};
			
			\node[anchor=south] at (1,7) {\(C_{i}\)};
			\node[anchor=south] at (2,7) {\(C_{i + 1}\)};
			\node[anchor=south] at (3,7) {\(C_{i + 2}\)};
		\end{tikzpicture}
		\caption{}
	\end{subfigure}
	\begin{subfigure}{0.3\textwidth}
	\centering
		\begin{tikzpicture}[font=\scriptsize]
			\draw (1,1) grid (3,6);
			\foreach \x in {1,2}
			\foreach \y in {2,...,6}
			\draw (\x,\y) -- (\x+1,\y-1);
			\draw[dotted]
			(1,0) -- (1,1) -- (2,0) -- (2,1) -- (3,0) -- (3,1)
			(1,6) -- (1,7) -- (2,6) -- (2,7) -- (3,6) -- (3,7);
			
			\draw[fill=myyellow] (1,1) circle(2.5pt) node[anchor=east] {};
			\draw[fill=myblue] (1,2) circle(2.5pt) node[anchor=east] {\(\List{L}{2}\)};
			\draw[fill=myblue] (1,3) circle(2.5pt) node[anchor=east] {\(\List{L}{2}\)};
			\draw[fill=mygreen] (1,4) circle(2.5pt) node[anchor=east] {\(\List{L}{3}\)};
			\draw[fill=myred] (1,5) circle(2.5pt) node[anchor=east] {\(\List{L}{1}\)};
			\draw[fill=myred] (1,6) circle(2.5pt) node[anchor=east] {\(\List{L}{1}\)};
			\draw[fill=myyellow] (2,1) circle(2.5pt) node[anchor=south west] {};
			\draw[fill=myblue] (2,2) circle(2.5pt) node[anchor=south west] {\(\List{L}{2}\)};
			\draw[fill=myblue] (2,3) circle(2.5pt) node[anchor=south west] {\(\List{L}{2}\)};
			\draw[fill=myred] (2,4) circle(2.5pt) node[anchor=south west] {\(\List{L}{1}\)};
			\draw[fill=myred] (2,5) circle(2.5pt) node[anchor=south west] {\(\List{L}{1}\)};
			\draw[fill=myyellow] (2,6) circle(2.5pt) node[anchor=south west] {};
			\draw[fill=myblue] (3,1) circle(2.5pt) node[anchor=west] {\(\List{L}{2}\)};
			\draw[fill=myblue] (3,2) circle(2.5pt) node[anchor=west] {\(\List{L}{2}\)};
			\draw[fill=mygreen] (3,3) circle(2.5pt) node[anchor=west] {\(\List{L}{4}\)};
			\draw[fill=myred] (3,4) circle(2.5pt) node[anchor=west] {\(\List{L}{1}\)};
			\draw[fill=myred] (3,5) circle(2.5pt) node[anchor=west] {\(\List{L}{1}\)};
			\draw[fill=myyellow] (3,6) circle(2.5pt) node[anchor=west] {};
			
			\node[anchor=south] at (1,7) {\(C_{i}\)};
			\node[anchor=south] at (2,7) {\(C_{i + 1}\)};
			\node[anchor=south] at (3,7) {\(C_{i + 2}\)};
		\end{tikzpicture}
		\caption{}
	\end{subfigure}
	\caption{Illustration of the configurations \ref{config-I'} through \ref{config-III'} of Lemma~\ref{L:with-iso}.}\label{F:iso-config}
\endgroup
\end{figure}
\renewcommand{\thesubfigure}{\alph{subfigure}}

\begin{proof}
	We start with \(\List{L}{(i, j + 1)} = L_{1}\), \(\List{L}{(i, j)} = L_{3}\) and
	\(\List{L}{(i, j - 1)} = L_{2}\), with \((i, j)\) belonging to an isolated
	component of the list-class \(G[L_{3}]\).
	By Lemma~\ref{LI:noniso-config}, we have \(\List{L}{(i, j + 2)} = L_{1}\)
	and \(\List{L}{(i, j - 2)} = L_{2}\). By Lemma~\ref{LI:iso-config}, we have
	\(\List{L}{(i + 1, j + 1)} = \List{L}{(i + 1, j)} = L_{1}\) and
	\(\List{L}{(i + 1, j - 1)} = L_{2}\). 
	By Lemma~\ref{LI:noniso-config}, we also have \(\List{L}{(i + 1, j - 2)} = L_{2}\).
	By criterion~\ref{criterion2.1}, this implies that
	\(\List{L}{(i + 2, j)} = L_{1}\) and \(\List{L}{(i + 2, j - 2)} = L_{2}\).
	Now, again by Lemma~\ref{LI:noniso-config}, we have the following three cases:
	\begin{itemize}
		\item \(\List{L}{(i + 2, j)} = L_{1}\) and \(\List{L}{(i + 2, j - 2)} = L_{2}\). This
		is configuration~\ref{config-I'}.
		
		\item \(\List{L}{(i + 2, j - 1)} = L_{1}\) and \(\List{L}{(i + 2, j - 3)} = L_{2}\). This
		is configuration~\ref{config-II'}.
		
		\item \(\List{L}{(i + 2, j + 1)} = L_{1}\), \(\List{L}{(i + 2, j - 3)} = L_{2}\) and
		\(\List{L}{(i + 2, j - 1)} = L_{4}\) where \(L_{1} \neq L_{4}\) and \(L_{2} \neq L_{4}\). In particular,
		by Lemma~\ref{LI:noniso-config}, \((i + 2, j - 1)\) must belong to an isolated component of the list-class \(G[L_{4}]\). This
		is configuration~\ref{config-III'}. Note that \(L_{4}\) may be identical to \(L_{3}\).
	\end{itemize}
\end{proof}

These configurations are listed in Figures~\ref{F:noniso-config1}--\ref{F:iso-config}.

We are now in a position to complete the proof of case~\ref{l1} in Theorem~\ref{T:Main}.

\section{Proof of \texorpdfstring{case~\ref{l1}}{case (1)} in \texorpdfstring{Theorem~\ref{T:Main}}{Theorem 1}}\label{S:Main}

By the results in Section~\ref{S:Preparation}, it suffices to assume that the list assignment
\(\ListAsgn{L}\) on \(G\) satisfies criteria~\ref{criterion2.1} to~\ref{criterion2.4},
and that, in particular, Lemma~\ref{L:with-iso} holds.
Suppose \((i, j + 1)\), \((i, j)\) and \((i, j - 1)\) are three vertices
in the column \(C_{i}\) that satisfy the hypotheses of Lemma~\ref{L:with-iso}.
Then, the vertices \((i + 1, j)\) and \((i + 1, j - 1)\) in the column \(C_{i + 1}\),
as well as the vertices \((i - 1, j + 1)\) and \((i - 1, j)\) in the column \(C_{i - 1}\),
satisfy the hypotheses of Lemma~\ref{L:without-iso}.
Thus, there always exists a column that has
a pair of adjacent vertices that satisfies the hypotheses of Lemma~\ref{L:without-iso},
which we shall now take to be \(C_{1}\) without loss of generality. Furthermore, without loss
of generality, let \((1, s)\) and \((1, s - 1)\) satisfy the hypotheses of Lemma~\ref{L:without-iso}.

Now, the first step of our algorithm to find an \(\ListAsgn{L}\)-coloring---which we elaborate
on below---is to properly color \(C_{1}\).
Then, the lists on \(C_{r}\) all reduce to \(3\)-lists, so \(C_{r}\) can be properly
colored by Lemma~\ref{L:N-algorithm}. This in turn causes the lists on \(C_{r - 1}\) to reduce
to \(3\)-lists. Thus, we can inductively color the columns from the right using Lemma~\ref{L:N-algorithm}
until we are only left to color the columns \(C_{2}\), \(C_{3}\) and \(C_{4}\).
Thus, it suffices to assume without loss of generality that \(r = 4\).

Fix \(1 \leq j \leq s\). We start with a few straightforward observations:
\begin{enumerate}[(O\arabic*)]
	\item\label{ob1} If \(\List{L}{(1, j)} = \List{L}{(2, j)} = \List{L}{(2, j - 1)}\),
	then any choice of color for \((1, j)\) will reduce the sizes of
	\(\List{L}{(2, j)}\) and \(\List{L}{(2, j - 1)}\) by \(1\) each.
	Similarly, if \(\List{L}{(1, j)} = \List{L}{(1, j - 1)} = \List{L}{(2, j - 1)}\),
	then any proper coloring of \((1, j)\) and \((1, j - 1)\)
	will reduce the size of \(\List{L}{(2, j - 1)}\) by \(2\).
	
	\item\label{ob2} Consider the vertices \((1, j)\), \((2, j)\), \((2, j - 1)\) and \((3, j - 1)\).
	Suppose that a color \(c \in \List{L}{(1, j)}\) has been chosen for \((1, j)\),
	so the sizes of \(\List{L}{(2, j)}\) and \(\List{L}{(2, j - 1)}\) have potentially reduced
	by \(1\) each. Now, if \(c \in \List{L}{(3, j - 1)}\) too,
	then coloring \((3, j - 1)\) with the color \(c\) does not reduce
	the sizes of the residual lists on \((2, j)\) and \((2, j - 1)\) any
	further.
	
	\item\label{ob3} Suppose that \(\List{L}{(1, j)} \cap \List{L}{(3, j - 1)} = \emptyset\).
	If \((1, j)\) is an isolated vertex, then \((1, j + 1)\), \((1, j)\), and \((1, j - 1)\)
	satisfy the hypotheses of Lemma~\ref{L:with-iso}, and moreover these vertices must be in
	configuration~\ref{config-III'}. If \((1, j)\) is not an isolated vertex, then
	either \((1, j)\) and \((1, j - 1)\) satisfy the hypotheses of Lemma~\ref{L:without-iso},
	or \((1, j + 1)\) and \((1, j)\) satisfy the hypotheses of Lemma~\ref{L:without-iso},
	or \(\List{L}{(1, j + 1)} = \List{L}{(1, j)} = \List{L}{(1, j - 1)}\).
	If the first case holds, then \((1, j)\) and \((1, j - 1)\) must be in
	configuration~\ref{config-I}; if the second case holds, then \((1, j + 1)\) and \((1, j)\) must be in configuration~\ref{config-IV};
	the third case is impossible, since by repeated application of criterion~\ref{criterion2.1} we must have
	\(\List{L}{(1, j)} = \List{L}{(3, j - 1)}\).
	
	Furthermore, in the case when \((1, j)\) is an isolated vertex,
	choosing a color for \((1, j)\) from \(\List{L}{(1, j)} \setminus \List{L}{(2, j)}\)
	and for \((3, j - 1)\) from \(\List{L}{(3, j - 1)} \setminus \List{L}{(2, j - 1)}\)
	will reduce the sizes of \(\List{L}{(2, j)}\) and \(\List{L}{(2, j - 1)}\) by \(1\) each.
	Note that such choices are possible by criterion~\ref{criterion2.1} and Lemma~\ref{LI:iso-config}.
	In the other cases, any choice of color for \((1, j)\) and for \((3, j - 1)\)
	will reduce the sizes of \(\List{L}{(2, j)}\) and \(\List{L}{(2, j - 1)}\) only by \(1\) each.
	
\end{enumerate}

These observations are crucial for step~\ref{step1} of the following two-step coloring algorithm:
\begin{enumerate}[label=\arabic*.,ref=\arabic*]
	\item\label{step1} Properly color \(C_{1}\) and a set \(J\) of alternate vertices
	in \(C_{3}\) such that the reduced list sizes on \(C_{2}\) are as follows:
	one vertex in \(C_{2}\) has a \(4\)-list and every other vertex in \(C_{2}\) has a \(3\)-list.
	
	\item\label{step2} Properly color \(C_{4}\), then the remaining vertices in \(C_{3}\), and finally \(C_{2}\).
\end{enumerate}

Assume for the moment that step~\ref{step1} has been completed.
Then, step~\ref{step2} can be completed by repeatedly invoking
Lemma~\ref{L:cycle-choosability} as follows.

As we shall see when we elaborate on step~\ref{step1},
we may assume that the \(4\)-list in the column \(C_{2}\) is on the vertex \((2, s - 1)\),
and that the rest of the vertices in \(C_{2}\) have \(3\)-lists.
Also, the set \(J\) will turn out to be either \(I \defn \Set{ (3, s - 2k + 2) : k = 1,\dotsc,\floor{s/2} }\)
or \(I' \defn \Set{ (3, s - 2k + 1) : k = 1,\dotsc,\floor{s/2} }\).

Now, the sizes of the lists on the remaining vertices of \(C_{3}\) after the completion
of step~\ref{step1} are as follows:
the vertices of \(C_{3}\)
that remain to be colored all have \(3\)-lists; moreover, when \(s\) is odd,
the vertices \((3, 1)\) and \((3, 2)\) each have a \(4\)-list when the set \(I\)
is colored in step~\ref{step1}, and the vertices \((3, 1)\) and \((3, s)\) each have a \(4\)-list when
the set \(I'\) is colored in step~\ref{step1}.

Next, the sizes of the lists on the column \(C_{4}\) are as follows:
each vertex in \(C_{4}\) has a \(2\)-list; moreover, when \(s\) is odd,
the vertex \((4, 1)\) has \(3\)-list when the set \(I\) is colored in step~\ref{step1},
and the vertex \((4, s)\) has a \(3\)-list when the set \(I'\) is colored in step~\ref{step1}.

Thus, regardless of the parity of \(s\),
properly color the column \(C_{4}\) using Lemma~\ref{L:cycle-choosability}.
This reduces the sizes of each of the remaining lists on \(C_{3}\)
by \(2\). Again by Lemma~\ref{L:cycle-choosability}, regardless of the parity of \(s\),
properly color the remaining vertices in the column \(C_{3}\).
This reduces the list sizes on \(C_{2}\) as follows. When \(s\) is even, each list
on \(C_{2}\) is reduced in size by \(1\). When \(s\) is odd,
each list is reduced in size by \(1\), but for the following exception:
if \(I\) is colored in step~\ref{step1}, then the list on \((2, 2)\) is reduced in size by \(2\),
and if \(I'\) is colored in step~\ref{step1}, then the list on \((2, 1)\) is reduced in size by \(2\).
In either case, properly color \(C_{2}\) using Lemma~\ref{L:cycle-choosability}.
This completes step~\ref{step2}.

Figures~\ref{F:example1} and \ref{F:example2} illustrate the sizes of the lists in step~\ref{step2}
when \(s\) is even and odd, respectively, assuming that the set \(I\) is colored
in step~\ref{step1}. The edges between the top and bottom rows are not shown in these figures.

\begin{figure}
\centering
	\begin{subfigure}{0.22\textwidth}
	\centering
		\begin{tikzpicture}[font=\scriptsize]
			\draw (1,1) grid (3,6);
			\foreach \y in {2,...,6}
			\foreach \x in {1,2}
			\draw (\x,\y) -- (\x+1,\y-1);
			\foreach \x in {1,2,3}
			\foreach \y in {1,...,6}
			\draw[fill=black] (\x,\y) circle(2.5pt);
			\foreach \y in {1,2,3,4,6}{
				\node[anchor=east] at (1,\y) {3};
			}
			\foreach \y in {1,...,6}{
				\node[anchor=west] at (3,\y) {3};
			}
			\node[anchor=east] at (1,5) {4};
			
			\node[anchor=north] at (2,1) {5};
			\node[anchor=south] at (2,6) {5};
			\foreach \y in {2,...,5}
			\node[anchor=south west] at (2,\y) {5};
		\end{tikzpicture}
		\caption{}
	\end{subfigure}
	\begin{subfigure}{0.22\textwidth}
	\centering
		\begin{tikzpicture}[font=\scriptsize]
			\draw
			(1,1) -- (1,2) -- (1,3) -- (1,4) -- (1,5) -- (1,6)
			(3,1) -- (3,2) -- (3,3) -- (3,4) -- (3,5) -- (3,6)
			(1,6) -- (2,5) -- (3,5)
			(1,5) -- (2,5) -- (3,4)
			(1,4) -- (2,3) -- (3,3)
			(1,3) -- (2,3) -- (3,2)
			(1,2) -- (2,1) -- (3,1)
			(1,1) -- (2,1);
			\draw[dotted, line width=1pt]
			(2,1) -- (2,2) -- (2,3) -- (2,4) -- (2,5) -- (2,6)
			(2,6) -- (3,6)
			(1,6) -- (2,6) -- (3,5)
			(1,5) -- (2,4) -- (3,4)
			(1,4) -- (2,4) -- (3,3)
			(1,3) -- (2,2) -- (3,2)
			(1,2) -- (2,2) -- (3,1);
			\foreach \x in {1,2,3}
			\foreach \y in {1,...,6}
			\draw[fill=black] (\x,\y) circle(2.5pt);
			\foreach \y in {1,2,3,4,6}{
				\node[anchor=east] at (1,\y) {3};
			}
			\foreach \y in {1,...,6}{
				\node[anchor=west] at (3,\y) {2};
			}
			\node[anchor=east] at (1,5) {4};
			
			\node[anchor=north] at (2,1) {3};
			\foreach \y in {3,5}
			\node[anchor=south west] at (2,\y) {3};
		\end{tikzpicture}
		\caption{}
	\end{subfigure}
	\begin{subfigure}{0.22\textwidth}
	\centering
		\begin{tikzpicture}[font=\scriptsize]
			\draw
			(1,1) -- (1,2) -- (1,3) -- (1,4) -- (1,5) -- (1,6)
			(1,6) -- (2,5)
			(1,5) -- (2,5)
			(1,4) -- (2,3)
			(1,3) -- (2,3)
			(1,2) -- (2,1)
			(1,1) -- (2,1);
			\draw[dotted, line width=1pt]
			(3,1) -- (3,2) -- (3,3) -- (3,4) -- (3,5) -- (3,6)
			(3,4) -- (2,5) -- (3,5)
			(3,2) -- (2,3) -- (3,3)
			(2,1) -- (3,1);
			\foreach \x in {1,2,3}
			\foreach \y in {1,...,6}
			\draw[fill=black] (\x,\y) circle(2.5pt);
			\foreach \y in {1,2,3,4,6}{
				\node[anchor=east] at (1,\y) {3};
			}
			\node[anchor=east] at (1,5) {4};
			
			\node[anchor=north] at (2,1) {1};
			\foreach \y in {3,5}
			\node[anchor=south west] at (2,\y) {1};
		\end{tikzpicture}
		\caption{}
	\end{subfigure}
	\begin{subfigure}{0.22\textwidth}
	\centering
		\begin{tikzpicture}[font=\scriptsize]
			\draw
			(1,1) -- (1,2) -- (1,3) -- (1,4) -- (1,5) -- (1,6);
			\draw[dotted, line width=1pt]
			(1,6) -- (2,5)
			(1,5) -- (2,5)
			(1,4) -- (2,3)
			(1,3) -- (2,3)
			(1,2) -- (2,1)
			(1,1) -- (2,1);
			\foreach \x in {1,2,3}
			\foreach \y in {1,...,6}
			\draw[fill=black] (\x,\y) circle(2.5pt);
			\foreach \y in {1,2,3,4,6}{
				\node[anchor=east] at (1,\y) {2};
			}
			\node[anchor=east] at (1,5) {3};
			\node[anchor=north] at (2,1) {\vphantom{5}};
		\end{tikzpicture}
		\caption{}
	\end{subfigure}
	\caption{Illustration of the sizes of the lists on the columns
		\(C_{2}\), \(C_{3}\) and \(C_{4}\) in step~\ref{step2} when \(s = 6\)
		and \(I\) is colored in step~\ref{step1}.}\label{F:example1}
\end{figure}

\begin{figure}
\centering
	\begin{subfigure}{0.22\textwidth}
	\centering
		\begin{tikzpicture}[font=\scriptsize]
			\draw (1,0) grid (3,6);
			\foreach \y in {1,...,6}
			\foreach \x in {1,2}
			\draw (\x,\y) -- (\x+1,\y-1);
			\foreach \x in {1,2,3}
			\foreach \y in {0,...,6}
			\draw[fill=black] (\x,\y) circle(2.5pt);
			\foreach \y in {0,1,2,3,4,6}{
				\node[anchor=east] at (1,\y) {3};
			}
			\foreach \y in {0,...,6}{
				\node[anchor=west] at (3,\y) {3};
			}
			\node[anchor=east] at (1,5) {4};
			
			\node[anchor=north] at (2,0) {5};
			\node[anchor=south] at (2,6) {5};
			\foreach \y in {1,...,5}
			\node[anchor=south west] at (2,\y) {5};
		\end{tikzpicture}
		\caption{}
	\end{subfigure}
	\begin{subfigure}{0.22\textwidth}
	\centering
		\begin{tikzpicture}[font=\scriptsize]
			\draw
			(1,0) -- (1,1) -- (1,2) -- (1,3) -- (1,4) -- (1,5) -- (1,6)
			(3,0) -- (3,1) -- (3,2) -- (3,3) -- (3,4) -- (3,5) -- (3,6)
			(1,6) -- (2,5) -- (3,5)
			(1,5) -- (2,5) -- (3,4)
			(1,4) -- (2,3) -- (3,3)
			(1,3) -- (2,3) -- (3,2)
			(1,2) -- (2,1) -- (3,1)
			(1,1) -- (2,1) -- (3,0)
			(1,1) -- (2,0) -- (3,0)
			(1,0) -- (2,0) -- (2,1);
			\draw[dotted, line width=1pt]
			(2,1) -- (2,2) -- (2,3) -- (2,4) -- (2,5) -- (2,6)
			(2,6) -- (3,6)
			(1,6) -- (2,6) -- (3,5)
			(1,5) -- (2,4) -- (3,4)
			(1,4) -- (2,4) -- (3,3)
			(1,3) -- (2,2) -- (3,2)
			(1,2) -- (2,2) -- (3,1);
			\foreach \x in {1,2,3}
			\foreach \y in {0,...,6}
			\draw[fill=black] (\x,\y) circle(2.5pt);
			\foreach \y in {0,1,2,3,4,6}{
				\node[anchor=east] at (1,\y) {3};
			}
			\foreach \y in {1,...,6}{
				\node[anchor=west] at (3,\y) {2};
			}
			\node[anchor=east] at (1,5) {4};
			
			\node[anchor=west] at (3,0) {3};		
			\node[anchor=south west] at (2,1) {4};	
			\node[anchor=north] at (2,0) {4};
			\foreach \y in {3,5}
			\node[anchor=south west] at (2,\y) {3};
		\end{tikzpicture}
		\caption{}
	\end{subfigure}
	\begin{subfigure}{0.22\textwidth}
	\centering
		\begin{tikzpicture}[font=\scriptsize]
			\draw
			(1,0) -- (1,1) -- (1,2) -- (1,3) -- (1,4) -- (1,5) -- (1,6)
			(1,6) -- (2,5)
			(1,5) -- (2,5)
			(1,4) -- (2,3)
			(1,3) -- (2,3)
			(1,2) -- (2,1)
			(1,1) -- (2,1)
			(1,0) -- (2,0) -- (1,1)
			(2,0) -- (2,1);
			\draw[dotted, line width=1pt]
			(3,0) -- (3,1) -- (3,2) -- (3,3) -- (3,4) -- (3,5) -- (3,6)
			(3,4) -- (2,5) -- (3,5)
			(3,2) -- (2,3) -- (3,3)
			(2,0) -- (3,0) -- (2,1) -- (3,1);
			\foreach \x in {1,2,3}
			\foreach \y in {0,...,6}
			\draw[fill=black] (\x,\y) circle(2.5pt);
			\foreach \y in {0,1,2,3,4,6}{
				\node[anchor=east] at (1,\y) {3};
			}
			\node[anchor=east] at (1,5) {4};
			
			\node[anchor=north] at (2,0) {2};
			\node[anchor=south west] at (2,1) {2};
			\foreach \y in {3,5}
			\node[anchor=south west] at (2,\y) {1};
		\end{tikzpicture}
		\caption{}
	\end{subfigure}
	\begin{subfigure}{0.22\textwidth}
	\centering
		\begin{tikzpicture}[font=\scriptsize]
			\draw
			(1,0) -- (1,1) -- (1,2) -- (1,3) -- (1,4) -- (1,5) -- (1,6);
			\draw[dotted, line width=1pt]
			(1,6) -- (2,5)
			(1,5) -- (2,5)
			(1,4) -- (2,3)
			(1,3) -- (2,3)
			(1,2) -- (2,1)
			(1,1) -- (2,1)
			(1,1) -- (2,1)
			(1,0) -- (2,0) -- (1,1)
			(2,0) -- (2,1);
			\foreach \x in {1,2,3}
			\foreach \y in {0,...,6}
			\draw[fill=black] (\x,\y) circle(2.5pt);
			\foreach \y in {0,2,3,4,6}{
				\node[anchor=east] at (1,\y) {2};
			}
			\node[anchor=east] at (1,1) {1};
			\node[anchor=east] at (1,5) {3};
			\node[anchor=north] at (2,0) {\vphantom{5}};
		\end{tikzpicture}
		\caption{}
	\end{subfigure}
	\caption{Illustration of the sizes of the lists on the columns
		\(C_{2}\), \(C_{3}\) and \(C_{4}\) in step~\ref{step2} when \(s = 7\)
		and \(I\) is colored in step~\ref{step1}.}\label{F:example2}
\end{figure}

We now describe step~\ref{step1}.
If \((1, s)\) and \((1, s - 1)\)
are in any configuration other than \ref{config-IV} and~\ref{config-VI},
then take \(J = I\), and if \((1, s)\) and \((1, s - 1)\) are in configuration~\ref{config-IV} or~\ref{config-VI},
then take \(J = I'\), where the sets \(I\) and \(I'\) are as defined earlier in the description of step~\ref{step2}.
From observations~\ref{ob1} to~\ref{ob3}, every vertex in \(C_{2}\) can have a \(3\)-list
at the end of step~\ref{step1} if for every \((3, j) \in J\), either
\(\List{L}{(1, j + 1)} \cap \List{L}{(3, j)} = \emptyset\),
or \((1, j + 1)\) and \((3, j)\) are assigned the same color.
Clearly, if the lists on \((1, j + 1)\) and \((3, j)\) are identical, then for
any assignment of a color on \((1, j + 1)\) we can pick the same color
for \((3, j)\). On the other hand, if the lists on \((1, j + 1)\) and \((3, j)\) are distinct
but not disjoint, then we need to ensure that the color assigned on \((1, j + 1)\)
belongs to \(\List{L}{(1, j + 1)} \cap \List{L}{(3, j)}\).

So, call the pair of vertices \((1, j + 1), (3, j)\) to be a \defining{good pair}
if either \(\List{L}{(1, j + 1)} = \List{L}{(3, j)}\), or
\(\List{L}{(1, j + 1)} \cap \List{L}{(3, j)} = \emptyset\) and
\((1, j + 2)\), \((1, j + 1)\), and \((1, j)\) are not in configuration~\ref{config-III'}.
Define \(A\) to be the set of all pairs \((1, j + 1), (3, j)\)
that are not good pairs.
We now carry out step~\ref{step1} in three stages.
In the first stage, we shall color the vertices in \(A\).
In the second stage, we color the vertices \((1, s)\)
and \((1, s - 1)\) in such a way that the list on \((2, s - 1)\) reduces to a \(4\)-list.
Finally, we color the remaining vertices
of the column \(C_{1}\), followed by the remaining vertices in \(J\).

Now, for the first stage. Suppose \((3, j) \in A\). If the lists on \((1, j + 1)\)
and \((3, j)\) are distinct but not disjoint, then choose a common color for \((1, j + 1)\)
and \((3, j)\) from \(\List{L}{(1, j + 1)} \cap \List{L}{(3, j)}\).
Otherwise, we have that the lists on
\((1, j + 1)\) and \((3, j)\) are disjoint and 
the vertices \((1, j + 2)\), \((1, j + 1)\) and \((1, j)\) are in configuration~\ref{config-III'}.
In this case, choose a color for \((1, j + 1)\) from \(\List{L}{(1, j + 1)} \setminus \List{L}{(2, j + 1)}\)
and for \((3, j)\) from \(\List{L}{(3, j)} \setminus \List{L}{(2, j)}\).

Note that our choice of \(J\) ensures that the vertices \((1, 1)\), \((1, s)\), \((1, s - 1)\)
and \((1, s - 2)\) are not colored in the first stage above (cf.~Figure~\ref{F:noniso-config1} and \ref{F:noniso-config2}).
So, for the second stage, pick colors for \((1, s)\) and \((1, s - 1)\) as follows:
\begin{itemize}
	\item If \((1, s)\) and \((1, s - 1)\) are in any of the configurations~\ref{config-I} to~\ref{config-IV},
	then choose a color for \((1, s)\)
	from \(\List{L}{(1, s)} \setminus \List{L}{(1, s - 1)}\)
	and for \((1, s - 1)\) from \(\List{L}{(1, s - 1)} \setminus \List{L}{(1, s)}\).
	
	\item If \((1, s)\) and \((1, s - 1)\) are in configuration~\ref{config-V}, choose a color for \((1, s)\) from
	\(\List{L}{(1, s)} \cap \List{L}{(3, s - 1)} \setminus \List{L}{(1, s - 1)}\)
	(this can be done because of criterion~\ref{criterion2.1} and Lemma~\ref{LI:iso-config}), and choose a color
	for \((1, s - 1)\) from \(\List{L}{(1, s - 1)} \setminus \List{L}{(1, s)}\).
	
	\item If \((1, s)\) and \((1, s - 1)\) are in configuration~\ref{config-VI}, choose a color
	for \((1, s)\) from \(\List{L}{(1, s)} \setminus \List{L}{(1, s - 1)}\),
	and for \((1, s - 1)\) from \(\List{L}{(1, s - 1)} \cap \List{L}{(3, s - 2)} \setminus \List{L}{(1, s)}\)
	(this can be done because of criterion~\ref{criterion2.1} and Lemma~\ref{LI:iso-config}).
	
	\item If \((1, s)\) and \((1, s - 1)\) are in configuration~\ref{config-VII}, choose a color for \((1, s)\)
	from \(\List{L}{(1, s)} \setminus \List{L}{(2, s - 1)}\), and for
	\((1, s - 1)\) from \(\List{L}{(1, s - 1)} \setminus \List{L}{(1, s)}\).
\end{itemize}
This coloring ensures that the vertex \((2, s - 1)\) now has a \(4\)-list.

Finally, for the third stage. Color the remaining vertices in the column \(C_{1}\) using Lemma~\ref{L:cycle-choosability}.
Then, color the remaining vertices in \(J\) as follows.
Suppose \((3, j) \in J\) was uncolored in the first stage.
If the lists on \((3, j)\) and \((1, j + 1)\) assigned by \(\ListAsgn{L}\)
were identical, choose the same color on \((3, j)\)
as that assigned on \((1, j + 1)\). If the lists on \((3, j)\) and \((1, j + 1)\) are disjoint,
then choose any color for \((3, j)\) from its list.

Notice that at the end of this procedure
the vertex \((2, s - 1)\) still has a \(4\)-list, and that all the other vertices in the column
\(C_{2}\) have \(3\)-lists. So, this completes step~\ref{step1}.
Combined with step~\ref{step2}, this completes the proof.

It is also clear from the above description that the coloring can be found in linear time.

\section{Proofs of \texorpdfstring{cases~\ref{l2} and~\ref{l3}}{cases (2) and (3)} in \texorpdfstring{Theorem~\ref{T:Main}}{Theorem 1}}\label{S:Sub}

\begin{proof}[Proof of case~\ref{l2} in Theorem~\ref{T:Main}]
	Let \(G = T(1, s, 2)\) for \(s \geq 9\), \(s \neq 11\). Then, every four successive vertices
	\((1, j)\), \((1, j + 1)\), \((1, j + 2)\), \((1, j + 3)\) induce a \(K_{4}\).
	Suppose that \(\ListAsgn{L}\) is a list assignment on \(G\)
	with lists of size equal to \(5\). Since \(G\) is \(5\)-colorable
	in linear time by the results in~\cite{CollinsHutchinson1999,Sankarnarayanan2022},
	it suffices to assume that not all the lists assigned by \(\ListAsgn{L}\) are identical.
	Without loss of generality, suppose that \(\List{L}{(1, 1)} \neq \List{L}{(1, s)}\).
	Choose a color for \((1, s)\) from \(\List{L}{(1, s)} \setminus \List{L}{(1, 1)}\).
	Next, one can properly color the vertices \((1, s - 1), (1, s - 2), \dotsc, (1, 7)\)
	in that order by successively picking a color for each vertex from its (reduced) list.
	Then, the lists on the remaining vertices are as follows: \((1, 6)\) has
	a \(2\)-list; \((1, 1)\), \((1, 2)\) and \((1, 5)\) have \(3\)-lists;
	\((1, 3)\) and \((1, 4)\) have \(4\)-lists. There are two special cases
	that can be easily dealt with.
	
	\begin{itemize}[leftmargin=*,widest=Case II:]
		\item[Case I:] \(\List{L}{(1, 2)} \cap \List{L}{(1, 6)} \neq \emptyset\).
		
		Choose a common color for \((1, 2)\) and \((1, 6)\) from \(\List{L}{(1, 2)} \cap \List{L}{(1, 6)}\).
		Then, \((1, 1)\) and \((1, 5)\) have \(2\)-lists, and
		\((1, 3)\) and \((1, 4)\) have \(3\)-lists. If we can pick a color for \((1, 1)\)
		that does not belong to both \(\List{L}{(1, 3)}\) and \(\List{L}{(1, 4)}\), then
		we will be done by Lemma~\ref{L:cycle-choosability}, so assume that
		\(\List{L}{(1, 1)} \subset \List{L}{(1, 3)} \cap \List{L}{(1, 4)}\).
		Then, for any choice of color for \((1, 1)\), the remaining \(3\)-cycle
		will have \(2\)-lists, so it will have a proper coloring only when the \(2\)-lists
		are not identical, by Lemma~\ref{L:cycle-choosability}. But, if picking \(a \in \List{L}{(1, 1)}\) results
		in identical \(2\)-lists being present on the remaining \(3\)-cycle, then
		we instead pick the other color \(a' \in \List{L}{(1, 1)} \setminus \Set{a}\)
		for \((1, 1)\) to get non-identical \(2\)-lists on the remaining \(3\)-cycle.
	\end{itemize}
	
	So, it suffices to assume that \(\List{L}{(1, 2)} \cap \List{L}{(1, 6)} = \emptyset\).
	
	\begin{itemize}[leftmargin=*,widest=Case II:]
		\item[Case II:]	\(\List{L}{(1, 1)} \cap \List{L}{(1, 5)} \neq \emptyset\).
		
		Choose a common color \(c\) for \((1, 1)\) and \((1, 5)\) from \(\List{L}{(1, 1)} \cap \List{L}{(1, 5)}\).
		If \(c \not\in \List{L}{(1, 2)}\), then the remaining vertices form a \(K_{4}^{-}\)
		with three \(3\)-lists and one \(1\)-list, and one can see that a proper coloring
		can always be found from this configuration of lists. If \(c \in \List{L}{(1, 2)}\),
		then we have a \(K_{4}^{-}\) with three \(2\)-lists and one \(3\)-list. Since
		\(\List{L}{(1, 2)} \cap \List{L}{(1, 6)} = \emptyset\) by assumption, we can choose
		a color for either \((1, 2)\) or \((1, 6)\) that does not belong to \(\List{L}{(1, 3)}\).
		Then, we can properly color the rest of the vertices using Lemma~\ref{L:cycle-choosability}.
	\end{itemize}
	
	So, we additionally assume that \(\List{L}{(1, 1)} \cap \List{L}{(1, 5)} = \emptyset\).
	
	Now, choose a color for \((1, 3)\) from \(\List{L}{(1, 3)} \setminus \List{L}{(1, 6)}\).
	Then, the lists are now as follows: \((1, 4)\) has a \(3\)-list; \((1, 2)\) and \((1, 6)\)
	have \(2\)-lists; lastly, either \((1, 1)\) has a \(2\)-list and \((1, 5)\) has a \(3\)-list,
	or vice-versa, since the color chosen for \((1, 3)\) can belong to at most one of \(\List{L}{(1, 1)}\)
	and \(\List{L}{(1, 5)}\). In either case, a color can be chosen for \((1, 4)\) such that
	both \((1, 1)\) and \((1, 5)\) end up with \(2\)-lists. The lists on \((1, 2)\) and \((1, 6)\)
	are now a \(1\)-list and a \(2\)-list, not necessarily in that order, since the color
	chosen for \((1, 4)\) can belong to at most one of \(\List{L}{(1, 2)}\) and \(\List{L}{(1, 6)}\).
	In any case, the remaining four vertices form a path graph with three \(2\)-lists and one
	\(1\)-list, so a proper coloring can be found using Lemma~\ref{L:cycle-choosability}.
	
	This completes the proof. It is also clear that the list coloring
	can be found in linear time.
\end{proof}

The following lemma will be repeatedly invoked in the proof of case~\ref{l3} in Theorem~\ref{T:Main}.

\begin{lemma}\label{L:small}
	Let \(G = K_{4}^{-}\) be the complete graph on four vertices with an edge removed,
	where \(V(G) = \Set{a, b, x, y}\) and \(\Set{x, y}\) is an independent set.
	Suppose that \(\ListAsgn{L}\) is a list assignment on \(G\) such that
	\(\card{\List{L}{a}} + \card{\List{L}{b}} = \card{\List{L}{x}} + \card{\List{L}{y}}\).
	Then, one can choose colors for \(x\) and \(y\) such that the sizes of the lists on \(a\) and \(b\)
	reduce by \(1\) each.
\end{lemma}
\begin{proof}
	If \(\List{L}{x} \cap \List{L}{y} \neq \emptyset\), then choose a common color
	for \(x\) and \(y\) from \(\List{L}{x} \cap \List{L}{y}\). Clearly, this reduces
	the sizes of the lists on \(a\) and \(b\) by \(1\) each.
	
	So, suppose \(\List{L}{x} \cap \List{L}{y} = \emptyset\). If \(\List{L}{x} \not\subset \List{L}{a} \cup \List{L}{b}\),
	then we can choose a color for \(x\) from \(\List{L}{x} \setminus (\List{L}{a} \cup \List{L}{b})\),
	and any color for \(y\), to reduce the sizes of the lists on \(a\) and \(b\)
	by \(1\) each. Similarly, when \(\List{L}{y} \not\subset \List{L}{a} \cup \List{L}{b}\).
	
	So, suppose that \(\List{L}{x}, \List{L}{y} \subset \List{L}{a} \cup \List{L}{b}\).
	But then there are \(k\) distinct
	available colors for \(x\) and \(y\) together, where
	\(k = \card{\List{L}{x}} + \card{\List{L}{y}}\),
	as well as for \(a\) and \(b\) together. Thus, any color in \(\List{L}{x}\) can belong to at most one of
	\(\List{L}{a}\) and \(\List{L}{b}\), and similarly for any color in \(\List{L}{y}\).
	Moreover, it is not possible that
	all of the \(k\) available colors on \(x\) and \(y\) belong
	to a single list (say, \(\List{L}{a}\)), because the other list
	(\(\List{L}{b}\)) will then be empty, a contradiction. Therefore,
	it is possible to choose colors for \(x\) and \(y\) from their respective lists
	such that each color belongs to a different list between \(\List{L}{a}\) and \(\List{L}{b}\).	
	This reduces the sizes of the lists on \(a\) and \(b\) by \(1\) each.
\end{proof}

\begin{proof}[Proof of case~\ref{l3} in Theorem~\ref{T:Main}]
	Let \(G = T(2, s, t)\) for even \(s \geq 4\) and even \(t \neq 0\), \(s - 2\).
	By the remarks following Theorem~\ref{T:Altshuler}, it follows that the graphs
	\(T(2, s, t)\) and \(T(2, s, s - t - 2)\) are isomorphic,
	so without loss of generality we assume that \(2 \leq t \leq \frac{s}{2} - 1\).
	
	Our strategy is to properly color the column \(C_{2}\) in such a way that
	the lists on the column \(C_{1}\) are all reduced to \(2\)-lists, so that
	\(C_{1}\) can then be properly colored using Lemma~\ref{L:cycle-choosability}.
	Note that, for every \(j\), the vertices \((1, j)\), \((1, j - 1)\), \((2, j - 1)\)
	and \((2, j + t)\) form a \(K_{4}^{-}\) with the vertices on the column \(C_{2}\)
	forming an independent set. This suggests the following scheme of coloring.
	
	Fix \(1 \leq j \leq s\). Using Lemma~\ref{L:small}, we can color
	\((2, j - 1)\) and \((2, j + t)\) such that the lists on \((1, j)\)
	and \((1, j - 1)\) reduce in size by \(1\) each.
	Then, regardless of how the rest of the neighbors of \((1, j)\) and of \((1, j - 1)\)
	in the column \(C_{2}\) are colored, the end result is that the lists on these two vertices reduce
	to \(2\)-lists, as required.
	If we can do this for each even \(j\), then we will have colored \(C_{2}\) in such a way that
	the lists on \(C_{1}\) all reduce to \(2\)-lists. A little bit of care is required
	to ensure that this can be done for every even \(j\), while also ensuring
	that the coloring on \(C_{2}\) is proper. The details now follow.

	Suppose that \(t < \frac{s}{2} - 1\).
	\begin{itemize}
		\item
		For each \(j \in \Set{2, 4, \dotsc, t}\), we may use
		Lemma~\ref{L:small} to color the vertices of the form \((2, j - 1)\) and \((2, j + t)\).
		
		\item
		For \(j = t + 2\), we need to color \((2, t + 1)\) and \((2, 2t + 2)\),
		but notice that the list on \((2, t + 1)\) has been reduced to a \(4\)-list, since
		\((2, t + 2)\) has already been colored (when \(j = 2\)). However,
		the list on \((1, t + 2)\) has also been reduced to a \(4\)-list,
		so Lemma~\ref{L:small} is still applicable.
		
		\item
		For each \(j \in \Set{ t + 4, t + 6, \dotsc, s - t - 2 }\),
		notice that the list on \((2, j - 1)\)
		has been reduced to a \(3\)-list, but the lists on \((1, j)\) and \((1, j - 1)\)
		have also been reduced to \(4\)-lists each, so Lemma~\ref{L:small} is still applicable.
		
		\item
		For \(j = s - t\), notice that the list on \((2, s - t - 1)\) is reduced
		to a \(3\)-list and the list on \((2, s)\) is reduced to a \(4\)-list,
		but the list on \((1, s - t)\) is reduced to a \(3\)-list and the list
		on \((1, s - t - 1)\) is reduced to a \(4\)-list. So, Lemma~\ref{L:small} is still applicable.
		
		\item
		Lastly, for each \(j \in \Set{s - t + 2, s - t + 4, \dotsc, s}\), 
		notice that the lists on \((2, j - 1)\)
		and \((2, j + t)\) are reduced to \(3\)-lists each, but the lists on \((1, j)\) and \((1, j - 1)\)
		have also been reduced to \(3\)-lists each, so Lemma~\ref{L:small} is still applicable.
	\end{itemize}
	
	Next, consider the case when \(t = \frac{s}{2} - 1\). Since
	\(t + 2 = s - t\), some of the cases above reduce to a single degenerate case,
	as explained below:
	\begin{itemize}
		\item
		As before, for each \(j \in \bigl\{2, 4, \dotsc, \frac{s}{2} - 1\bigr\}\), we may use
		Lemma~\ref{L:small} to color the vertices of the form \((2, j - 1)\) and \((2, j + t)\).
		
		\item
		For \(j = \frac{s}{2} + 1\), we need to color \(\bigl( 2, \frac{s}{2} \bigr)\) and \((2, s)\),
		but notice that the lists on \(\bigl( 2, \frac{s}{2} \bigr)\) and \((2, s)\) have been reduced to \(4\)-lists.
		However, the list on \(\bigl( 1, \frac{s}{2} + 1 \bigr)\) has also been reduced to a \(3\)-list,
		so Lemma~\ref{L:small} is still applicable.
		
		\item
		Lastly, just as before, for each \(j \in \bigl\{\frac{s}{2} + 3, \frac{s}{2} + 5, \dotsc, s\bigr\}\), 
		notice that the lists on \((2, j - 1)\)
		and \((2, j + t)\) are reduced to \(3\)-lists each, but the lists on \((1, j)\) and \((1, j - 1)\)
		have also been reduced to \(3\)-lists each, so Lemma~\ref{L:small} is still applicable.
	\end{itemize}
	
	Thus, the coloring scheme outlined in the beginning can be implemented
	over all even \(j\) to get a proper coloring of \(C_{2}\) so that
	the lists on \(C_{1}\) are all reduced to \(2\)-lists. The proof is completed
	by using Lemma~\ref{L:cycle-choosability} to properly color \(C_{1}\). Clearly,
	the list coloring is found in linear time.
\end{proof}

Figure~\ref{F:part2} illustrates the coloring algorithm for \(G = T(2, 10, 4)\).

\begin{figure}
\centering
	\begin{subfigure}{0.3\textwidth}
	\centering
		\begin{tikzpicture}[scale=0.8,font=\tiny]
			\draw (1,1) grid (3,10);
			\foreach \y in {2,...,10}
			\foreach \x in {1,2}
			\draw (\x,\y) -- (\x+1,\y-1);
			\foreach \x in {1,2,3}
			\foreach \y in {1,...,10}
			\draw[fill=black] (\x,\y) circle(2.5pt);
			\foreach \y in {1,...,10}{
				\node[anchor=east] at (1,\y) {(1,\y)};
				\pgfmathparse{Mod(\y+3,10)+1}
				\node[anchor=west] at (3,\pgfmathresult) {(1,\y)};
			}
			\foreach \y in {2,...,9}{
				\node[anchor=south west] at (2,\y) {\!\!(2,\y)};
			}
			\node[anchor=north] at (2,1) {(2,1)};
			\node[anchor=south] at (2,10) {(2,10)};
		\end{tikzpicture}
		\caption{}
	\end{subfigure}
	\begin{subfigure}{0.3\textwidth}
	\centering
		\begin{tikzpicture}[scale=0.8,font=\scriptsize]	
			\draw
			(1,2) -- (2,2) -- (3,2)
			(1,3) -- (2,3) -- (3,3)
			(1,4) -- (2,4) -- (3,4)
			(1,5) -- (2,5) -- (3,5)
			(1,7) -- (2,7) -- (3,7)
			(1,8) -- (2,8) -- (3,8)
			(1,9) -- (2,9) -- (3,9)
			(1,10) -- (2,10) -- (3,10);
			\draw
			(1,3) -- (2,2) -- (3,1)
			(1,4) -- (2,3) -- (3,2)
			(1,5) -- (2,4) -- (3,3)
			(1,6) -- (2,5) -- (3,4)
			(1,8) -- (2,7) -- (3,6)
			(1,9) -- (2,8) -- (3,7)
			(1,10) -- (2,9) -- (3,8)
			(2,10) -- (3,9);
			\draw
			(2,2) -- (2,3) -- (2,4) -- (2,5)
			(2,7) -- (2,8) -- (2,9) -- (2,10);
			
			\draw[dotted, line width=1pt]
			(1,1) -- (2,1) -- (3,1)
			(1,2) -- (2,1)
			(2,1) -- (2,2)
			(1,6) -- (2,6) -- (3,6)
			(1,7) -- (2,6) -- (3,5)
			(2,5) -- (2,6) -- (2,7);
			\foreach \x in {1,3}
			\foreach \y in {2,...,10}
			\draw (\x,\y) -- (\x,\y-1);
			\draw[fill=white] (1,1) circle(2.5pt) node[anchor=east] {\textbf{4}};
			\draw[fill=white] (1,2) circle(2.5pt) node[anchor=east] {\textbf{4}};
			\draw[fill=black] (1,3) circle(2.5pt) node[anchor=east] {5};
			\draw[fill=black] (1,4) circle(2.5pt) node[anchor=east] {5};
			\draw[fill=black] (1,5) circle(2.5pt) node[anchor=east] {5};
			\draw[fill=black] (1,6) circle(2.5pt) node[anchor=east] {\textbf{3}};
			\draw[fill=black] (1,7) circle(2.5pt) node[anchor=east] {\textbf{3}};
			\draw[fill=black] (1,8) circle(2.5pt) node[anchor=east] {5};
			\draw[fill=black] (1,9) circle(2.5pt) node[anchor=east] {5};
			\draw[fill=black] (1,10) circle(2.5pt) node[anchor=east] {5};
			
			\draw[fill=gray] (2,1) circle(2.5pt) node[anchor=north] {\vphantom{(2,1)}};
			\draw[fill=black] (2,2) circle(2.5pt) node[anchor=south west] {\textbf{4}};
			\draw[fill=black] (2,3) circle(2.5pt) node[anchor=south west] {5};
			\draw[fill=black] (2,4) circle(2.5pt) node[anchor=south west] {5};
			\draw[fill=black] (2,5) circle(2.5pt) node[anchor=south west] {\textbf{4}};
			\draw[fill=gray] (2,6) circle(2.5pt) node[anchor=south west] {};
			\draw[fill=black] (2,7) circle(2.5pt) node[anchor=south west] {\textbf{4}};
			\draw[fill=black] (2,8) circle(2.5pt) node[anchor=south west] {5};
			\draw[fill=black] (2,9) circle(2.5pt) node[anchor=south west] {5};
			\draw[fill=black] (2,10) circle(2.5pt) node[anchor=south] {\textbf{4}};
			
			\draw[fill=black] (3,1) circle(2.5pt) node[anchor=west] {\textbf{3}};
			\draw[fill=black] (3,2) circle(2.5pt) node[anchor=west] {5};
			\draw[fill=black] (3,3) circle(2.5pt) node[anchor=west] {5};
			\draw[fill=black] (3,4) circle(2.5pt) node[anchor=west] {5};
			\draw[fill=white] (3,5) circle(2.5pt) node[anchor=west] {\textbf{4}};
			\draw[fill=white] (3,6) circle(2.5pt) node[anchor=west] {\textbf{4}};
			\draw[fill=black] (3,7) circle(2.5pt) node[anchor=west] {5};
			\draw[fill=black] (3,8) circle(2.5pt) node[anchor=west] {5};
			\draw[fill=black] (3,9) circle(2.5pt) node[anchor=west] {5};
			\draw[fill=black] (3,10) circle(2.5pt) node[anchor=west] {\textbf{3}};
		\end{tikzpicture}
		\caption{}
	\end{subfigure}
	\begin{subfigure}{0.3\textwidth}
	\centering
		\begin{tikzpicture}[scale=0.8,font=\scriptsize]	
			\draw
			(1,2) -- (2,2) -- (3,2)
			(1,4) -- (2,4) -- (3,4)
			(1,5) -- (2,5) -- (3,5)
			(1,7) -- (2,7) -- (3,7)
			(1,9) -- (2,9) -- (3,9)
			(1,10) -- (2,10) -- (3,10);
			\draw
			(1,3) -- (2,2) -- (3,1)
			(1,5) -- (2,4) -- (3,3)
			(1,6) -- (2,5) -- (3,4)
			(1,8) -- (2,7) -- (3,6)
			(1,10) -- (2,9) -- (3,8)
			(2,10) -- (3,9);
			\draw
			(2,4) -- (2,5)
			(2,9) -- (2,10);
			
			\draw[dotted, line width=1pt]
			(1,3) -- (2,3) -- (3,3)
			(1,4) -- (2,3) -- (3,2)
			(2,2) -- (2,3) -- (2,4)
			(1,8) -- (2,8) -- (3,8)
			(1,9) -- (2,8) -- (3,7)
			(2,7) -- (2,8) -- (2,9);
			
			\foreach \x in {1,3}
			\foreach \y in {2,...,10}
			\draw (\x,\y) -- (\x,\y-1);
			\draw[fill=black] (1,1) circle(2.5pt) node[anchor=east] {4};
			\draw[fill=black] (1,2) circle(2.5pt) node[anchor=east] {4};
			\draw[fill=white] (1,3) circle(2.5pt) node[anchor=east] {\textbf{4}};
			\draw[fill=white] (1,4) circle(2.5pt) node[anchor=east] {\textbf{4}};
			\draw[fill=black] (1,5) circle(2.5pt) node[anchor=east] {5};
			\draw[fill=black] (1,6) circle(2.5pt) node[anchor=east] {3};
			\draw[fill=black] (1,7) circle(2.5pt) node[anchor=east] {3};
			\draw[fill=black] (1,8) circle(2.5pt) node[anchor=east] {\textbf{3}};
			\draw[fill=black] (1,9) circle(2.5pt) node[anchor=east] {\textbf{3}};
			\draw[fill=black] (1,10) circle(2.5pt) node[anchor=east] {5};
			
			\node[anchor=north] at (2,1) {\vphantom{(2,1)}};
			\draw[fill=black] (2,2) circle(2.5pt) node[anchor=south west] {\textbf{3}};
			\draw[fill=gray] (2,3) circle(2.5pt) node[anchor=south west] {};
			\draw[fill=black] (2,4) circle(2.5pt) node[anchor=south west] {\textbf{4}};
			\draw[fill=black] (2,5) circle(2.5pt) node[anchor=south west] {4};
			\draw[fill=black] (2,7) circle(2.5pt) node[anchor=south west] {\textbf{3}};
			\draw[fill=gray] (2,8) circle(2.5pt) node[anchor=south west] {};
			\draw[fill=black] (2,9) circle(2.5pt) node[anchor=south west] {\textbf{4}};
			\draw[fill=black] (2,10) circle(2.5pt) node[anchor=south] {4};
			
			\draw[fill=black] (3,5) circle(2.5pt) node[anchor=west] {4};
			\draw[fill=black] (3,6) circle(2.5pt) node[anchor=west] {4};
			\draw[fill=white] (3,7) circle(2.5pt) node[anchor=west] {\textbf{4}};
			\draw[fill=white] (3,8) circle(2.5pt) node[anchor=west] {\textbf{4}};
			\draw[fill=black] (3,9) circle(2.5pt) node[anchor=west] {5};
			\draw[fill=black] (3,10) circle(2.5pt) node[anchor=west] {3};
			\draw[fill=black] (3,1) circle(2.5pt) node[anchor=west] {3};
			\draw[fill=black] (3,2) circle(2.5pt) node[anchor=west] {\textbf{3}};
			\draw[fill=black] (3,3) circle(2.5pt) node[anchor=west] {\textbf{3}};
			\draw[fill=black] (3,4) circle(2.5pt) node[anchor=west] {5};
			
		\end{tikzpicture}
		\caption{}
	\end{subfigure}\\[2em]
	\begin{subfigure}{0.3\textwidth}
	\centering
		\begin{tikzpicture}[scale=0.8,font=\scriptsize]	
			\draw
			(1,2) -- (2,2) -- (3,2)
			(1,4) -- (2,4) -- (3,4)
			(1,7) -- (2,7) -- (3,7)
			(1,9) -- (2,9) -- (3,9);
			\draw
			(1,3) -- (2,2) -- (3,1)
			(1,5) -- (2,4) -- (3,3)
			(1,8) -- (2,7) -- (3,6)
			(1,10) -- (2,9) -- (3,8);
			
			\draw[dotted, line width=1pt]
			(1,5) -- (2,5) -- (3,5)
			(1,6) -- (2,5) -- (3,4)
			(2,4) -- (2,5)
			(1,10) -- (2,10) -- (3,10)
			(2,10) -- (3,9)
			(2,9) -- (2,10);
			
			\foreach \x in {1,3}
			\foreach \y in {2,...,10}
			\draw (\x,\y) -- (\x,\y-1);
			\draw[fill=black] (1,1) circle(2.5pt) node[anchor=east] {\textbf{2}};
			\draw[fill=black] (1,2) circle(2.5pt) node[anchor=east] {4};
			\draw[fill=black] (1,3) circle(2.5pt) node[anchor=east] {4};
			\draw[fill=black] (1,4) circle(2.5pt) node[anchor=east] {4};
			\draw[fill=white] (1,5) circle(2.5pt) node[anchor=east] {\textbf{4}};
			\draw[fill=white] (1,6) circle(2.5pt) node[anchor=east] {\textbf{2}};
			\draw[fill=black] (1,7) circle(2.5pt) node[anchor=east] {3};
			\draw[fill=black] (1,8) circle(2.5pt) node[anchor=east] {3};
			\draw[fill=black] (1,9) circle(2.5pt) node[anchor=east] {\phantom{(,11)}3};
			\draw[fill=black] (1,10) circle(2.5pt) node[anchor=east] {\textbf{3}};
			
			\node[anchor=north] at (2,1) {\vphantom{(2,1)}};
			\draw[fill=black] (2,2) circle(2.5pt) node[anchor=south west] {3};
			\draw[fill=black] (2,4) circle(2.5pt) node[anchor=south west] {\textbf{3}};
			\draw[fill=gray] (2,5) circle(2.5pt) node[anchor=south west] {};
			\draw[fill=black] (2,7) circle(2.5pt) node[anchor=south west] {3};
			\draw[fill=black] (2,9) circle(2.5pt) node[anchor=south west] {\textbf{3}};
			\draw[fill=gray] (2,10) circle(2.5pt) node[anchor=south] {};
			
			\draw[fill=black] (3,5) circle(2.5pt) node[anchor=west] {\textbf{2}};
			\draw[fill=black] (3,6) circle(2.5pt) node[anchor=west] {4};
			\draw[fill=black] (3,7) circle(2.5pt) node[anchor=west] {4};
			\draw[fill=black] (3,8) circle(2.5pt) node[anchor=west] {4};
			\draw[fill=white] (3,9) circle(2.5pt) node[anchor=west] {\textbf{4}};
			\draw[fill=white] (3,10) circle(2.5pt) node[anchor=west] {\textbf{2}};
			\draw[fill=black] (3,1) circle(2.5pt) node[anchor=west] {3};
			\draw[fill=black] (3,2) circle(2.5pt) node[anchor=west] {3};
			\draw[fill=black] (3,3) circle(2.5pt) node[anchor=west] {3\phantom{(,11)}};
			\draw[fill=black] (3,4) circle(2.5pt) node[anchor=west] {\textbf{3}};
			
		\end{tikzpicture}
		\caption{}
	\end{subfigure}
	\begin{subfigure}{0.3\textwidth}
	\centering
		\begin{tikzpicture}[scale=0.8,font=\scriptsize]	
			\draw
			(1,4) -- (2,4) -- (3,4)
			(1,9) -- (2,9) -- (3,9);
			\draw
			(1,5) -- (2,4) -- (3,3)
			(1,10) -- (2,9) -- (3,8);
			
			\draw[dotted, line width=1pt]
			(1,7) -- (2,7) -- (3,7)
			(1,8) -- (2,7) -- (3,6)
			(1,2) -- (2,2) -- (3,2)
			(1,3) -- (2,2) -- (3,1);
			
			\foreach \x in {1,3}
			\foreach \y in {2,...,10}
			\draw (\x,\y) -- (\x,\y-1);
			\draw[fill=black] (1,1) circle(2.5pt) node[anchor=east] {2};
			\draw[fill=black] (1,2) circle(2.5pt) node[anchor=east] {\textbf{2}};
			\draw[fill=black] (1,3) circle(2.5pt) node[anchor=east] {\textbf{2}};
			\draw[fill=black] (1,4) circle(2.5pt) node[anchor=east] {4};
			\draw[fill=black] (1,5) circle(2.5pt) node[anchor=east] {4};
			\draw[fill=black] (1,6) circle(2.5pt) node[anchor=east] {2};
			\draw[fill=white] (1,7) circle(2.5pt) node[anchor=east] {\textbf{2}};
			\draw[fill=white] (1,8) circle(2.5pt) node[anchor=east] {\textbf{2}};
			\draw[fill=black] (1,9) circle(2.5pt) node[anchor=east] {3};
			\draw[fill=black] (1,10) circle(2.5pt) node[anchor=east] {3};
			
			\node[anchor=north] at (2,1) {\vphantom{(2,1)}};
			\draw[fill=gray] (2,2) circle(2.5pt) node[anchor=south west] {};
			\draw[fill=black] (2,4) circle(2.5pt) node[anchor=south west] {3};
			\draw[fill=gray] (2,7) circle(2.5pt) node[anchor=south west] {};
			\draw[fill=black] (2,9) circle(2.5pt) node[anchor=south west] {3};
			
			\draw[fill=black] (3,5) circle(2.5pt) node[anchor=west] {2};
			\draw[fill=black] (3,6) circle(2.5pt) node[anchor=west] {\textbf{2}};
			\draw[fill=black] (3,7) circle(2.5pt) node[anchor=west] {\textbf{2}};
			\draw[fill=black] (3,8) circle(2.5pt) node[anchor=west] {4};
			\draw[fill=black] (3,9) circle(2.5pt) node[anchor=west] {4};
			\draw[fill=black] (3,10) circle(2.5pt) node[anchor=west] {2};
			\draw[fill=white] (3,1) circle(2.5pt) node[anchor=west] {\textbf{2}};
			\draw[fill=white] (3,2) circle(2.5pt) node[anchor=west] {\textbf{2}};
			\draw[fill=black] (3,3) circle(2.5pt) node[anchor=west] {3};
			\draw[fill=black] (3,4) circle(2.5pt) node[anchor=west] {3};
			
		\end{tikzpicture}
		\caption{}
	\end{subfigure}
	\begin{subfigure}{0.3\textwidth}
	\centering
		\begin{tikzpicture}[scale=0.8,font=\scriptsize]	
			\draw[dotted, line width=1pt]
			(1,9) -- (2,9) -- (3,9)
			(1,10) -- (2,9) -- (3,8)
			(1,4) -- (2,4) -- (3,4)
			(1,5) -- (2,4) -- (3,3);
			
			\foreach \x in {1,3}
			\foreach \y in {2,...,10}
			\draw (\x,\y) -- (\x,\y-1);
			
			\draw[fill=black] (1,1) circle(2.5pt) node[anchor=east] {2};
			\draw[fill=black] (1,2) circle(2.5pt) node[anchor=east] {2};
			\draw[fill=black] (1,3) circle(2.5pt) node[anchor=east] {2};
			\draw[fill=black] (1,4) circle(2.5pt) node[anchor=east] {\textbf{2}};
			\draw[fill=black] (1,5) circle(2.5pt) node[anchor=east] {\textbf{2}};
			\draw[fill=black] (1,6) circle(2.5pt) node[anchor=east] {2};
			\draw[fill=black] (1,7) circle(2.5pt) node[anchor=east] {2};
			\draw[fill=black] (1,8) circle(2.5pt) node[anchor=east] {2};
			\draw[fill=white] (1,9) circle(2.5pt) node[anchor=east] {\textbf{2}};
			\draw[fill=white] (1,10) circle(2.5pt) node[anchor=east] {\textbf{2}};
			
			\node[anchor=north] at (2,1) {\vphantom{(2,1)}};
			\draw[fill=gray] (2,4) circle(2.5pt) node[anchor=south west] {};
			\draw[fill=gray] (2,9) circle(2.5pt) node[anchor=south west] {};
			
			\draw[fill=black] (3,5) circle(2.5pt) node[anchor=west] {2};
			\draw[fill=black] (3,6) circle(2.5pt) node[anchor=west] {2};
			\draw[fill=black] (3,7) circle(2.5pt) node[anchor=west] {2};
			\draw[fill=black] (3,8) circle(2.5pt) node[anchor=west] {\textbf{2}};
			\draw[fill=black] (3,9) circle(2.5pt) node[anchor=west] {\textbf{2}};
			\draw[fill=black] (3,10) circle(2.5pt) node[anchor=west] {2};
			\draw[fill=black] (3,1) circle(2.5pt) node[anchor=west] {2};
			\draw[fill=black] (3,2) circle(2.5pt) node[anchor=west] {2};
			\draw[fill=white] (3,3) circle(2.5pt) node[anchor=west] {\textbf{2}};
			\draw[fill=white] (3,4) circle(2.5pt) node[anchor=west] {\textbf{2}};
			
		\end{tikzpicture}
		\caption{}
	\end{subfigure}
	\caption{Illustration of the proof of case~\ref{l3} in Theorem~\ref{T:Main} for \(G = T(2, 10, 4)\).}\label{F:part2}
\end{figure}

\section{Proof of \texorpdfstring{case~\ref{l4}}{case (4)} in \texorpdfstring{Theorem~\ref{T:Main}}{Theorem 1}}\label{S:baaki}

\begin{lemma}\label{L:fix}
	Let \(G\) be a \(3\)-regular bipartite graph. Let \(V(G)\)
	be partitioned into two independent sets \(A\) and \(B\).
	Let \(\ListAsgn{L}\) be a list assignment that assigns a list of size \(3\) to each vertex in \(A\)
	and a list of size \(2\) to each vertex in \(B\). Then, \(G\) is \(\ListAsgn{L}\)-choosable.
	Moreover, such a list coloring can be found in linear time.
\end{lemma}
\begin{proof}
	Find a perfect matching \(M\) in \(G\) and consider the graph \(G - M\).
	Since \(G - M\) is a \(2\)-regular bipartite graph, it breaks up into a union
	of disjoint even cycles. Place an orientation on the edges of \(G\) such
	that each of these cycles becomes a directed cycle, and such that the edges in \(M\)
	are oriented from \(A\) to \(B\). Then, \(\outdegree(v) = 2\) for every \(v \in A\)
	and \(\outdegree(w) = 1\) for every \(w \in B\).
	Note that there are no odd directed cycles in this orientation, since \(G\) is bipartite.
	Thus, by Lemma~\ref{L:BBS}, \(G\) is \(\ListAsgn{L}\)-choosable
	for any list assignment that assigns a list of size \(3\) to each vertex in \(A\)
	and a list of size \(2\) to each vertex in \(B\).
	
	Furthermore, Theorem~\ref{T:matching} shows that \(M\) can be found in \(O(\card{E})\) time,
	which is also \(O(\card{V})\) time, since \(G\) is of bounded degree.
	Thus, the list coloring can be found in linear time.
\end{proof}

In an earlier paper~\cite{BalachandranSankarnarayanan2021}, we proved the following theorem:
\begin{theorem}[Balachandran--Sankarnarayanan~\cite{BalachandranSankarnarayanan2021}, 2021]\label{T:old}
	Let \(G\) be a simple \(6\)-regular toroidal triangulation. If \(G\)
	is \(3\)-chromatic, then \(G\) is \(5\)-choosable.
\end{theorem}

The proof of this theorem is entirely algorithmic, except for one use of a theorem of Alon
and Tarsi~\cite{AlonTarsi1992} to show that a toroidal \(3\)-regular bipartite graph is \(\ListAsgn{L}\)-choosable
for a list assignment \(\ListAsgn{L}\) as in the hypothesis of Lemma~\ref{L:fix}.
Using the proof of Lemma~\ref{L:fix} in its place, we obtain the proof of case~\ref{l4} in Theorem~\ref{T:Main}
as a corollary:

\begin{corollary}
	Every simple \(3\)-chromatic \(6\)-regular toroidal triangulation
	is \(5\)-choosable. Moreover, a \(5\)-list coloring can be found in linear time.
\end{corollary}

Lastly, we show that all the simple \(3\)-chromatic graphs \(T(r, s, t)\) are not \(3\)-choosable.
Note that \(T(r, s, t)\) is \(3\)-chromatic if and only if \(s \equiv 0 \equiv r - t \pmod{3}\).
Let
\[
\List{L}{1} \defn \Set{1, 2, 3}, \quad \List{L}{2} \defn \Set{2, 3, 4},\quad \List{L}{3} \defn \Set{1, 3, 4}.
\]

\subsection{The graphs \texorpdfstring{\(T(r, s, t)\)}{T(r, s, t)} for \texorpdfstring{\(r \geq 4\)}{r >= 4}, \texorpdfstring{\(s \geq 3\)}{s >= 3}}
Let \(\ListAsgn{L}\) be the list-assignment that assigns
the above lists to the columns of \(T(r, s, t)\) as follows:
\begin{align*}
	\List{L}{1} &: C_{1}, C_{2};\\
	\List{L}{2} &: C_{3};\\
	\List{L}{3} &: C_{4}, \dotsc, C_{r}.
\end{align*}
Let the vertices \((1, 1)\) and \((1, 2)\) be properly
colored using \(\ListAsgn{L}\) in any manner.
This uniquely determines a proper coloring of the induced subgraph on \(C_{1} \cup C_{2}\).

Now, there is a unique way to extend this coloring properly to the induced subgraph
on \(C_{2} \cup C_{3}\) as follows:
simply extend the coloring from \(C_{2}\) to \(C_{3}\) using the same lists
used on \(C_{2}\), namely \(\List{L}{1} = \Set{1, 2, 3}\); then, recolor all the vertices
in \(C_{3}\) that have the color \(1\) with the color \(4\). The reason
behind this is as follows:
\begin{itemize}
	\item whenever there exist two adjacent vertices in \(C_{2}\) that
	are colored using \(\Set{2, 3} \subset \List{L}{1} \cap \List{L}{2}\),
	the common neighbor of these two vertices in \(C_{3}\) must receive the color
	\(4\);
	
	\item in any proper coloring of \(C_{1} \cup C_{2}\), there will
	be \(s/3\) pairs of vertices in \(C_{2}\) that are colored using \(\Set{2, 3}\),
	and no two of these pairs are adjacent in \(C_{2}\);
	
	\item if a vertex in \(C_{3}\) has its color fixed to be \(4\) as above,
	then the colors of its two vertical neighbors are also fixed.
\end{itemize}
In this manner, one can see that the coloring is extended uniquely to the rest of \(C_{3}\),
with \(4\) occurring in those places where \(1\) would have occured had \(C_{3}\) also
been colored using \(\List{L}{1} = \Set{1, 2, 3}\).

Next, repeat the same process
to extend the coloring on \(C_{3}\) to a proper coloring on the induced subgraph
on \(C_{3} \cup C_{4} \cup \dotsb \cup C_{r}\) as follows: color the vertices in \(C_{4} \cup \dotsb \cup C_{r}\)
using the colors used on \(C_{3}\), namely \(\List{L}{2} = \Set{2, 3, 4}\), and then recolor
those vertices in \(C_{4} \cup \dotsb \cup C_{r}\) that have the color \(2\) with the color \(1\).

Now, we note that this coloring cannot be proper on all of \(T(r, s, t)\)
because this process of successive relabeling
has mapped the tuple \((1, 2, 3)\) to \((2, 1, 3)\). Thus, for this to be a proper
coloring of \(T(r, s, t)\), the original coloring on \(C_{1}\) must arise
as the unique extension of the coloring on \(C_{r}\) to the induced subgraph
on \(C_{r} \cup C_{1}\); but, \((2, 1, 3)\) is not a cyclic permutation
of \((1, 2, 3)\), so this cannot happen for any \(t\).

\subsection{The graphs \texorpdfstring{\(T(2, s, t)\)}{T(2, s, t)} for \texorpdfstring{\(s \geq 6\)}{s >= 6}}
First, consider the case \(s \geq 12\).
Since \(T(2, s, t)\) is assumed to be simple, we ignore the case \(t = s - 1\).
Next, by the remarks in Section~\ref{S:Preliminaries},
\(T(2, s, t)\) is isomorphic to \(T(2, s, s - t - 2)\).
Furthermore, \(t \equiv 2 \pmod{3}\) since \(T(2, s, t)\)
is assumed to be \(3\)-chromatic. Hence,
it suffices to assume that either \(t = s - 4\), or \(t\)
lies in the range \(5 \leq t \leq \floor{s/2} - 1\).

Now, let \(R_{1}, \dotsc, R_{s}\) denote the \(s\) rows of \(T(2, s, t)\). Let
\(\ListAsgn{L}\) be the list assignment that assigns the lists \(\List{L}{1}, \List{L}{2}, \List{L}{3}\)
to the rows of \(T(2, s, t)\) as follows, where \(s = 3\ell\):
\begin{align*}
	\List{L}{1} &: R_{1}, \dotsc, R_{\ell};\\
	\List{L}{2} &: R_{\ell+1},\dotsc,R_{2\ell};\\
	\List{L}{3} &: R_{2\ell+1},\dotsc,R_{s}.
\end{align*}
Figure~\ref{SF:1} illustrates this list assignment for the graph \(T(2, 12, 5)\);
the vertices colored red, yellow, and blue are assigned the lists
\(\List{L}{1}\), \(\List{L}{2}\), and \(\List{L}{3}\), respectively.

Let the vertices \((1, 1)\) and \((1, 2)\) be properly
colored using \(\ListAsgn{L}\) in any manner.
This uniquely determines a proper coloring of the first \(\ell\) rows.
Now, we can find a vertex \(v\) in one of the remaining
blocks of size \(\ell\) such that the residual list on \(v\) has size equal to \(1\),
as follows:
\begin{itemize}
	\item If the colors on \((1,\ell)\) and \((2,\ell)\) are both present
	in \(\List{L}{2}\), then take \(v = (1,\ell + 1)\).
	
	\item If the color on \((1, \ell)\) is \(1\), then the pairs
	\((1, \ell - 1), (1, \ell - 2)\) and \((2, \ell), (2, \ell - 1)\)
	are colored using \(\Set{2, 3}\). Additionally, if the color on \((2, \ell)\)
	is \(2\), then the pair \((2,\ell - 1), (2,\ell - 2)\) is colored using
	\(\Set{1, 3}\), and if the color on \((2, \ell)\) is \(3\), then the pair
	\((1, \ell), (1, \ell - 1)\) is colored using \(\Set{1, 3}\).
	
	From this data, one choice for \(v\) is given by:
	\[
	v =
	\begin{cases}
		(2, \ell - 1 + t), & 5 \leq t \leq \ell + 1;\\
		(1, \ell - 2 - t), & \ell + 2 \leq t \leq \floor{s/2} - 1
		\text{ and } (2, \ell) \text{ is colored } 2;\\
		(2, \ell + t), & \ell + 2 \leq t \leq \floor{s/2} - 1
		\text{ and } (2, \ell) \text{ is colored } 3;\\
		(1, \ell + 3), & t = s - 4.
	\end{cases}
	\]
	
	\item A similar analysis can be done when \((2, \ell)\) is given the color \(1\).
	In that case, one choice for \(v\) is:
	\[
	v =
	\begin{cases}
		(2, \ell - 2 + t), & 5 \leq t \leq \ell + 1;\\
		(1, \ell - 1 - t), & \ell + 2 \leq t \leq \floor{s/2} - 1
		\text{ and } (1, \ell) \text{ is colored } 2;\\
		(2, \ell + t), & \ell + 2 \leq t \leq \floor{s/2} - 1
		\text{ and } (1, \ell) \text{ is colored } 3;\\
		(1, \ell + 2), & t = s - 4.
	\end{cases}
	\]
\end{itemize}

This allows one to extend the coloring to the entire block of \(\ell\) vertices
in which \(v\) belongs. Finally, we can repeat the process to color the third block
of \(\ell\) vertices. But, this is not a proper coloring of \(T(2, s, t)\) for the
same reason as in the previous case.

The remaining cases when \(r = 2\) are \(T(2, 6, 2)\), \(T(2, 9, 2)\) and \(T(2, 9, 5)\).
Since the last two are isomorphic to each other, we shall only consider \(T(2, 9, 5)\).
In this case, we use the same list assignment as in the case when \(s \geq 12\),
and we note that after the first block of \(\ell\) vertices is colored
we can choose \(v\) to be either \((1, 6)\) or \((1, 5)\), depending
on whether the color \(1\) is given to \((1, 3)\) or \((2, 3)\), respectively.
The only graph left to consider is \(T(2, 6, 2)\), which we handle in an ad hoc
manner in~\ref{S:Appendix}.

\begin{figure}
\centering
	\begin{tikzpicture}[font=\scriptsize]
		\draw (1,1) grid (4,12);
		\foreach \y in {2,...,12}
		\foreach \x in {1,2,3}
		\draw (\x,\y) -- (\x+1,\y-1);
		\foreach \y in {1,...,4}{
			\draw[fill=myred] (2,\y) circle(2.5pt);
			\draw[fill=myred] (3,\y) circle(2.5pt);
		}
		\foreach \y in {5,...,8}{
			\draw[fill=mygreen] (2,\y) circle(2.5pt);
			\draw[fill=mygreen] (3,\y) circle(2.5pt);
		}
		\foreach \y in {9,...,12}{
			\draw[fill=myblue] (2,\y) circle(2.5pt);
			\draw[fill=myblue] (3,\y) circle(2.5pt);
		}
		\foreach \y in {1,...,4}{
			\pgfmathparse{Mod(\y+6,12)+1}
			\draw[fill=myred] (1,\pgfmathresult) circle(2.5pt) node[anchor=east] {(2,\y)};
			\pgfmathparse{Mod(\y+4,12)+1}
			\draw[fill=myred] (4,\pgfmathresult) circle(2.5pt) node[anchor=west] {(1,\y)};
		}
		\foreach \y in {5,...,8}{
			\pgfmathparse{Mod(\y+6,12)+1}
			\draw[fill=mygreen] (1,\pgfmathresult) circle(2.5pt) node[anchor=east] {(2,\y)};
			\pgfmathparse{Mod(\y+4,12)+1}
			\draw[fill=mygreen] (4,\pgfmathresult) circle(2.5pt) node[anchor=west] {(1,\y)};
		}
		\foreach \y in {9,...,12}{
			\pgfmathparse{Mod(\y+6,12)+1}
			\draw[fill=myblue] (1,\pgfmathresult) circle(2.5pt) node[anchor=east] {(2,\y)};
			\pgfmathparse{Mod(\y+4,12)+1}
			\draw[fill=myblue] (4,\pgfmathresult) circle(2.5pt) node[anchor=west] {(1,\y)};
		}
		\foreach \y in {2,...,11}{
			\node[right=10pt,above=0.1pt] at (2,\y) {(1,\y)};
			\node[right=10pt,above=0.1pt] at (3,\y) {(2,\y)};
		}
		\node[anchor=north] at (2,1) {(1,1)};
		\node[anchor=north] at (3,1) {(2,1)};
		\node[anchor=south] at (2,12) {(1,12)};
		\node[anchor=south] at (3,12) {(2,12)};
	\end{tikzpicture}
	\caption{Illustration of non-\(\ListAsgn{L}\)-colorable
		\(3\)-list-assignments \(\ListAsgn{L}\) on \(3\)-chromatic graphs \(T(2, s, t)\)
		for \(s \geq 12\) via \(G = T(2, 12, 5)\).
		Distinct colors denote distinct lists
		among \(\List{L}{1}, \List{L}{2}, \List{L}{3}\).}\label{SF:1}
\end{figure}

\subsection{The graphs \texorpdfstring{\(T(3, s, t)\)}{T(3, s t)} for \texorpdfstring{\(s \geq 3\)}{s >= 3}}
First, consider the case \(s \geq 12\) and \(t \neq 0\).
By the remarks in Section~\ref{S:Preliminaries},
\(T(3, s, t)\) is isomorphic to \(T(3, s, s - t - 3)\),
and \(t \equiv 0 \pmod{3}\) since \(T(3, s, t)\)
is assumed to be \(3\)-chromatic. Hence,
it suffices to assume that \(t\)
lies in the range \(3 \leq t \leq \floor{(s - 3)/2}\).
We use the same list assignment on these graphs as in the case
\(T(2, s, t)\) for \(s \geq 12\). 
A similar analysis shows that \(T(3, s, t)\)
is not \(3\)-choosable in these cases, so we omit the details.

The only remaining cases are \(T(3, s, 0)\) for \(s \geq 3\), \(T(3, 6, 3)\), \(T(3, 9, 3)\) and \(T(3, 9, 6)\).
It is easy to see that the same list assignment as above
also works for \(T(3, s, 0)\) for \(s \geq 6\). Also, the graph \(T(3, 3, 0)\) is isomorphic to \(K_{3, 3, 3}\),
which is known to be \(4\)-list chromatic~\cite{Kierstead2000}. Lastly, the graphs
\(T(3, 6, 0)\) and \(T(3, 6, 3)\) are isomorphic to each other,
and so are \(T(3, 9, 0)\) and \(T(3, 9, 6)\).
So, the only graph left to consider is \(T(3, 9, 3)\), which we handle in an ad hoc manner in \ref{S:Appendix}.

\subsection{The graphs \texorpdfstring{\(T(1, s, t)\)}{T(1, s, t)} for \texorpdfstring{\(s \geq 9\)}{s >= 9}}\label{SS:not-3-choosable}
Suppose that the \(3\)-chromatic graph \(T(1, s, t)\)
is not isomorphic to \(T(r', s', t')\) for any \(r' > 1\).
This happens if and only if \(\gcd(s, t) = 1 = \gcd(s, t + 1)\),
so we must have \(s \equiv 3 \pmod{6}\).
By the remarks in Section~\ref{S:Preliminaries},
\(T(1, s, t)\) is isomorphic to \(T(1, s, s - t - 1)\),
so we can assume \(0 \leq t \leq \floor{(s - 1)/2}\). Moreover,
\(T(1, s, t)\) has loops when \(t = 0\) and has multiple edges
when \(t = 1, \floor{(s - 1)/2}\), so we ignore these cases.
Since we assume that \(T(1, s, t)\) is \(3\)-chromatic, we also
have \(t \equiv 1 \pmod{3}\).
Thus, it suffices to consider only those \(t\)
in the range \(4 \leq t \leq (s - 7)/2\).
Note that the least value of \(s\) for which there exists
some \(t\) in the above range and for which \(\gcd(s, t) = 1 = \gcd(s, t + 1)\)
is \(s = 21\).

For simplicity,
we label the vertex \((1, j)\) with the integer \(j\)
(recall that \(j\) is taken modulo \(s\)).
We shall use the following modifications of the above coloring scheme.

First, suppose that \(7 \leq t < (s + 1) / 4\). 
Fix \(\List{L}{0}\) to be an arbitrary \(3\)-list. Let \(\ListAsgn{L}\) be the list assignment
on \(T(1, s, t)\) that assigns the lists \(\List{L}{0}, \List{L}{1}, \List{L}{2}, \List{L}{3}\)
as follows:
\begin{align*}
	\List{L}{1} &: \Set{ s - kt, s - 1 - kt : k = 0, 1, 2, 3, 4 },\\
	\List{L}{2} &: \Set{ s - 2 - kt, s - 3 - kt : k = 0, 1, 2, 3},\\
	\List{L}{3} &: \Set{ s - 4 - kt, \dotsc, s - t + 1 - kt : k = 0, 1, 2, 3 },
\end{align*}
and any remaining vertices are assigned the list \(\List{L}{0}\).
Figure~\ref{SF:circulant1} illustrates this list assignment for the graph \(T(1, 33, 7)\);
the vertices colored lilac, red, yellow, and blue are assigned the lists
\(\List{L}{0}\), \(\List{L}{1}\), \(\List{L}{2}\), and \(\List{L}{3}\), respectively.

Essentially the same arguments as before work in this case as well, so we omit the details
from here onward. Figure~\ref{SF:closeup1} shows a selected portion of Figure~\ref{SF:circulant1}
on which a similar argument as in the previous cases can be applied to show that
\(T(1, 33, 7)\) is not \(3\)-choosable.

\begin{figure}
\centering
	\begin{subfigure}{0.45\textwidth}
	\centering
		\begin{tikzpicture}[font=\scriptsize,scale=0.5]
			\draw (1,1) grid (5,33);
			\foreach \y in {2,...,33}
			\foreach \x in {1,...,4}
			\draw (\x,\y) -- (\x+1,\y-1);
			\foreach \y in {1,...,33}{
				\node[anchor=east] at (1,\y) {\y};
				\pgfmathparse{Mod(\y+27,33)+1}
				\node[anchor=west] at (5,\pgfmathresult) {\y};
			}
			\node[anchor=north] at (1,1) {\vphantom{\scriptsize{(1,1)}}};
			\node[anchor=north] at (2,1) {27};
			\node[anchor=north] at (3,1) {20};
			\node[anchor=north] at (4,1) {13};
			\node[anchor=south] at (2,33) {26};
			\node[anchor=south] at (3,33) {19};
			\node[anchor=south] at (4,33) {12};
			\foreach \y in {1,2,3}{
				\draw[fill=myyellow] (1,\y) circle(4pt);
				\pgfmathparse{Mod(\y+6,33)+1}
				\draw[fill=myyellow] (2,\pgfmathresult) circle(4pt);	
				\pgfmathparse{Mod(\y+13,33)+1}
				\draw[fill=myyellow] (3,\pgfmathresult) circle(4pt);	
				\pgfmathparse{Mod(\y+20,33)+1}
				\draw[fill=myyellow] (4,\pgfmathresult) circle(4pt);	
				\pgfmathparse{Mod(\y+27,33)+1}
				\draw[fill=myyellow] (5,\pgfmathresult) circle(4pt);	
			}
			\foreach \y in {4,5,11,12,18,19,25,26,32,33}{
				\draw[fill=myred] (1,\y) circle(4pt);
				\pgfmathparse{Mod(\y+6,33)+1}
				\draw[fill=myred] (2,\pgfmathresult) circle(4pt);	
				\pgfmathparse{Mod(\y+13,33)+1}
				\draw[fill=myred] (3,\pgfmathresult) circle(4pt);	
				\pgfmathparse{Mod(\y+20,33)+1}
				\draw[fill=myred] (4,\pgfmathresult) circle(4pt);	
				\pgfmathparse{Mod(\y+27,33)+1}
				\draw[fill=myred] (5,\pgfmathresult) circle(4pt);	
			}
			\foreach \y in {6,7,8,13,14,15,20,21,22,27,28,29}{
				\draw[fill=myblue] (1,\y) circle(4pt);
				\pgfmathparse{Mod(\y+6,33)+1}
				\draw[fill=myblue] (2,\pgfmathresult) circle(4pt);	
				\pgfmathparse{Mod(\y+13,33)+1}
				\draw[fill=myblue] (3,\pgfmathresult) circle(4pt);	
				\pgfmathparse{Mod(\y+20,33)+1}
				\draw[fill=myblue] (4,\pgfmathresult) circle(4pt);	
				\pgfmathparse{Mod(\y+27,33)+1}
				\draw[fill=myblue] (5,\pgfmathresult) circle(4pt);	
			}
			\foreach \y in {9,10,16,17,23,24,30,31}{
				\draw[fill=mygreen] (1,\y) circle(4pt);
				\pgfmathparse{Mod(\y+6,33)+1}
				\draw[fill=mygreen] (2,\pgfmathresult) circle(4pt);	
				\pgfmathparse{Mod(\y+13,33)+1}
				\draw[fill=mygreen] (3,\pgfmathresult) circle(4pt);	
				\pgfmathparse{Mod(\y+20,33)+1}
				\draw[fill=mygreen] (4,\pgfmathresult) circle(4pt);	
				\pgfmathparse{Mod(\y+27,33)+1}
				\draw[fill=mygreen] (5,\pgfmathresult) circle(4pt);	
			}
		\end{tikzpicture}
		\caption{\(T(1,33,7)\)}\label{SF:circulant1}
	\end{subfigure}
	\begin{subfigure}{0.45\textwidth}
	\centering
		\begin{tikzpicture}[font=\scriptsize]
			\draw (1,26) grid (5,33);
			\draw[dotted] (1,25) -- (1,26) -- (2,25) -- (2,26) -- (3,25) -- (3,26) -- (4,25) -- (4,26) -- (5,25) -- (5,26);
			\node[anchor=north] at (2,25) {\vphantom{(2,25)}};
			\foreach \y in {27,...,33}
			\foreach \x in {1,...,4}
			\draw (\x,\y) -- (\x+1,\y-1);
			\foreach \y in {26,...,33}{
				\node[anchor=east] at (1,\y) {\y};
			}
			\foreach \y in {19,...,25}{
				\pgfmathparse{Mod(\y+6,33)+1}
				\node[anchor=south west] at (2,\pgfmathresult) {\y};
			}
			\foreach \y in {12,...,18}{
				\pgfmathparse{Mod(\y+13,33)+1}
				\node[anchor=south west] at (3,\pgfmathresult) {\y};
			}
			\foreach \y in {5,...,11}{
				\pgfmathparse{Mod(\y+20,33)+1}
				\node[anchor=south west] at (4,\pgfmathresult) {\y};
			}
			\foreach \y in {1,2,3,4,5,31,32,33}{
				\pgfmathparse{Mod(\y+27,33)+1}
				\node[anchor=west] at (5,\pgfmathresult) {\y};
			}
			\node[anchor=south] at (2,33) {26};
			\node[anchor=south] at (3,33) {19};
			\node[anchor=south] at (4,33) {12};
			\foreach \x in {1,2,3,4,5}{
				\draw[fill=myred] (\x,33) circle(2.5pt);
				\draw[fill=myred] (\x,32) circle(2.5pt);
			}
			\foreach \x in {1,2,3,4}{
				\draw[fill=mygreen] (\x,31) circle(2.5pt);
				\draw[fill=mygreen] (\x,30) circle(2.5pt);
				\draw[fill=myblue] (\x,29) circle(2.5pt);
				\draw[fill=myblue] (\x,28) circle(2.5pt);
				\draw[fill=myblue] (\x,27) circle(2.5pt);
				\draw[fill=myred] (\x,26) circle(2.5pt);
			}
			\draw[fill=myyellow] (5,31) circle(2.5pt);
			\draw[fill=myyellow] (5,30) circle(2.5pt);
			\draw[fill=myyellow] (5,29) circle(2.5pt);
			\draw[fill=myred] (5,28) circle(2.5pt);
			\draw[fill=myred] (5,27) circle(2.5pt);
			\draw[fill=mygreen] (5,26) circle(2.5pt);
			
		\end{tikzpicture}
		\caption{Close-up of \(T(1,33,7)\)}\label{SF:closeup1}
	\end{subfigure}
	\caption{Illustration of non-\(\ListAsgn{L}\)-colorable
		\(3\)-list-assignments \(\ListAsgn{L}\) on \(3\)-chromatic graphs \(T(1, s, t)\)
		for \(s \geq 9\) and \(7 \leq t < (s+1)/4\) via \(G = T(1, 33, 7)\).
		Distinct colors denote distinct lists
		among \(\List{L}{0}, \List{L}{1}, \List{L}{2}, \List{L}{3}\).}\label{F:non-3-choosability}
\end{figure}

Next, suppose that \((s + 6) / 4 \leq t < (s - 3)/3\). 
Define \(\ListAsgn{L}\) as follows:
\begin{align*}
	\List{L}{1} &: \Set{ s - kt, s - 1 - kt : k = 0, 1, 2, 3, 4 },\\
	\List{L}{2} &: \Set{ s - 2 - kt, s - 3 - kt : k = 0, 1, 2, 3, 4},\\
	\List{L}{3} &: \Set{ s - 4 - kt, s - 5 - kt : k = 0, 1, 2, 3, 4 }\\
	&\qquad \cup \Set{ s - 6 - kt, \dotsc, 2s - 4t + 1 - kt : k = 0, 1, 2 }\\
	&\qquad \cup \Set{ 2s - 4t - kt, \dotsc, 2s - 4t - 5 - kt : k = 1, 2 }\\
	&\qquad \cup \Set{ 2s - 4t - 6 - kt, \dotsc, s - t + 1 - kt : k = 0, 1, 2 },
\end{align*}
and any remaining vertices are assigned the list \(\List{L}{0}\).

Next, suppose that \((s + 3)/3 < t \leq (s - 7)/2\). 
Define \(\ListAsgn{L}\) as follows:
\begin{align*}
	\List{L}{1} &: \Set{ s - kt, s - 1 - kt : k = 0, 1, 2, 3, 4 },\\
	\List{L}{2} &: \Set{ s - 2 - kt, s - 3 - kt : k = 0, 1, 2, 3, 4},\\
	\List{L}{3} &: \Set{ s - 4 - kt, s - 5 - kt : k = 0, 1, 2, 3, 4 }\\
	&\qquad \cup \Set{ s - 6 - kt, \dotsc, 2s - 3t + 1 - kt : k = 0, 1 }\\
	&\qquad \cup \Set{ 2s - 3t - 2 - kt, \dotsc, s - t + 1 - kt : k = 0, 1 },
\end{align*}
and any remaining vertices are assigned the list \(\List{L}{0}\).

The remaining cases are when \(t = 4, (s + 1)/4, (s - 3)/3, (s + 3)/3\).
These are handled by modifying these list assignments
appropriately, as we show in \ref{S:Appendix}.

The proofs of Theorem~\ref{T:Main} and Corollary~\ref{C:main} are now complete from the above results in
Sections~\ref{S:Main} to~\ref{S:baaki}.

\section{Concluding remarks and further questions}\label{S:Conclusion}

In the remarks following the statement of Theorem~\ref{T:Main} in Section~\ref{S:Introduction},
we noted that Theorem~\ref{T:Main} excludes only a small, finite set of \(5\)-chromatic
graphs, as well as an infinite subset of \(4\)-chromatic graphs, both of
the form \(T(1, s, t)\). We shall elaborate on these details now.

\subsection{The simple graphs \texorpdfstring{\(T(r, s, t)\)}{T(r, s, t)} for \texorpdfstring{\(r < 4\)}{r < 4} or \texorpdfstring{\(s < 3\)}{s < 3}}

We first look at the graphs that are not covered by case~\ref{l1} in Theorem~\ref{T:Main};
these are the graphs \(T(r, s, t)\) with \(r < 4\) or \(s < 3\).
But, in particular, we need only be concerned with the choosability
of the simple graphs among these, because
if \(T(r, s, t)\) is a loopless multigraph, then we may
remove the duplicated edges to get a simple \(d\)-regular graph for \(d \leq 5\),
which is \(5\)-choosable by Brooks's theorem~\cite{Brooks1941,ErdosRubinEtAl1980,Vizing1976},
except when the graph is isomorphic to \(K_{6}\) (but this happens only when
\((r, s) \in \Set{ (1, 6), (2, 3), (3, 2) }\)).
Moreover, these graphs can be \(5\)-list colored in linear time~\cite{Skulrattanakulchai2006}.

Now, one can check that the graphs \(T(r, s, t)\) with \(s < 3\) either contain loops or multiple edges,
so it suffices to assume \(r < 4\) and \(s \geq 3\).

For \(r = 3, s \geq 3\), there are no graphs \(T(3, s, t)\) with loops or multiple edges.

For \(r = 2, s \geq 3\), there are no graphs with loops, and the loopless multigraphs are precisely those
with \(t = 0\), \(s - 2\), \(s - 1\), so we assume that \(1 \leq t \leq s - 3\).
In particular, it suffices to assume that \(s \geq 4\) in this case.

For \(r = 1, s \geq 3\), the graph \(T(1, s, t)\) is isomorphic to \(T(1, s, s - t - 1)\), so
it suffices to consider the values of \(t\) in the range \(0 \leq t \leq \floor{(s - 1) / 2}\).
Then, the graphs \(T(1, s, t)\) with loops are precisely those with \(t = 0\),
and the loopless multigraphs are precisely those with \(t = 1\), \(\floor{(s - 1) / 2}\).
So, when \(r = 1\), we need only consider the graphs \(T(1, s, t)\) with
\(2 \leq t \leq \floor{(s - 1) / 2} - 1\). In particular, it suffices to assume
that \(s \geq 7\) in this case.

\subsubsection{The graphs \texorpdfstring{\(T(3, s, t)\)}{T(3, s, t)} for \texorpdfstring{\(s \geq 3\)}{s >= 3}}\label{SS:T(3st)}

The three normal circuits for \(T(3, s, t)\) have lengths
\(s\), \(3s/\gcd(s, t)\) and \(3s/\gcd(s, t + 3)\).
Again, if either \(3s/\gcd(s, t)\) or \(3s/\gcd(s, t + 3)\)
is at least \(4\), then \(T(3, s, t)\) is isomorphic
to \(T(r', s', t')\) for some \(r' \geq 4\), and
we are done by case~\ref{l1} in Theorem~\ref{T:Main}.
So, assume that both are at most \(3\).
If either one equals \(3\), then so does the other,
and \(T(3, s, t)\) is in fact \(3\)-chromatic in this case, so it is \(5\)-choosable
by case~\ref{l4} in Theorem~\ref{T:Main}.
So, assume that both are at most \(2\).
But, it is not possible that both \(\gcd(s, t)\) and \(\gcd(s, t + 3)\) equal \(2\).
Note that all the above \(5\)-list colorings can be found in linear time as well.
The only case left is when \(T(3, s, t)\) is isomorphic
to \(T(1, 3s, t'')\) and it does not satisfy any of the cases~\ref{l1} to~\ref{l4}.

\subsubsection{The graphs \texorpdfstring{\(T(2, s, t)\)}{T(2, s, t)} for \texorpdfstring{\(s \geq 4\)}{s >= 4}}\label{SS:T(2st)}

The three normal circuits of \(T(2, s, t)\) have lengths
\(s\), \(2s/\gcd(s, t)\) and \(2s/\gcd(s, t + 2)\).
If either \(2s/\gcd(s, t)\) or \(2s/\gcd(s, t + 2)\)
is at least \(4\), then \(T(2, s, t)\) is isomorphic
to \(T(r', s', t')\) for some \(r' \geq 4\), and so it
is \(5\)-choosable by case~\ref{l1} in Theorem~\ref{T:Main}.
So, suppose that both \(2s/\gcd(s, t)\) and \(2s/\gcd(s, t + 2)\)
are at most \(3\). Note that both cannot be equal to \(3\) simultaneously.
If any one equals \(2\), then so does the other, and this case
is covered by case~\ref{l3} in Theorem~\ref{T:Main}. All the above
\(5\)-list colorings can clearly be found in linear time as well.
The only case left is when \(T(2, s, t)\) is isomorphic
to \(T(1, 2s, t'')\) and it does not satisfy any of the cases~\ref{l1} to~\ref{l4}.

\subsubsection{The graphs \texorpdfstring{\(T(1, s, t)\)}{T(1, s, t)} for \texorpdfstring{\(s \geq 7\)}{s >= 7}}\label{SS:T(1st)}
When \(s = 7\), we have to only
consider the case \(t = 2\), and \(T(1, 7, 2)\) is isomorphic to \(K_{7}\),
which is both \(7\)-chromatic and \(7\)-list chromatic. When \(s = 11\),
we have to consider the cases \(t = 2\), \(3\), \(4\), but in each
case the graph is isomorphic to the \(6\)-chromatic triangulation \(J\)
of Albertson and Hutchinson~\cite{AlbertsonHutchinson1980}
mentioned in Section~\ref{S:Introduction}. By Dirac's map color theorem for choosability~\cite{BohmeMoharEtAl1999},
the graph \(J\) is also \(6\)-list chromatic.
So, assume that \(s \geq 8\) and \(s \neq 11\).

As shown by Yeh and Zhu~\cite{YehZhu2003}, other than a small, finite list of exceptions, the simple \(5\)-chromatic
\(6\)-regular toroidal triangulations are those isomorphic to \(T(1, s, 2)\)
for \(s \not\equiv 0 \pmod{4}\), \(s \geq 9\), \(s \neq 11\). Case~\ref{l2} of Theorem~\ref{T:Main} shows that
the graphs \(T(1, s, 2)\) for \(s \geq 9\), \(s \neq 11\) are all \(5\)-choosable
in linear time. Thus, we obtain an infinite class
of \(5\)-chromatic-choosable simple toroidal triangulations,
proving Corollary~\ref{C:main}.

Thus, the only graphs that are
not covered by Theorem~\ref{T:Main} are of the form \(T(1, s, t)\). Moreover,
these consist only of the finitely many \(5\)-chromatic graphs
not of the form \(T(1, s, 2)\), as well as the \(4\)-chromatic graphs not covered
by cases~\ref{l1} to~\ref{l3}. We are presently unable to comment on the choosability of these remaining graphs.

\subsection{Further questions}\label{SS:fq}

As shown by Yeh and Zhu~\cite{YehZhu2003}, there is a small, finite set of \(5\)-chromatic
graphs of the form \(T(1, s, t)\) that are not isomorphic to \(T(1, s, 2)\),
with the largest among them having \(37\) vertices: these have parameters \((s,t)\) in the set \(\{ (13,3) \), \((17,3)\), \((18,3)\), \((19,3)\), \((25,3)\), \((25,9)\), \((26,7)\), \((33,6)\), \((37,10) \}\)\footnote{The list of these parameters in \cite{YehZhu2003} contains pairs \((s,t)\) that give rise to isomorphic graphs, for instance \((13,3)\) as well as \((13,4)\). It is straightforward to obtain the above list of nine values for \((s,t)\) from the full list: for instance, see \cite[Theorem 4]{MeszkaNedelaEtAl2006} (though they mistakenly omit the case \((s,t) = (25,9)\)).}.
Among these, the graph \(T(1,25,9)\) is isomorphic to \(T(5,5,2)\) and so its \(5\)-choosability is covered by case~\ref{l1}.
For the remaining eight graphs, we have the following natural question:

\begin{question}\label{Q:1}
	Is \(\cho(G) = 5\) for the eight \(6\)-regular toroidal triangulations
	that are \(5\)-chromatic and not isomorphic to \(T(1,25,9)\) or \(T(1, s, 2)\) for any \(s\)?
\end{question}

In our earlier paper~\cite{BalachandranSankarnarayanan2021}, we asked whether any
of the \(3\)-chromatic \(6\)-regular toroidal triangulations are \(5\)-list chromatic.
In light of the results in this paper, we consider a similar question for the
\(4\)-chromatic triangulations not covered in Theorem~\ref{T:Main}.

\begin{question}\label{Q:1.5}
	Is \(\cho(G) \in \Set{4, 5}\) for every \(4\)-chromatic \(6\)-regular toroidal triangulation?
\end{question}
Essentially, we ask whether or not there exists any \(4\)-chromatic \(6\)-list chromatic graph \(T(r, s, t)\).

In our earlier paper~\cite{BalachandranSankarnarayanan2021}, we defined the \defining{jump} of a graph \(G\), \(\jump(G)\),
to be \(\cho(G) - \chr(G)\).
There we showed that every loopless \(6\)-regular toroidal triangulation
satisfies \(\jump(G) \leq 2\). The largest jump for any toroidal graph (which we
defined as \(\jump(g)\) for \(g = 0\)) is at least \(2\) since there exist
\(3\)-chromatic planar graphs that are \(5\)-list chromatic~\cite{Mirzakhani1996,VoigtWirth1997}.
However, we do not have any ``legitimate'' example of a nonplanar toroidal graph \(G\)
that satisfies \(\jump(G) = 2\).

\begin{question}\label{Q:3}
	Does there exist a nonplanar toroidal graph \(G\) with \(\jump(G) \geq 2\)
	and such that any planar subgraph \(H\) of \(G\) has \(\jump(H) < 2\)?
\end{question}

We note that such ``legitimate'' examples must exist as the genus \(g\) increases:
we have shown~\cite{BalachandranSankarnarayanan2021} that for connected graphs embeddable on an orientable
surface with genus \(g > 0\), the largest jump among the \(r\)-chromatic graphs
is of the order \(o(\sqrt{g})\) when \(r\) is of the order \(o(\sqrt{g}/\log_{2}(g))\),
so graphs with small chromatic number (in particular, any planar graph) cannot
be the sole examples of graphs attaining a large jump on a surface of large genus \(g > 0\).

\subsection*{Declarations of interest}
None

\subsection*{Acknowledgements}
The work of Brahadeesh Sankarnarayanan was done while at the Indian Institute of Technology Bombay, and was supported by the National Board for Higher Mathematics (NBHM), Department of Atomic Energy (DAE), Govt.\ of India.

\appendix	
	
	\section{Miscellaneous cases in \texorpdfstring{Theorem~\ref{T:Main}}{Theorem 1}}\label{S:Appendix}
	
	We describe \(3\)-list assignments on the \(3\)-chromatic graphs
	not discussed in Section~\ref{S:baaki}, in order to show
	that they are not \(3\)-choosable. Similar arguments
	as described in Section~\ref{SS:not-3-choosable} will show that these graphs
	are not \(3\)-choosable for the given list assignments, so we omit
	the details.
	
	First, consider \(T(1, s, 4)\) for \(s \geq 21\).
	Let \(\ListAsgn{L}\) be the list assignment
	that assigns the lists \(\List{L}{0}, \List{L}{1}, \List{L}{2}, \List{L}{3}\)
	as follows:
	\begin{align*}
		\List{L}{1} &: \Set{ s - kt, s - 1 - kt : k = 0, 1, 2, 3, 4 },\\
		\List{L}{2} &: \Set{ s - 2 - kt : k = 0, 1, 2, 3},\\
		\List{L}{3} &: \Set{ s - 3 - kt : k = 0, 1, 2, 3 },
	\end{align*}
	and any remaining vertices are assigned the list \(\List{L}{0}\).
	Figure~\ref{F:case2} illustrates this list assignment for the graph \(T(1, 21, 4)\).

	Next, consider \(T(1, s, t)\) for \(t = (s + 1)/4\).
	Define \(\ListAsgn{L}\) as follows:
	\begin{align*}
		\List{L}{1} &: \Set{ s - kt, s - 1 - kt : k = 0, 1, 2, 3 },\\
		\List{L}{2} &: \Set{ s - 2 - kt, s - 3 - kt : k = 0, 1, 2, 3},\\
		\List{L}{3} &: \Set{ s - 4 - kt, s - 5 - kt : k = 0, 1, 2, 3 },\\
		&\qquad \cup \Set{ s - 6 - kt, \dotsc, s - t + 1 : k = 0, 1, 2 },
	\end{align*}
	and any remaining vertices are assigned the list \(\List{L}{0}\).
	Figure~\ref{F:case3} illustrates this list assignment for the graph \(T(1, 27, 7)\).

	Next, consider \(T(1, s, t)\) for \(t = (s + 3)/3\), \(s > 27\).
	Define \(\ListAsgn{L}\) as follows:
	\begin{align*}
		\List{L}{1} &: \Set{ s - kt, s - 1 - kt : k = 0, 1, 2, 3, 4 },\\
		\List{L}{2} &: \Set{ s - 2 - kt : k = 1, 2, 3, 4},\\
		&\qquad \cup \Set{ s - 3 - kt : k = 2, 3, 4 }\\
		&\qquad \cup \Set{ s - 4 - kt, \dotsc, s - 6 - kt : k = 2, 3 }\\
		\List{L}{3} &: \Set{1, s - 2} \cup \Set{ s - 7 - kt, \dotsc, s - 9 - kt : k = 1, 2, 3 }\\
		&\qquad \cup \Set{ s - 10 - kt, \dotsc, s - 12 - kt : k = 1, 2 }\\
		&\qquad \cup \Set{ s - 13 - kt, \dotsc, s - t + 1 - kt : k = 0, 1 },
	\end{align*}
	and any remaining vertices are assigned the list \(\List{L}{0}\).
	Figure~\ref{F:case4} illustrates this list assignment for the graph \(T(1, 45, 16)\).
	The only remaining case for \(t = (s + 3)/3\) is \(T(1,27,10)\), which
	is isomorphic to \(T(1,27,4)\), so this case is completed.
	
	Lastly, consider \(T(1, s, t)\) for \(t = (s - 3)/3\). By the remarks
	in Section~\ref{S:Preliminaries}, we can choose the horizontal normal circuit
	of \(T(1, s, t)\) as the vertical normal circuit to get that \(T(1, s, t)\)
	is isomorphic to a graph \(T(1, s, t')\) for some \(0 \leq t' \leq \floor{(s - 1)/2}\).
	A simple calculation shows that either \(t' \equiv -(1 + t^{-1}) \pmod{s}\)
	or \(t' \equiv t^{-1} \pmod{s}\). Using \(s = 3t + 3\) and the fact that \(t - 1 \equiv 0 \equiv s \pmod{3}\),
	we see that \(t' \neq t\). Thus, the graph \(T(1, s, t)\) is isomorphic to
	one of the cases considered earlier, so it is also not \(3\)-choosable.
	
	This covers the \(3\)-chromatic graphs of the form \(T(1, s, t)\).
	The graph \(T(2, 6, 2)\) requires an ad hoc list assignment
	\(\ListAsgn{L}\) as follows:
	\begin{align*}
		\List{L}{1} &: \Set{ (1, 1), (1, 3), (1, 5), (2, 1) }\\
		\List{L}{2} &: \Set{ (1, 2), (1, 4), (1, 6), (2, 6) }\\
		\List{L}{3} &: \Set{ (2, 2), (2, 3), (2, 4), (2, 5) }.
	\end{align*}
	Figure~\ref{F:T(262)} illustrates this list assignment.
	
	The only case left is \(T(3,9,3)\), which requires an ad hoc list assignment \(\ListAsgn{L}\) as follows:
	\begin{align*}
		\List{L}{1} &: \Set{ (1, 1), (1, 2), (1, 7), (1, 8), (2, 1), (2, 2), (3, 1), (3, 2) }\\
		\List{L}{2} &: \Set{ (1, 3), (1, 4), (1, 9), (2, 3), (2, 4), (3, 3), (3, 4) }\\
		\List{L}{3} &: \Set{ (1, 5), (1, 6), (2, 5), (2, 6), (2, 7), (2, 8), (2, 9), (3, 5), (3, 6), (3, 7), (3, 8), (3, 9) }.
	\end{align*}
	Figure~\ref{F:T(393)} illustrates this list assignment.

	\begin{figure}[H]
	\centering
		\begin{tikzpicture}[font=\scriptsize]
			\draw (1,1) grid (4,6);
			\foreach \y in {2,...,6}
			\foreach \x in {1,...,3}
			\draw (\x,\y) -- (\x+1,\y-1);
			\foreach \y in {1,...,6}{
				\pgfmathparse{Mod(\y+3,6)+1}
				\node[anchor=east] at (1,\pgfmathresult) {(2,\y)};
				\pgfmathparse{Mod(\y+1,6)+1}
				\node[anchor=west] at (4,\pgfmathresult) {(1,\y)};
			}
			\foreach \y in {2,...,5}{
				\node[anchor=south west] at (2,\y) {\!\!(1,\y)};
				\node[anchor=south west] at (3,\y) {\!\!(2,\y)};
			}
			\node[anchor=north] at (2,1) {(1,1)};
			\node[anchor=north] at (3,1) {(2,1)};
			\node[anchor=south] at (2,6) {(1,6)};
			\node[anchor=south] at (3,6) {(2,6)};
			
			\draw[fill=myred] (2,1) circle(2.5pt);
			\draw[fill=myred] (2,3) circle(2.5pt);
			\draw[fill=myred] (2,5) circle(2.5pt);
			\draw[fill=myred] (3,1) circle(2.5pt);
			\draw[fill=mygreen] (2,2) circle(2.5pt);
			\draw[fill=mygreen] (2,4) circle(2.5pt);
			\draw[fill=mygreen] (2,6) circle(2.5pt);
			\draw[fill=mygreen] (3,6) circle(2.5pt);
			\draw[fill=myblue] (3,2) circle(2.5pt);
			\draw[fill=myblue] (3,3) circle(2.5pt);
			\draw[fill=myblue] (3,4) circle(2.5pt);
			\draw[fill=myblue] (3,5) circle(2.5pt);
			
			\draw[fill=myred] (4,3) circle(2.5pt);
			\draw[fill=myred] (4,5) circle(2.5pt);
			\draw[fill=myred] (4,1) circle(2.5pt);
			\draw[fill=myred] (1,5) circle(2.5pt);
			\draw[fill=mygreen] (4,4) circle(2.5pt);
			\draw[fill=mygreen] (4,6) circle(2.5pt);
			\draw[fill=mygreen] (4,2) circle(2.5pt);
			\draw[fill=mygreen] (1,4) circle(2.5pt);
			\draw[fill=myblue] (1,6) circle(2.5pt);
			\draw[fill=myblue] (1,1) circle(2.5pt);
			\draw[fill=myblue] (1,2) circle(2.5pt);
			\draw[fill=myblue] (1,3) circle(2.5pt);
		\end{tikzpicture}
		\caption{A non-\(\ListAsgn{L}\)-colorable
			\(3\)-list-assignment on \(T(2, 6, 2)\).}\label{F:T(262)}
	\end{figure}
	
	\begin{figure}[H]
	\centering
		\begin{tikzpicture}[font=\scriptsize]
			\draw (1,1) grid (5,9);
			\foreach \y in {2,...,9}
			\foreach \x in {1,...,4}
			\draw (\x,\y) -- (\x+1,\y-1);
			\foreach \y in {1,...,9}{
				\pgfmathparse{Mod(\y+5,9)+1}
				\node[anchor=east] at (1,\pgfmathresult) {(3,\y)};
				\pgfmathparse{Mod(\y+2,9)+1}
				\node[anchor=west] at (5,\pgfmathresult) {(1,\y)};
			}
			\foreach \y in {2,...,8}{
				\node[anchor=south west] at (2,\y) {\!\!(1,\y)};
				\node[anchor=south west] at (3,\y) {\!\!(2,\y)};
				\node[anchor=south west] at (4,\y) {\!\!(3,\y)};
			}
			\node[anchor=north] at (2,1) {(1,1)};
			\node[anchor=north] at (3,1) {(2,1)};
			\node[anchor=north] at (4,1) {(3,1)};
			\node[anchor=south] at (2,9) {(1,9)};
			\node[anchor=south] at (3,9) {(2,9)};
			\node[anchor=south] at (4,9) {(3,9)};
			
			\draw[fill=mygreen] (1,1) circle(2.5pt);
			\draw[fill=myblue] (1,2) circle(2.5pt);
			\draw[fill=myblue] (1,3) circle(2.5pt);
			\draw[fill=myblue] (1,4) circle(2.5pt);
			\draw[fill=myblue] (1,5) circle(2.5pt);
			\draw[fill=myblue] (1,6) circle(2.5pt);
			\draw[fill=myred] (1,7) circle(2.5pt);
			\draw[fill=myred] (1,8) circle(2.5pt);
			\draw[fill=mygreen] (1,9) circle(2.5pt);

			\draw[fill=myred] (2,1) circle(2.5pt);
			\draw[fill=myred] (2,2) circle(2.5pt);
			\draw[fill=mygreen] (2,3) circle(2.5pt);
			\draw[fill=mygreen] (2,4) circle(2.5pt);
			\draw[fill=myblue] (2,5) circle(2.5pt);
			\draw[fill=myblue] (2,6) circle(2.5pt);
			\draw[fill=myred] (2,7) circle(2.5pt);
			\draw[fill=myred] (2,8) circle(2.5pt);
			\draw[fill=mygreen] (2,9) circle(2.5pt);

			\draw[fill=myred] (3,1) circle(2.5pt);
			\draw[fill=myred] (3,2) circle(2.5pt);
			\draw[fill=mygreen] (3,3) circle(2.5pt);
			\draw[fill=mygreen] (3,4) circle(2.5pt);
			\draw[fill=myblue] (3,5) circle(2.5pt);
			\draw[fill=myblue] (3,6) circle(2.5pt);
			\draw[fill=myblue] (3,7) circle(2.5pt);
			\draw[fill=myblue] (3,8) circle(2.5pt);
			\draw[fill=myblue] (3,9) circle(2.5pt);

			\draw[fill=myred] (4,1) circle(2.5pt);
			\draw[fill=myred] (4,2) circle(2.5pt);
			\draw[fill=mygreen] (4,3) circle(2.5pt);
			\draw[fill=mygreen] (4,4) circle(2.5pt);
			\draw[fill=myblue] (4,5) circle(2.5pt);
			\draw[fill=myblue] (4,6) circle(2.5pt);
			\draw[fill=myblue] (4,7) circle(2.5pt);
			\draw[fill=myblue] (4,8) circle(2.5pt);
			\draw[fill=myblue] (4,9) circle(2.5pt);

			\draw[fill=myred] (5,1) circle(2.5pt);
			\draw[fill=myred] (5,2) circle(2.5pt);
			\draw[fill=mygreen] (5,3) circle(2.5pt);
			\draw[fill=myred] (5,4) circle(2.5pt);
			\draw[fill=myred] (5,5) circle(2.5pt);
			\draw[fill=mygreen] (5,6) circle(2.5pt);
			\draw[fill=mygreen] (5,7) circle(2.5pt);
			\draw[fill=myblue] (5,8) circle(2.5pt);
			\draw[fill=myblue] (5,9) circle(2.5pt);
		\end{tikzpicture}
		\caption{A non-\(\ListAsgn{L}\)-colorable
			\(3\)-list-assignment on \(T(3, 9, 3)\).}\label{F:T(393)}
	\end{figure}

	\begin{figure}[H]
	\centering
		\begin{subfigure}{0.45\textwidth}
			\centering
			\begin{tikzpicture}[font=\scriptsize,scale=0.5]
				\draw (1,1) grid (5,21);
				\foreach \y in {2,...,21}
				\foreach \x in {1,...,4}
				\draw (\x,\y) -- (\x+1,\y-1);
				\foreach \y in {1,...,21}{
					\node[anchor=east] at (1,\y) {\y};
					\pgfmathparse{Mod(\y+15,21)+1}
					\node[anchor=west] at (5,\pgfmathresult) {\y};
				}
				\node[anchor=north] at (1,1) {\vphantom{\scriptsize{(1,1)}}};
				\node[anchor=north] at (2,1) {18};
				\node[anchor=north] at (3,1) {14};
				\node[anchor=north] at (4,1) {10};
				\node[anchor=south] at (2,21) {17};
				\node[anchor=south] at (3,21) {13};
				\node[anchor=south] at (4,21) {9};
				
				\foreach \y in {1,2,3}{
					\draw[fill=myyellow] (1,\y) circle(4pt);
					\pgfmathparse{Mod(\y+3,21)+1}
					\draw[fill=myyellow] (2,\pgfmathresult) circle(4pt);	
					\pgfmathparse{Mod(\y+7,21)+1}
					\draw[fill=myyellow] (3,\pgfmathresult) circle(4pt);	
					\pgfmathparse{Mod(\y+11,21)+1}
					\draw[fill=myyellow] (4,\pgfmathresult) circle(4pt);	
					\pgfmathparse{Mod(\y+15,21)+1}
					\draw[fill=myyellow] (5,\pgfmathresult) circle(4pt);	
				}
				\foreach \y in {21,20,17,16,13,12,9,8,5,4}{
					\draw[fill=myred] (1,\y) circle(4pt);
					\pgfmathparse{Mod(\y+3,21)+1}
					\draw[fill=myred] (2,\pgfmathresult) circle(4pt);	
					\pgfmathparse{Mod(\y+7,21)+1}
					\draw[fill=myred] (3,\pgfmathresult) circle(4pt);	
					\pgfmathparse{Mod(\y+11,21)+1}
					\draw[fill=myred] (4,\pgfmathresult) circle(4pt);	
					\pgfmathparse{Mod(\y+15,21)+1}
					\draw[fill=myred] (5,\pgfmathresult) circle(4pt);	
				}
				\foreach \y in {18,14,10,6}{
					\draw[fill=myblue] (1,\y) circle(4pt);
					\pgfmathparse{Mod(\y+3,21)+1}
					\draw[fill=myblue] (2,\pgfmathresult) circle(4pt);	
					\pgfmathparse{Mod(\y+7,21)+1}
					\draw[fill=myblue] (3,\pgfmathresult) circle(4pt);	
					\pgfmathparse{Mod(\y+11,21)+1}
					\draw[fill=myblue] (4,\pgfmathresult) circle(4pt);	
					\pgfmathparse{Mod(\y+15,21)+1}
					\draw[fill=myblue] (5,\pgfmathresult) circle(4pt);	
				}
				\foreach \y in {19,15,11,7}{
					\draw[fill=mygreen] (1,\y) circle(4pt);
					\pgfmathparse{Mod(\y+3,21)+1}
					\draw[fill=mygreen] (2,\pgfmathresult) circle(4pt);	
					\pgfmathparse{Mod(\y+7,21)+1}
					\draw[fill=mygreen] (3,\pgfmathresult) circle(4pt);	
					\pgfmathparse{Mod(\y+11,21)+1}
					\draw[fill=mygreen] (4,\pgfmathresult) circle(4pt);	
					\pgfmathparse{Mod(\y+15,21)+1}
					\draw[fill=mygreen] (5,\pgfmathresult) circle(4pt);	
				}
			\end{tikzpicture}
			\caption{}
		\end{subfigure}
		\begin{subfigure}{0.45\textwidth}
			\centering
			\begin{tikzpicture}[font=\scriptsize]
				\draw (1,17) grid (5,21);
				\draw[dotted]
				(1,16) -- (1,17) -- (2,16) -- (2,17) -- (3,16) -- (3,17) -- (4,16) -- (4,17) -- (5,16) -- (5,17);
				\node[anchor=north] at (2,16) {\vphantom{(2,16)}};
				\foreach \y in {18,...,21}
				\foreach \x in {1,...,4}
				\draw (\x,\y) -- (\x+1,\y-1);
				\foreach \y in {17,...,21}{
					\node[anchor=east] at (1,\y) {\y};
				}
				\foreach \y in {13,...,16}{
					\pgfmathparse{Mod(\y+3,21)+1}
					\node[anchor=south west] at (2,\pgfmathresult) {\y};
				}
				\foreach \y in {9,...,12}{
					\pgfmathparse{Mod(\y+7,21)+1}
					\node[anchor=south west] at (3,\pgfmathresult) {\y};
				}
				\foreach \y in {5,...,8}{
					\pgfmathparse{Mod(\y+11,21)+1}
					\node[anchor=south west] at (4,\pgfmathresult) {\y};
				}
				\foreach \y in {1,...,5}{
					\pgfmathparse{Mod(\y+15,21)+1}
					\node[anchor=west] at (5,\pgfmathresult) {\y};
				}
				\node[anchor=south] at (2,21) {17};
				\node[anchor=south] at (3,21) {13};
				\node[anchor=south] at (4,21) {9};
				\foreach \x in {1,2,3,4,5}{
					\draw[fill=myred] (\x,21) circle(2.5pt);
					\draw[fill=myred] (\x,20) circle(2.5pt);
				}
				\foreach \x in {1,2,3,4}{
					\draw[fill=mygreen] (\x,19) circle(2.5pt);
					\draw[fill=myblue] (\x,18) circle(2.5pt);
					\draw[fill=myred] (\x,17) circle(2.5pt);
				}
				\foreach \y in {17,18,19}{
					\draw[fill=myyellow] (5,\y) circle(2.5pt);
				}
			\end{tikzpicture}
			\caption{}
		\end{subfigure}
	\caption{A non-\(\ListAsgn{L}\)-colorable \(3\)-list-assignment on \(T(1, s, 4)\) for \(s = 21\).}\label{F:case2}
	\end{figure}

	\begin{figure}[H]
	\centering
		\begin{subfigure}{0.45\textwidth}
		\centering
			\begin{tikzpicture}[font=\scriptsize,scale=0.5]
				\draw (1,1) grid (5,27);
				\foreach \y in {2,...,27}
				\foreach \x in {1,...,4}
				\draw (\x,\y) -- (\x+1,\y-1);
				\foreach \y in {1,...,27}{
					\node[anchor=east] at (1,\y) {\y};
					\pgfmathparse{Mod(\y+27,27)+1}
					\node[anchor=west] at (5,\pgfmathresult) {\y};
				}
				\node[anchor=north] at (1,1) {\vphantom{\scriptsize{(1,1)}}};
				\node[anchor=north] at (2,1) {21};
				\node[anchor=north] at (3,1) {14};
				\node[anchor=north] at (4,1) {7};
				\node[anchor=south] at (2,27) {20};
				\node[anchor=south] at (3,27) {13};
				\node[anchor=south] at (4,27) {6};
				
				\foreach \y in {27,26,20,19,13,12,6,5}{
					\draw[fill=myred] (1,\y) circle(4pt);
					\pgfmathparse{Mod(\y+6,27)+1}
					\draw[fill=myred] (2,\pgfmathresult) circle(4pt);	
					\pgfmathparse{Mod(\y+13,27)+1}
					\draw[fill=myred] (3,\pgfmathresult) circle(4pt);	
					\pgfmathparse{Mod(\y+20,27)+1}
					\draw[fill=myred] (4,\pgfmathresult) circle(4pt);	
					\pgfmathparse{Mod(\y+27,27)+1}
					\draw[fill=myred] (5,\pgfmathresult) circle(4pt);	
				}
				\foreach \y in {23,22,16,15,9,8,2,1,21,14,7}{
					\draw[fill=myblue] (1,\y) circle(4pt);
					\pgfmathparse{Mod(\y+6,27)+1}
					\draw[fill=myblue] (2,\pgfmathresult) circle(4pt);	
					\pgfmathparse{Mod(\y+13,27)+1}
					\draw[fill=myblue] (3,\pgfmathresult) circle(4pt);	
					\pgfmathparse{Mod(\y+20,27)+1}
					\draw[fill=myblue] (4,\pgfmathresult) circle(4pt);	
					\pgfmathparse{Mod(\y+27,27)+1}
					\draw[fill=myblue] (5,\pgfmathresult) circle(4pt);	
				}
				\foreach \y in {25,24,18,17,11,10,4,3}{
					\draw[fill=mygreen] (1,\y) circle(4pt);
					\pgfmathparse{Mod(\y+6,27)+1}
					\draw[fill=mygreen] (2,\pgfmathresult) circle(4pt);	
					\pgfmathparse{Mod(\y+13,27)+1}
					\draw[fill=mygreen] (3,\pgfmathresult) circle(4pt);	
					\pgfmathparse{Mod(\y+20,27)+1}
					\draw[fill=mygreen] (4,\pgfmathresult) circle(4pt);	
					\pgfmathparse{Mod(\y+27,27)+1}
					\draw[fill=mygreen] (5,\pgfmathresult) circle(4pt);	
				}
			\end{tikzpicture}
			\caption{}
		\end{subfigure}
		\begin{subfigure}{0.45\textwidth}
		\centering
			\begin{tikzpicture}[font=\scriptsize]
				\draw (1,20) grid (5,27);
				\draw[dotted]
				(1,19) -- (1,20) -- (2,19) -- (2,20) -- (3,19) -- (3,20) -- (4,19) -- (4,20) -- (5,19) -- (5,20);
				\node[anchor=north] at (2,19) {\vphantom{(2,19)}};
				\foreach \y in {21,...,27}
				\foreach \x in {1,...,4}
				\draw (\x,\y) -- (\x+1,\y-1);
				\foreach \y in {20,...,27}{
					\node[anchor=east] at (1,\y) {\y};
				}
				\foreach \y in {13,...,19}{
					\pgfmathparse{Mod(\y+6,27)+1}
					\node[anchor=south west] at (2,\pgfmathresult) {\y};
				}
				\foreach \y in {6,...,12}{
					\pgfmathparse{Mod(\y+13,27)+1}
					\node[anchor=south west] at (3,\pgfmathresult) {\y};
				}
				\foreach \y in {1,...,5,26,27}{
					\pgfmathparse{Mod(\y+20,27)+1}
					\node[anchor=south west] at (4,\pgfmathresult) {\y};
				}
				\foreach \y in {19,...,26}{
					\pgfmathparse{Mod(\y+27,27)+1}
					\node[anchor=west] at (5,\pgfmathresult) {\y};
				}
				\node[anchor=south] at (2,27) {20};
				\node[anchor=south] at (3,27) {13};
				\node[anchor=south] at (4,27) {6};
				\foreach \x in {1,2,3,4}{
					\draw[fill=myred] (\x,27) circle(2.5pt);
					\draw[fill=myred] (\x,26) circle(2.5pt);
					\draw[fill=mygreen] (\x,25) circle(2.5pt);
					\draw[fill=mygreen] (\x,24) circle(2.5pt);
					\draw[fill=myblue] (\x,23) circle(2.5pt);
					\draw[fill=myblue] (\x,22) circle(2.5pt);
					\draw[fill=myred] (\x,20) circle(2.5pt);
				}
				\foreach \x in {1,2,3}{
					\draw[fill=myblue] (\x,21) circle(2.5pt);
				}
				\draw[fill=myred] (4,21) circle(2.5pt);
				\draw[fill=myred] (5,20) circle(2.5pt);
				\draw[fill=myred] (5,21) circle(2.5pt);
				\draw[fill=myblue] (5,22) circle(2.5pt);
				\draw[fill=myblue] (5,23) circle(2.5pt);
				\draw[fill=myblue] (5,24) circle(2.5pt);
				\draw[fill=mygreen] (5,25) circle(2.5pt);
				\draw[fill=mygreen] (5,26) circle(2.5pt);
				\draw[fill=myred] (5,27) circle(2.5pt);
			\end{tikzpicture}
			\caption{}
		\end{subfigure}
		\caption{A non-\(\ListAsgn{L}\)-colorable \(3\)-list-assignment on \(T(1, s, (s+1)/4)\) for \(s = 27\).}\label{F:case3}
	\end{figure}

	\begin{figure}[H]
	\centering
		\begin{subfigure}{0.45\textwidth}
		\centering
			\begin{tikzpicture}[font=\scriptsize,scale=0.4]
				\draw (1,1) grid (5,45);
				\foreach \y in {2,...,45}
				\foreach \x in {1,...,4}
				\draw (\x,\y) -- (\x+1,\y-1);
				\foreach \y in {1,...,45}{
					\node[anchor=east] at (1,\y) {\y};
					\pgfmathparse{Mod(\y+63,45)+1}
					\node[anchor=west] at (5,\pgfmathresult) {\y};
				}
				\node[anchor=north] at (1,1) {\vphantom{\scriptsize{(1,1)}}};
				\node[anchor=north] at (2,1) {30};
				\node[anchor=north] at (3,1) {14};
				\node[anchor=north] at (4,1) {43};
				\node[anchor=south] at (2,45) {29};
				\node[anchor=south] at (3,45) {13};
				\node[anchor=south] at (4,45) {42};
				
				\foreach \y in {45,44,42,41,29,28,26,25,13,12}{
					\draw[fill=myred] (1,\y) circle(4pt);
					\pgfmathparse{Mod(\y+15,45)+1}
					\draw[fill=myred] (2,\pgfmathresult) circle(4pt);	
					\pgfmathparse{Mod(\y+31,45)+1}
					\draw[fill=myred] (3,\pgfmathresult) circle(4pt);	
					\pgfmathparse{Mod(\y+47,45)+1}
					\draw[fill=myred] (4,\pgfmathresult) circle(4pt);	
					\pgfmathparse{Mod(\y+63,45)+1}
					\draw[fill=myred] (5,\pgfmathresult) circle(4pt);	
				}
				\foreach \y in {43,35,34,33,32,31,30,22,21,20,19,18,17,16,15,14,6,5,4,3,2,1}{
					\draw[fill=myblue] (1,\y) circle(4pt);
					\pgfmathparse{Mod(\y+15,45)+1}
					\draw[fill=myblue] (2,\pgfmathresult) circle(4pt);	
					\pgfmathparse{Mod(\y+31,45)+1}
					\draw[fill=myblue] (3,\pgfmathresult) circle(4pt);	
					\pgfmathparse{Mod(\y+47,45)+1}
					\draw[fill=myblue] (4,\pgfmathresult) circle(4pt);	
					\pgfmathparse{Mod(\y+63,45)+1}
					\draw[fill=myblue] (5,\pgfmathresult) circle(4pt);	
				}
				\foreach \y in {40,39,38,37,36,27,24,23,11,10,9,8,7}{
					\draw[fill=mygreen] (1,\y) circle(4pt);
					\pgfmathparse{Mod(\y+15,45)+1}
					\draw[fill=mygreen] (2,\pgfmathresult) circle(4pt);	
					\pgfmathparse{Mod(\y+31,45)+1}
					\draw[fill=mygreen] (3,\pgfmathresult) circle(4pt);	
					\pgfmathparse{Mod(\y+47,45)+1}
					\draw[fill=mygreen] (4,\pgfmathresult) circle(4pt);	
					\pgfmathparse{Mod(\y+63,45)+1}
					\draw[fill=mygreen] (5,\pgfmathresult) circle(4pt);	
				}
			\end{tikzpicture}
			\caption{}
		\end{subfigure}
		\begin{subfigure}{0.45\textwidth}
		\centering
			\begin{tikzpicture}[font=\scriptsize,scale=0.8]
				\draw (1,29) grid (5,45);
				\draw[dotted]
				(1,28) -- (1,29) -- (2,28) -- (2,29) -- (3,28) -- (3,29) -- (4,28) -- (4,29) -- (5,28) -- (5,29);
				\node[anchor=north] at (2,28) {\vphantom{(2,28)}};
				\foreach \y in {30,...,45}
				\foreach \x in {1,...,4}
				\draw (\x,\y) -- (\x+1,\y-1);
				\foreach \y in {29,...,45}{
					\node[anchor=east] at (1,\y) {\y};
				}
				\foreach \y in {13,...,28}{
					\pgfmathparse{Mod(\y+15,45)+1}
					\node[anchor=south west] at (2,\pgfmathresult) {\y};
				}
				\foreach \y in {1,...,12,42,43,44,45}{
					\pgfmathparse{Mod(\y+31,45)+1}
					\node[anchor=south west] at (3,\pgfmathresult) {\y};
				}
				\foreach \y in {26,...,41}{
					\pgfmathparse{Mod(\y+47,45)+1}
					\node[anchor=south west] at (4,\pgfmathresult) {\y};
				}
				\foreach \y in {10,...,26}{
					\pgfmathparse{Mod(\y+63,45)+1}
					\node[anchor=west] at (5,\pgfmathresult) {\y};
				}
				\node[anchor=south] at (2,45) {29};
				\node[anchor=south] at (3,45) {13};
				\node[anchor=south] at (4,45) {42};
				\foreach \x in {1,2,3,4,5}{
					\draw[fill=myred] (\x,45) circle(2.5pt);
					\draw[fill=myred] (\x,44) circle(2.5pt);
					\draw[fill=myblue] (\x,35) circle(2.5pt);
					\draw[fill=myblue] (\x,34) circle(2.5pt);
					\draw[fill=myblue] (\x,33) circle(2.5pt);
				}
				\foreach \x in {2,3,4,5}{
					\draw[fill=mygreen] (\x,43) circle(2.5pt);
					\draw[fill=myblue] (\x,38) circle(2.5pt);
					\draw[fill=myblue] (\x,37) circle(2.5pt);
					\draw[fill=myblue] (\x,36) circle(2.5pt);
				}
				\foreach \x in {1,2,3,4}{
					\draw[fill=mygreen] (\x,40) circle(2.5pt);
					\draw[fill=mygreen] (\x,39) circle(2.5pt);
					\draw[fill=myred] (\x,29) circle(2.5pt);
				}
				\foreach \x in {1,2}{
					\draw[fill=myred] (\x,42) circle(2.5pt);
					\draw[fill=myred] (\x,41) circle(2.5pt);
					\draw[fill=myblue] (\x,32) circle(2.5pt);
					\draw[fill=myblue] (\x,31) circle(2.5pt);
				}
				\foreach \x in {3,4,5}{
					\draw[fill=mygreen] (\x,42) circle(2.5pt);
					\draw[fill=myred] (\x,32) circle(2.5pt);
					\draw[fill=myred] (\x,31) circle(2.5pt);
				}
				\draw[fill=myblue] (1,43) circle(2.5pt);
				\draw[fill=mygreen] (3,41) circle(2.5pt);
				\draw[fill=mygreen] (4,41) circle(2.5pt);
				\draw[fill=mygreen] (1,38) circle(2.5pt);
				\draw[fill=mygreen] (1,37) circle(2.5pt);
				\draw[fill=mygreen] (1,36) circle(2.5pt);
				\draw[fill=mygreen] (5,29) circle(2.5pt);
				\draw[fill=mygreen] (5,30) circle(2.5pt);
				\draw[fill=mygreen] (4,30) circle(2.5pt);
				\draw[fill=myblue] (5,41) circle(2.5pt);
				\draw[fill=myblue] (5,40) circle(2.5pt);
				\draw[fill=myblue] (5,39) circle(2.5pt);
				\draw[fill=myblue] (1,30) circle(2.5pt);
				\draw[fill=myblue] (2,30) circle(2.5pt);
				\draw[fill=myblue] (3,30) circle(2.5pt);
			\end{tikzpicture}
			\caption{}
		\end{subfigure}
		\caption{A non-\(\ListAsgn{L}\)-colorable \(3\)-list-assignment on \(T(1, s, (s+3)/3)\) for \(s = 45\).}\label{F:case4}
	\end{figure}

\end{document}